\documentclass[12pt]{article}
\usepackage{amsmath,amssymb}
\usepackage[a4paper, left=2.5cm, right=2.5cm, top=2.5cm, bottom=2.5cm]{geometry}\usepackage{graphicx}
\usepackage{graphicx}
\usepackage{hyperref}
\usepackage{mathtools}
\usepackage[thmmarks]{ntheorem}
\usepackage[utf8]{inputenc}
 \usepackage{hyperref}[colorlinks=true,linkcolor=purple]
  \usepackage{stmaryrd}

\newcounter{results}
\counterwithin{results}{section}

\newtheorem{corollary}[results]{Corollary}
\newtheorem{proposition}[results]{Proposition}

\newtheorem{remark}[results]{Remark}
\newtheorem{lemma}[results]{Lemma}
\newtheorem{lemma+definition}[results]{Lemma \& Definition}
\newtheorem{theorem}[results]{Theorem}
\newtheorem{theorem+definition}[results]{Theorem \& Definition}
\theorembodyfont{\upshape}
\newtheorem{definition}[results]{Definition}

\theoremstyle{nonumberplain}
\theoremsymbol{\ensuremath{\square}}
\newtheorem{proof}{Proof}

\newcommand\A{\mathbb{A}}

\newcommand\C{\mathbb{C}}

\newcommand\N{\mathbb{N}}
\newcommand\OO{\mathbb{O}}
\newcommand\PP{\mathbb{P}}

\newcommand\Z{\mathbb{Z}}
\newcommand\ppp{distinguished\ }

\DeclareMathOperator{\Cay}{Cay}

\DeclareMathOperator{\id}{id}
\DeclareMathOperator{\Ima}{Im}

\DeclareMathOperator{\lab}{lab}
\DeclareMathOperator{\len}{len}
\DeclareMathOperator{\val}{val}

\DeclarePairedDelimiter\abs{\lvert}{\rvert}

\usepackage{xcolor}

\definecolor{darkred}{rgb}{0,0,0} 
\definecolor{darkgreen}{rgb}{0,0,0}
\definecolor{darkblue}{rgb}{0,0,0}

\begin{document}

\title{Foldings in relatively hyperbolic groups}

\author{Richard Weidmann\\ \href{mailto:weidmann@math.uni-kiel.de}
{weidmann@math.uni-kiel.de}
\and
Thomas Weller \\ \href{mailto:t_weller@gmx.net}
{t\_weller@gmx.net}}

\date{\today}

\maketitle


\begin{abstract} Carrier graphs of groups representing subgroups of a given relatively hyperbolic groups are introduced and a combination theorem for relatively quasi-convex subgroups is proven. Subsequently a theory of folds for such carrier graphs is introduced and finiteness results for subgroups of locally relatively quasiconvex relatively hyperbolic groups and Kleinian groups are established. 
\end{abstract}

\maketitle

\section*{Introduction}

Generalizing a result of Thurston \cite{Thurston} formulated in the context of surface subgroups of Kleinian groups, Gromov \cite{Gromov1987} established finiteness of conjugacy classes of subgroups of a torsion-free hyperbolic group $G$ isomorphic to a given finitely presented one-ended group $H$, a proof using the Rips machine was given in \cite{Rips1994}. Delzant generalized this result  in \cite{Delzant1995} to hyperbolic groups with torsion and strengthened the result by establishing finiteness of the number of conjugacy classes of images of a given finitely presented group in a hyperbolic group under homomomorphisms that do not factor through one-ended groups. 

Dahmani \cite{Dahmani2006} generalized this result to relatively hyperbolic groups and established finiteness of homomorphic images under the additional assumption that the homomorphism does not have accidental hyperbolics, i.e. does not factor through a group that splits over a subgroup that is mapped to a parabolic subgroup.

Makanin-Razborov diagrams refine the above arguments as they give a complete parametrization of all homomorphisms from a given finitely presented (or even finitely generated) group to a given (toral relatively) hyperbolic group. They rely essentially on the ideas of Makanin \cite{Makanin1982} and Razborov \cite{Razborov1987} and are in their modern form due to Sela \cite{Sela2001} and Kharlampovich-Myasnikov \cite{Kharlampovich1998,Kharlampovich1998a} in the case of free groups, see \cite{Sela2009} and \cite{Weidmann2019} for hyperbolic groups and  \cite{Groves2009,Groves2005} for toral relatively hyperbolic groups, i.e. groups that are hyperbolic relative to finitely generated free Abelian groups.

\smallskip
In the present paper we are interested in the totality of $k$-generated subgroups of a given relatively hyperbolic group, i.e. images of the free group $F_k$. The above mentioned methods fail in this context as all homomorphisms factor through a free product. Thus the question is in some sense complementary to the above results. 

Clearly no finiteness can be expected, not even of isomorphism classes of $k$-generated subgroups. Examples are  the fundamental groups of closed hyperbolic 3-manifolds. Indeed any such manifold  $M$ has a finite sheeted cover $\tilde M$ that fiberes over the circle \cite{Agol2013} and therefore $\pi_1(M)$ has infinitely many finite index subgroups of the same rank which are pairwise non-isomorphic. It turns out that the absence of non-quasiconvex subgroups rules out such phenomena.

\smallskip
In the case of a torsion-free locally quasiconvex hyperbolic group $G$ I.~Kapovich and the first named author \cite{Kapovich2004} generalized ideas from \cite{Weidmann2002} and showed that there are only finitely many conjugacy classes of one-ended $k$-generated subgroups and therefore by Grushko's theorem only finitely many isomorphism classes of $k$-generated subgroups.  The main objective of the present paper is to generalize the results from \cite{Kapovich2004} in the appropriate way to the relatively hyperbolic setting.

\smallskip
That the finiteness of conjugacy classes of one-ended $k$-generated subgroups does not hold in the relatively hyperbolic setting was already observed by Thurston \cite{Thurston}. Thurston observed that the fundamental group of a cusped finite volume hyperbolic 3-manifold can contain infintely many conjugacy classes of subgroups isomorphic to to the fundamental group of a fixed closed surface, hower all but finitely many of those must split over a cyclic subgroup $C$ contained in the peripheral subgroup corresponding to some cusp, this group is then said to have accidental parabolics. Conversely, if one such surface group is found, then infinitely many can be constructed as in (1) below if the simple closed curve is separating or as in  (2) below if it is non-separating. 

\smallskip
Suppose that $G$ is a group that is hyperbolic relative to a collection of subgroups $\{P_1,\ldots ,P_n\}$. The following three settings can (and often do) yield infinitely many conjugacy classes of one-ended isomorphic subgroups of $G$: 

\begin{enumerate}
\item Let $H=H_1*_C H_2$ be a subgroup of $G$ such that $C\le P_i$ for some $i$ and let $p\in P_i\setminus C$ such that $p$ centralizes $C$. Then $\varphi_p:H\to G$ given by $\varphi_p(h_1)=h_1$ for all $h_1\in H_1$ and $\varphi_p(h_2)=ph_2p^{-1}$ for all $h_2\in H_2$ defines a homomorphism. It can often be ensured using a combination theorem of Mart\'inez-Perdrosa \cite{MartinezPedroza2009} that $\varphi_p$ is injective and that there are infinitely many such $p$ such that images of the $\varphi_p$ yield isomorphic and pairwise non-conjugate subgroups isomorphic to $H$. 
\item Let $H=H_1*_C=\langle H,t\mid tct^{-1}=\psi(c)\hbox{ for all }c\in C\rangle$ where $C\subset H_1\subset G$ and $C\subset P_i$. Let $p\in P_i\setminus C$ such that $p$ centralizes $C$. Then $\varphi_p:H\to G$ given by $\varphi_p(h_1)=h_1$ for all $h_1\in H_1$ and $\varphi_p(t)=tp$ defines a homomorphism. As in (1) this construction often yields infintely many isomorphic and pairwise non-conjugate subgroups isomorphic to $H$.
\item $H=H_1*_CQ$ with $Q\subset P_i$. Let $Q'\subset P_i$ such that $C\subset Q'$ and that  there exists an isomorphism $\psi:Q\to Q'$ with $\psi|_C=\hbox{id}|_C$. Then $\varphi_{\psi}:H\to G$ given by $\varphi_{\psi}(h_1)=h_1$ for all $h_1\in H_1$ and $\varphi_{\psi}(q)=\psi(q)$ for all $q\in Q$ defines a homomorphism. Again this often yields infinitely many isomorphic and pairwise non-conjugate subgroups isomorphic to $H$.
\end{enumerate}

The homomorphisms discussed in the above cases are sometimes refered to as bending moves. They were introduced in the context of Makanin-Razborov diagrams for toral relatively hyperbolic groups, see Alibegovic \cite{Alibegovic2007} and Groves \cite{Groves2005}.

\smallskip
The strategy in \cite{Kapovich2004}  is to reduce a given generating set to a generating set consisting of elements of bounded length using ideas  similar to Nielsen reduction or foldings. In the current paper we follows a similar strategy, however the notions that need to be introduced are significantly more complicated.

\smallskip
This is achieved by the following three steps, here $(G,\PP)$ is a torsion-free group $G$ that is hyperbolic relative to a collection of subgroups $\PP=\{P_1,\ldots,P_n\}$.
\begin{enumerate}
\item[(A)]  Define so called $(G,\PP)$-carrier graphs of groups $\mathcal A$ representing subgroups of $G$. These consist of a  graph of groups $\mathbb A$ with some additional structure that defines a map $\nu_{\mathcal A}:\pi_1(\mathbb A)\to G$. This can be thought of as a variation of Stallings graphs \cite{Stallings1983} or of the $\mathbb A$-graphs defined in \cite{Kapovich2005}. We then say that the image of $\nu_{\mathcal A}$ is represented by $\mathcal A$. It is easily verified that any subgroup of $G$ is represented by a $(G,\PP)$-carrier graph of groups.
\item[(B)] Establish sufficient conditions for the map $\nu_{\mathcal A}$ to be injective. This amounts to establishing a general combination theorem for relatively quasiconvex subgroups of a relatively hyperbolic group, see Theorem~\ref{thm:main} and Corollary~\ref{cor:main}. The combination theorem generalizes previous results from Arzhantseva \cite{Arzhantseva2001,Arzhantseva2006} and Martinez-Pedrosa \cite{MartinezPedroza2009}.
\item[(C)] Define folds and related moves that allow the transformation of a given $(G,\PP)$-carrier graph into one that satisfies the conditions of the combination theorem and represents the same subgroup. This is closely related to Stallings folds for graphs (of groups), see \cite{Stallings1983}, \cite{Dunwoody1998}, \cite{Bestvina1991} and \cite{Kapovich2005}. We then observe that under the appropriate conditions all $k$-generated subgroups are represented by $(G,\PP)$-carrier graphs from only finitely many appropriately defined equivalence classes, for details see Theorem~\ref{mainfinitenesstheorem}. The equivalence relation is defined in such a way that isomorphic subgroups that are related by the above bending moves are represented by equivalent $(G,\PP)$-carrier graphs.
\end{enumerate}

Most of the above mentioned notions and results are too technical to state in the introduction. However, having these tools at hand, finiteness results for subgroups of locally relatively  quasiconvex torsion-free relatively hyperbolic groups are easily established. In particular we obtain the following two facts, the first of which applies in particular to limit groups (over free groups):

\begin{theorem}\label{thm:locqua}[Corollary~\ref{cor:locqua}]
Let $G$ be a finitely generated torsion-free locally relatively quasiconvex  toral relatively hyperbolic group and $n\in\mathbb N$. Then there are only finitely many isomorphism classes of $n$-generated subgroups.
\end{theorem}

\begin{theorem}\label{thm:kleinian}[Corollary~\ref{cor:kleinian}] Let $G$ be a finitely generated torsion-free Kleinian group and $n\in\mathbb N$.
\begin{enumerate}
\item If $G$ is of infinite covolume then there are only finitely many isomorphism classes of $n$-generated subgroups.
\item If $G$ is of finite covolume  then there are only finitely many isomorphism classes of $n$-generated subgroups of infinite index.
\end{enumerate}
\end{theorem}

In chapter one we recall basic notions for relatively hyperbolic groups and establish a sufficient condition for a word to be a quasigeodesic. In Chapter 2 we introduce $(G,\PP)$ carrier graphs of groups and study their basic properties before we prove the combination theorem in Section 3. Section 4 is then dedicated to the introduction of folds and the proofs of the finiteness theorems. We conclude the paper with a brief discussion of the case of relatively hyperbolic groups with torsion. 

The first named author would like to thank Ilya Kapovich for numerous discussions on topics closely related to this paper. Many of these discussions are reflected in the present paper. Further thanks goes to Jakob Heikamp, Nir Lazarovich and Zachary Munroe who pointed out various inaccuracies in an earlier version of this paper. In recent joint work with Edgar A. Bering IV and Jack Kohov \cite{HeikampLazarovichMunroe}  they employ the result of the present paper to study ascending chains of subgroups in relatively hyperbolic groups and 3-manifold groups.

\section{Preliminaries}\label{chp:Preliminaries}
In this chapter relatively hyperbolic groups and relatively quasiconvex subgroups are reviewed. Moreover a sufficient condition for a path in the relative Caley graph to be quasi-geodesic is established. 

\subsection{Relatively hyperbolic groups}

Relatively hyperbolic groups were introduced by Gromov \cite{Gromov1987}. Various ultimately equivalent definitions have since appeared in the literature \cite{Farb1998,Bowditch2012,Osin2006,Groves2008}, see  \cite{Hruska2010} for a comparison of the various notions. The discussion in this paper follows Osin \cite{Osin2006}.
\begin{definition}[Osin \cite{Osin2006}, Definition~2.35]\label{Def:rel:hyp}
Let $G$ be a group. A finite set $\PP=\{P_1,\ldots,P_n\}$ of subgroups of $G$ is called a \emph{peripheral structure of $G$}  and the elements of the $P_i$ are called \emph{peripheral}. Any element conjugate to some peripheral element is called \emph{parabolic}.

Let $\langle X,P_1,\ldots,P_n\mid R\rangle$ be a finite relative presentation of $(G,\PP)$. For each $i\in\{1,\ldots,n\}$ let $\tilde{P}_i$ be an isomorphic copy of $P_i$, such that $\tilde{P}_1,\ldots,\tilde{P}_n,X$ are mutually disjoint. Let $\mathcal{P}:=\bigcup (\tilde{P}_i\setminus\{1\})$.  Any word in $X\cup X^{-1}\cup \mathcal P$ can be interpreted as an element of $F:=F(X)\ast P_1\ast\ldots \ast P_n$ and as an element of $G$ in the obvious way.

$G$ is \emph{hyperbolic relative to $\PP$}, or $(G,\PP)$ is relatively hyperbolic, if the relative Dehn function of this presentation is linear, i.e.\ if there exists $C>0$, such that any for word $w$  in $X\cup X^{-1}\cup \mathcal P$ of length at most $l$ representing the trivial element in $G$ the element $\bar w\in F$ represented by $w$ can be written as the product of at most $C\cdot l$ conjugates of elements of $R$ in $F$.

\end{definition}

 Note that $X$ is only required to be a finite relative  generating set of $G$ in Definition~\ref{Def:rel:hyp}. However, in the remainder of this paper,  $X$ is always a generating set, in particular all relatively hyperbolic groups under consideration will be finitely generated. \color{black}

\smallskip
 For any path $s$ the initial vertex of $s$ is denoted by $s_-$ and the terminal vertex by $s_+$. Any vertex of $s$ that is distinct from $s_-$ and $s_+$ is called an inner vertex of $s$. 
 The following terminology was introduced by Osin \cite{Osin2006}:

\begin{definition}
A subpath of a path $p$ in  $\hbox{Cay}(G,X\cup\mathcal{P})$ is called a \emph{$P_i$-component} of $p$, if it is a maximal non-trivial subpath, that is labeled by letters in $\tilde{P}_i$. We call a subpath of  a path $p$ a {\em $\mathcal P$-component} or {\em a peripheral component} if it is a $P_i$-component for some $i$.

Every vertex of $p$ which is not an inner vertex of some $P_i$-component is called a \emph{phase vertex} of $p$.

A path is \emph{locally minimal} if every $P_i$-component has length $1$.

Two $P_i$-components are said to be \emph{connected} if there is an edge labeled by an element of $P_i$ connecting them. A $P_i$-component of a path $p$ is \emph{isolated}, if it is not connected to another $P_i$-component of $p$.

A path $p$ is \emph{without backtracking} if all its components are isolated.
\end{definition}

Osin showed that  the relative Cayley graph of a relatively hyperbolic group is Gromov hyperbolic and that it fulfills the BCP-property introduced by Farb \cite{Farb1998}. Thus Osin's notion of relative hyperbolicity implies Farb's notion of (strong) relative hyperbolicity.

\begin{lemma}[Osin \cite{Osin2006}, Theorem 3.23 and Theorem 3.26]\label{lem:BCP}
Let $(G,\PP)$ be a relatively hyperbolic group and $X$ a finite generating set of $G$. Let $\lambda\geq1$, $C\geq0$. 

Then  $\Cay(G,X\cup\mathcal P)$ is Gromov hyperbolic and  there exists a constant $\varepsilon=\varepsilon(G,\PP,X,\lambda,C)$ (or simply $\varepsilon(\lambda,C)$ if $G$, $\PP$ and $X$ are clear from the context) such that the following hold for any two $(\lambda,C)$-quasigeodesics $p$ and $q$  without backtracking  in $\Cay(G,X\cup\mathcal{P})$ with the same initial and terminal vertex:
\begin{enumerate}
\item The sets of phase vertices of p and q are contained in the closed $\varepsilon$-neighborhoods (with respect to $d_X$) of each other.
\item Suppose that $s$ is a $P_i$–component of $p$ such that $d_X(s_-, s_+)>\varepsilon$. Then there exists a $P_i$–component $t$ of $q$ which is connected to $s$.
\item Suppose that $s$ and $t$ are connected $P_i$–components of $p$ and $q$, respectively.
Then
\[
\max\{d_X(s_-,t_-),d_X(s_+,t_+)\}\leq\varepsilon.
\]
\end{enumerate}
\end{lemma}

The following definition gives two different notions of length for paths in the relative Cayley graph of a relatively hyperbolic group.
\begin{definition}
Let $(G,\PP)$ be a relatively hyperbolic group with finite generating set $X$. For any path $s=(e_1,\ldots,e_k)$ in $\Cay(G,X\cup\mathcal{P})$, $\len(s):=k$ is called the \emph{length} of $s$.

The \emph{$X$-length} of $s$ is defined as
\[
\len_X(s):=\sum_{i=1}^kd_X(\alpha(e_i),\omega(e_i)).
\]

If $s$ is a piecewise geodesic, i.e.\ there is a tuple $\mathfrak{s}=(s_1,\ldots,s_l)$ of geodesics in $\Cay(G,X\cup\mathcal{P})$ with $s=s_1\ldots s_l$, the \emph{piecewise $X$-length} of $s$ is defined as
\[
\len_X^p(\mathfrak{s}):=\sum_{i=1}^l d_X((s_i)_-,(s_i)_+).
\]

If $\mathfrak{s}$ is clear from the context, write $\len_X^p(s)$ for $\len_X^p(\mathfrak{s})$.
\end{definition}

There are also natural notions of lengths for words and group elements. For some set $X$ and some word $w\in X^\ast$ over $X$, $\abs{w}$ will denote the number of letters in $w$.

For some group $G$ with generating set $X$ and some element $g\in G$, $\abs{g}_X$ denotes the word length of $g$ with respect to $X$, i.e.\ the length of some shortest word in  $X\cup X^{-1}$  representing $g$.

In a common abuse of notation, a word $w$ will often be identified with the element which it represents. So, while $\abs{w}$ denotes the length of $w$, $\abs{w}_X$ denotes the word length with respect to $X$ of the element represented by $w$.  The label of a path $s$ in a Cayley graph, denoted by $\lab(s)$ is the group element $s_-^{-1}s_+$, i.e. the element represented by the word read by $s$.

\subsection{Relatively quasiconvex subgroups}

A natural class of subgroups of a relatively hyperbolic group $(G,\PP)$ are the relatively quasiconvex subgroups.

\begin{definition}[Osin \cite{Osin2006}, Def.~4.9]\label{def:rel qc}
Let $(G,\PP)$ be a re\-la\-tive\-ly hyperbolic group and $X$ a finite generating system of $G$. Let $\nu\geq0$. $H\leq G$ is called a \emph{$\nu$-relatively quasiconvex subgroup} of $(G,\PP)$ with respect to $X$, if every vertex of a geodesic $p$ in $\Gamma(G,X\cup\mathcal{P})$ with endpoints in $H$ has at most a $d_X$-distance of $\nu$ from $H$.

If $X$ and $\PP$ are clear from the context, $H$ is simply called a $\nu$-relatively quasiconvex subgroup of $G$. Moreover, $H$ is called a relatively quasiconvex subgroup of $G$, if it is a $\nu$-relatively quasiconvex subgroup for some $\nu\geq0$.
\end{definition}

Osin (\cite{Osin2006}, Prop.~4.10) shows that the definition of a relatively quasiconvex subgroup does not depend on the choice of $X$. If $Y$ is another finite generating set of $G$ and $H$ is a $\nu$-relatively quasiconvex subgroup of $G$ with respect to $X$, then it is a $\nu'$-relatively quasiconvex subgroup of $G$ with respect to $Y$ for some $\nu'\geq0$.

Relatively quasiconvex subgroups inherit the geometric structure of $G$ and are themselves hyperbolic relative to a peripheral structure that is induced by $\PP$.

\begin{theorem+definition}[Hruska \cite{Hruska2010}, Thm.~9.1]\label{thm+def:induced structure}
Let $(G,\PP)$ be re\-la\-tive\-ly hyperbolic and $H$ a relatively quasiconvex subgroup of $G$. 

Then the set
\[
\bar{\OO}=\{H\cap P^g\mid g\in G,P\in\PP,\abs{H\cap P^g}=\infty\}
\]
consists of finitely many $H$-conjugacy classes of subgroups of $H$. For any set of representatives $\OO$ of these conjugacy classes, $(H,\OO)$ is relatively hyperbolic. Any such peripheral structure $\OO$ is called an \emph{induced structure} of $(G,\PP)$ on $H$.

Moreover, the inclusion $(H,d_{Y\cup\mathcal{O}})\to(G,d_{X\cup\mathcal{P}})$ is a quasiisometric embedding for any finite relative generating set $Y$ of $(H,\OO)$.
\end{theorem+definition}


Now let $(G,\PP)$ be a relatively hyperbolic group with symmetric finite generating set $X$, let $Y\subseteq X^*$ be a generating set of some subgroup $H\leq G$ and let $\OO=\{O_1,\ldots,O_m\}$ where each $O_j$ is a subgroup of $H\cap P_{i_j}^{g_j}$ with $1\leq i_j\leq n$ and $g_j\in X^*$. Then there is a canonical way of mapping each word in $(Y\cup\mathcal{O})^*$ to some word in $(X\cup\mathcal{P})^*$ representing the same group element.

\begin{definition}\label{def:canonical embedding}
The \emph{canonical map} $\iota\colon(Y\cup\mathcal{O})^*\to(X\cup\mathcal{P})^*$ is defined as the canonical extension of the map $Y\cup\mathcal{O}\to(X\cup\mathcal{P})^*$ that maps each element of $Y\subseteq(X\cup\mathcal{P})^*$ to itself and that maps every $o\in O_j$ to the unique word $g_j^{-1}pg_j\in(X\cup\mathcal{P})^*$ representing $o$ with $p\in P_{i_j}$.
\end{definition}

A finitely generated subgroup of a hyperbolic group is quasiconvex, if and only if the inclusion of the subgroup is a quasiisometric embedding. Theorem~\&~Definition~\ref{thm+def:induced structure} generalizes the forward implication by stating that for a relatively quasiconvex subgroup $H$ of $G$ with induced structure $\OO$ the inclusion $(H,d_{Y\cup\mathcal{O}})\to(G,d_{X\cup\mathcal{P}})$ is a quasiisometric embedding.

The following theorem and its proof are straightforward generalizations of a theorem by Osin (\cite{Osin2006}, Thm.~4.13) and can be viewed as generalizations of the converse implication in the above statement.

\begin{theorem}\label{thm:qc-embedding->rel qc}
 Let $(G,\PP)$ be a relatively hyperbolic group where $\PP=\{P_1,\ldots,P_n\}$ and let $X$ be a symmetric finite generating set of $G$.  Let $H\leq G$ be a subgroup with finite generating set $Y\subseteq X^*$.

Let $\OO=\{O_1,\ldots,O_m\}$ such that for each $1\leq j\leq m$ there is some $1\leq i_j\leq n$ and $g_j\in X^*$ with $O_j\subseteq H\cap P_{i_j}^{g_j}$.

Suppose the inclusion $(H,d_{Y\cup\mathcal{O}})\to(G,d_{X\cup\mathcal{P}})$ is a $(\lambda,c)$-quasi\-iso\-metric embedding.

Then $H$ is relatively quasiconvex in $G$.

\begin{proof}
Let
\[
\mu:=\max(\{\abs{y}\mid y\in Y\}\cup\{2\abs{g_j}+1\mid 1\leq j\leq m\}).
\]
Let $h\in H$ and $V=z_1\ldots z_l\in(Y\cup\mathcal{O})^*$ be a minimal word representing $h$, i.e.\ $l=\abs{h}_{Y\cup\mathcal{O}}$.

Define $U:=\iota(V)$, and let $U_0$ be a subword of $U$, i.e.
\[
U_0=A\iota(z_r)\ldots\iota(z_{r+s})B
\]
 where $A$ is trivial or a suffix of $\iota(z_{r-1})$ and $B$ is trivial or a prefix of $\iota(z_{r+s+1})$. In particular  $\abs{A},\abs{B}\leq\mu$.

Since every subword of $V$ is geodesic, the following holds:
\begin{align*}
\abs{U_0} &\leq 2\mu+\mu(s+1)\\
&= 2\mu+\mu\abs{z_r\ldots z_{r+s}}_{Y\cup\mathcal{O}}\\
&\leq 2\mu+\mu(\lambda\abs{z_r\ldots z_{r+s}}_{X\cup\mathcal{P}}+\lambda c)\\
&\leq 2\mu+\mu(\lambda(\abs{U_0}_{X\cup\mathcal{P}}+2\mu)+\lambda c)
\end{align*}

Therefore, the path $p$ in $\Gamma(G,X\cup\mathcal{P})$ from $1$ to $h$, which is labeled by $U$, is a $(\mu\lambda,2\mu+2\mu^2\lambda+\mu\lambda c)$-quasigeodesic.

Assume $p$ has connected $P_i$-components. Then there is some subword $p_1wp_2$ of $U$ for some $w\in(X\cup\mathcal{P})^\ast$ and $p_1,p_2\in P_i$, which represents an element $p_3\in P_i$. Replace the subpath of $p$ labeled by $p_1wp_2$ with the single edge connecting its endpoints and labeled by $p_3$. This new path is still a $(\mu\lambda,2\mu+2\mu^2\lambda+\mu\lambda c)$-quasigeodesic and its vertex set is a subset of the vertex set of $p$.

Repeating this process eventually yields a $(\mu\lambda,2\mu+2\mu^2\lambda+\mu\lambda c)$-quasigeodesic $\bar{p}$ without backtracking, whose vertex set is a subset of the vertex set of $p$.

For any geodesic $q$ in $\Gamma(G,X\cup\mathcal{P})$ from $1$ to $h$ and any vertex $v$ of $q$ it follows with $\varepsilon=\varepsilon(\mu\lambda,2\mu+2\mu^2\lambda+\mu\lambda c)$ as in the conclusion of Lemma~\ref{lem:BCP}, that there exists a vertex $u$ of $\bar{p}$ such that:
\[
d_X(u,v)\leq\varepsilon.
\]

Since $d_X(u,H)\leq\mu$ for any vertex $u$ of $p$ and therefore in particular for any vertex of $\bar{p}$, it follows that $d_X(v,H)\leq\varepsilon+\mu$. Hence, $H$ is relatively quasiconvex as in Definition~\ref{def:rel qc} with $\nu:=\varepsilon+\mu$.
\end{proof}
\end{theorem}

\begin{remark}\label{rem:not induced}
The inclusion $(H,d_{Y\cup\mathcal{O}})\to(G,d_{X\cup\mathcal{P}})$ is a quasiisometric embedding by Theorem~\&~Definition~\ref{thm+def:induced structure}, if $\OO$ is some induced structure on the relatively quasiconvex subgroup $H$ of $G$.

It is however not the case that any structure $\OO$ as in Theorem~\ref{thm:qc-embedding->rel qc} for which this inclusion is a quasiisometric embedding has to be an induced structure. In fact, let $G$ be a group with finite generating set $X$ and hyperbolic relative to $\PP=\{P_1,\ldots,P_n\}$, and $\OO:=\{O_1,\ldots,O_n\}$, where each $O_i$ is a finite index subgroup of $P_i$ and at least one $O_i$ is a proper subgroup of $P_i$. Then $G$ is not hyperbolic relative $\OO$, since $\OO$ is not an almost malnormal collection (cf.~Osin~\cite{Osin2006}, Prop.~2.36), but the inclusion $(G,d_{X\cup\mathcal{O}})\to(G,d_{X\cup\mathcal{P}})$ is a quasiisometric embedding.
\end{remark}

The following lemma is a consequence of Theorem~\ref{thm+def:induced structure}. 

\begin{lemma+definition}\label{lem:bounded-generation}
Let $G$ be be a group which is hyperbolic relative $\PP=\{P_1,\ldots,P_k\}$ and $X$ a symmetric finite generating set of $G$. Let $M\geq 0$.

Then there exists $\nu$ such that for every relatively quasiconvex subgroup $H$ of $G$, which is generated by elements of $X$-length at most $M$, the following hold:
\begin{enumerate}
\item $H$ is $\nu$-relatively quasiconvex.

\item There exist $l\in\N$, and $g_i\in G$ and $n_i\in\{1,\ldots,k\}$ for $i\in\{1,\ldots,l\}$ such that:
\begin{enumerate}
\item $\abs{g_i}_X\leq\nu/2$ for all $i\in\{1,\ldots,l\}$, and
\item $\{H\cap P_{n_i}^{g_i}\mid i\in\{1,\ldots,l\}\}$ is an induced structure of $(G,\PP)$ on $H$.
\item Let $s$ be a geodesic in $\Cay(G,X\cup\mathcal{P})$ with $\lab(s)\in H$. Let $e$ be a $P_n$-edge of $s$ with  $\len_X(e)\geq\nu/2$ and $s=s_1es_2$.

Then there exists some $i\in\{1,\ldots,l\}$ with $n=n_i$, and a path $s'=s_1'\bar{s}s_2'$ from $s_-$ to $s_+$ with the following properties:
\begin{enumerate}
\item $s_1'$ and $s_2'$ are geodesics with $\lab(s_1'),\lab(s_2')\in H$, and
\item $\bar{s}=r_1e'r_2$, where $r_1$ and $r_2$ are geodesics with $\lab(r_1)^{-1}=\lab(r_2)=g_i$, and $e'$ is a $P_n$-edge connected to $e$ with
\[
d_X(\alpha(e),\alpha(e')),d_X(\omega(e),\omega(e'))\leq\nu.
\]
\end{enumerate}
In particular, $\lab(\bar{s})=g_i^{-1}\lab(e')g_i\in H\cap P_{n_i}^{g_i}$.
\end{enumerate}
\end{enumerate}
Let $\nu(G,\PP,X,M)$ be defined as the smallest such $\nu$.
\begin{proof}
Since there are only finitely many elements of $G$ of $X$-length at most $M$, it suffices to show that for every relatively quasiconvex subgroup $H\leq G$, which is generated by elements of $X$-length at most $M$, there is some $\nu_H$ fulfilling (a)-(c).

Let $\bar{\nu}_H\geq0$ such that $H$ is $\bar{\nu}_H$-relatively quasiconvex and such that there is some induced structure $\OO:=\{H\cap P_{n_i}^{g_i}\mid i\in\{1,\ldots,l\}\}$ of $(G,\PP)$ on $H$ with  $\abs{g_i}_X\leq\bar{\nu}_H/2$ for all $i\in\{1,\ldots,l\}$. Let $\mu:=\max(M,2\bar{\nu}_H+1)$.

Let $Y$ be a finite generating set of $H$ such that every element of $Y$ has $X$-length at most $M$. By Theorem \& Definition~\ref{thm+def:induced structure}, the inclusion $(H,d_{Y\cup\mathcal{O}})\to(G,d_{X\cup\mathcal{P}})$ is a $(\lambda,c)$-quasiisometric embedding for some $\lambda\geq1$ and $c\geq0$.

Now let $s$ be a geodesic in $\Cay(G,X\cup\mathcal{P})$ with $\lab(s)\in H$. Let $e$ be a $P_n$-edge of $s$ with $\len_X(e)\geq\varepsilon:=\varepsilon(\mu\lambda,2\mu+2\mu^2\lambda+\mu\lambda c)$ and $s=s_1es_2$.

Let $w=w_1\ldots w_m\in(Y\cup\mathcal{O})^\ast$ be a shortest word representing $\lab(s)$. Let $t$ be the path in $\Cay(G,X\cup\mathcal{P})$ starting in $s_-$ and labeled by $\iota(w)$. Since every element of $Y\cup\mathcal{O}$ is mapped by $\iota$ to a word of length at most $\mu$. The minimality of the length of $w$ and the fact that $\mathbb O$ is an induced structure imply that $t$ is without backtracking. It follows analogously to the proof of Theorem~\ref{thm:qc-embedding->rel qc} that $t$ is a $(\mu\lambda,2\mu+2\mu^2\lambda+\mu\lambda c)$-quasigeodesic. Hence, by Lemma~\ref{lem:BCP}, $t$ has a $P_n$-edge $e'$ which is connected to $e$ and whose endpoints have $X$-distance at most $\varepsilon$ from the endpoints of $e$.

Since $t$ is labeled by $\iota(w)$, there must be a subpath $\bar{s}=r_1e'r_2$ of $t$ as required. Moreover, for geodesics $s_1'$ and $s_2'$ from $s_-$ to $\bar{s}_-$ and from $\bar{s}_+$ to $s_+$, respectively, $\lab(s_1'),\lab(s_2')\in H$.

The claim now follows for $\nu_H:=\max(\bar{\nu}_H,2\varepsilon)$.
\end{proof}
\end{lemma+definition}

\begin{remark}\label{remark_g_i} Let the $n_i$ and $g_i$ be as in the conclusion of Lemma~\ref{lem:bounded-generation}. If $n_i=n_j$ then $g_ig_j^{-1}\notin P_{n_i}$ as otherwise
\[
P_{n_j}^{g_j}=P_{n_i}^{g_j}=P_{n_i}^{(g_ig_j^{-1})g_j}=P_{n_i}^{g_i},
\]
contradicting the fact that $\{H\cap P_{n_i}^{g_i}\mid i\in\{1,\ldots,l\}\}$ is an induced structure.
\end{remark}

\begin{corollary}\label{cor_tame} Let $s=s_1es_2$ be as in Lemma~\ref{lem:bounded-generation} 2(c), $i$ as in the conclusion of Lemma~\ref{lem:bounded-generation} such that $\lab(s_2)\in P_{n_i}g_i$. Then there exists some $h\in H\cap P_{n_i}^{g_i}$ such that  $\lab(s)h=\lab(s_1)\hat p g_i$ with $\hat p\in P_{n_i}$ and $|\hat p|_X\le \nu$.
\end{corollary}

\begin{proof} We use the notation of the proof of Lemma~\ref{lem:bounded-generation}. Choose $p\in P_{n_i}$ such that $\lab(s_2)=pg_i$. Note that $p':=\omega(e)\cdot \omega(e')^{-1}\in P_{n_i}$ as $e$ and $e'$ are connected. It follows that $\lab(s_2')=\lab(r_2)^{-1}(\omega(e)\cdot \omega(e')^{-1}\lab(s_2)=g_i^{-1}p'pg_i\in H\cap P_{n_i}^{g_i}$. 

The claim now follows easily  with $h=\lab(s_2')^{-1}\lab(\bar s)^{-1}$ as  $\lab(s)h=\lab(s_1')$.
\end{proof}


\subsection{Elements of Infinite Order in Relatively Hyperbolic Groups}\label{sec:Elements of Infinite Order in Relatively Hyperbolic Groups}
This section records several simple results on the action of elements of infinite order on $\Cay(G,X\cup\mathcal{P})$.


\begin{definition}[cf. Osin \cite{Osin2006}, Def. 4.23]
Let $(G,\PP)$ be a relatively hyperbolic group and $X$ a symmetric finite generating set of $G$. Let $g\in G$. The \emph{translation length $\tau(g)$} of $g$ is defined by:
\[
\tau(g):=\liminf_{n\to\infty}\frac{\abs{g^n}_{X\cup\mathcal{P}}}{n}.
\]
\end{definition}

In the following we call an element of $G$ {\em hyperbolic}  if it acts hyperbolically on the Cayley graph  $\Cay(G,X\cup\mathcal{P})$. In the torsion-free setting this is equivalent to being non-parabolic.

\begin{remark}
The limit inferior in the above definition is in fact a limit. The argument is the same as for the non-relative case which was proven by Gersten and Short (\cite{Gersten1991},~Lemma~6.3).

Note further that $\tau(hgh^{-1})=\tau(g)$ for all $g,h\in G$  and that $g$ is a hyperbolic if and only if $\tau(g)>0$.
\end{remark}

The next lemma follows from the fact, that there is a lower bound on the translation lengths of hyperbolic elements in a relatively hyperbolic group (\cite{Osin2006},~Thm.~4.25). Recall that the axis of an element acting hyperbolically on a hyperbolic space is the union of all geodesics between its two fixed points in the ideal boundary at infinity.

\begin{lemma}\label{lem:axis}
Let $(G,\PP)$ be a relatively hyperbolic group, $X$ a symmetric finite generating set of $G$. Then there exists $D_1=D_1(G,\PP,X)$ such that the following holds:

Let $h\in G$ be hyperbolic of infinite order with axis $Y$ in $\Cay(G,X\cup\mathcal{P})$. Then for every $z\in\Z\setminus\{0\}$:
\[
\abs{z}\tau(h)+2d_{X\cup\mathcal{P}}(1,Y)-D_1\leq\abs{h^z}_{X\cup\mathcal{P}}\leq\abs{z}\tau(h)+2d_{X\cup\mathcal{P}}(1,Y)+D_1
\]
\end{lemma}

\begin{proof}
Since the axis of a non-trivial power of a hyperbolic element $h$ is the same as the axis of $h$ and since $\tau(h^z)=\abs{z}\tau(h)$ for all $z\in\Z$, it suffices to show the claim for $z=1$.

Let $\delta\geq0$ be such that $\Cay(G,X\cup\mathcal{P})$ is $\delta$-hyperbolic. Let $h\in G$ be hyperbolic of infinite order with axis $Y$ in $\Cay(G,X\cup\mathcal{P})$. Let $g\in VY$ be of minimal $X\cup\mathcal{P}$-length, i.e. a projection of $1$ onto $Y$.

Since $g\in VY$, $d_{X\cup\mathcal{P}}(g,hg)\leq\tau(h)+C$ for some constant $C$ independent of $g$ and $h$. It therefore suffices to show that the difference of lengths of the (piecewise) geodesic paths $[1,h]$ and $[1,g]\cup[g,hg]\cup[hg,h]$  is bounded by a constant independent of $h$.

If $\tau(h)\ge 100\delta$, the claim follows from the minimality of $g$ and the $2\delta$-thinness of the quadrilateral $[1,g]\cup[g,hg]\cup[hg,h]\cup[h,1]$.

Thus, it can be assumed that $\tau(h)<100\delta$. 
Let $x\in[1,g]$ with $d_{X\cup\mathcal{P}}(x,g)=d_{X\cup\mathcal{P}}(Y,[1,h])$. Then $d_{X\cup\mathcal{P}}(x,hx)\leq\eta$ and therefore $d_{X\cup\mathcal{P}}(h^{n-1}x,h^nx)\leq\eta$ for all $n\in \mathbb N$ for some constant $\eta=\eta(\delta)$.

Applying the fact that detours in a hyperbolic space are exponentially large (see Lemma~1.8 in \cite{Coornaert1990}) to the path
\[
\gamma=[g,x]\cup[x,hx]\cup[hx,h^2x]\cup\ldots\cup[h^{N-1}x,h^Nx]\cup[h^Nx,h^Ng]
\]
for an appropriate $N\in\N$ yields that $d_{X\cup\mathcal{P}}(Y,[1,h])$ must be bounded by some constant  depending only on $\delta$ and $\tau(h)$, such that this bound is monotonically decreasing in $\tau(h)$.

\begin{figure}[h!] 
\begin{center}
\includegraphics[scale=.7]{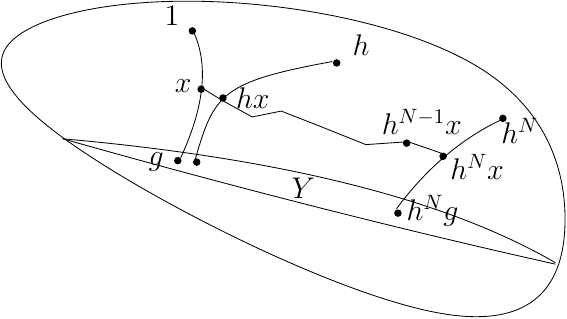}
\caption{The path $\gamma=[g,x]\cup[x,hx]\cup[hx,h^2x]\cup\ldots\cup[h^Nx,h^Ng]$ cannot be far from $Y$.}{\label{fig:hypelement}}
\end{center}
\end{figure}
\color{black}  

By Osin (\cite{Osin2006}, Thm.~4.25), there is some $d=d(G,\PP,X)>0$ such that $\tau(h)\geq d$ for all hyperbolic $h\in H$. Hence, there is an upper bound on $d_{X\cup\mathcal{P}}(Y,[1,h])$ independent of $h$.

Since $\tau(h)$ was assumed to be at most $100\delta$, this implies the claim.
\end{proof}

\begin{lemma}\label{lem:parabolic}
Let $(G,\PP)$ be a relatively hyperbolic group, $X$ a symmetric finite generating set of $G$. Then there exists  $D_2=D_2(G,\PP,X)$ such that the following holds:

Let $h\in gP_ig^{-1}$ be a parabolic element of infinite order. Then:
\[
2d_{X\cup\mathcal{P}}(1,gP_i)-D_2\leq\abs{h}_{X\cup\mathcal{P}}\leq2d_{X\cup\mathcal{P}}(1,gP_i)+1
\]
\begin{proof}
Let $\delta\geq0$ such that the Cayley graph $\Cay(G,X\cup\mathcal{P})$ is $\delta$-hyperbolic. Let $\varepsilon=\varepsilon(2,12\delta)$ be the constant from Lemma~\ref{lem:BCP}. Let $h=gpg^{-1}$, where $p\in P_i$ and $g$ is of minimal $X\cup\mathcal{P}$-length in $gP_i$.

Suppose first that $\abs{p}_X\geq\varepsilon$.
Let $w\in(X\cup\mathcal{P})^\ast$ be a geodesic word representing $g$, $t$ be a path in $\Cay(G,X\cup\mathcal{P})$ labeled by $wpw^{-1}$ and $e$ be the edge of $t$ with label $p$. As $g$ is of minimal $X\cup\mathcal{P}$-length in $gP_i$, $e$ is an isolated $P_i$-component of $t$. 

 We claim that any subpath $s$ of $t$ that is a $(2,12\delta)$-quasigeodesic is in fact a geodesic. Let $s$ be such a subpath of $t$.  If $s$ does not contain $e$ then $s$ is a geodesic as a subpath of a geodesic. Hence, assume $s$ contains $e$. Let $s'$ be a geodesic from $s_-$ to $s_+$. Since $s$ is a $(2,12\delta)$-quasigeodesic and since $\abs{p}_X\geq\varepsilon$, $s'$ must have a $P_i$-component connected to $e$. It follows from the minimality of $g$ in $gP_i$, that $s$ and $s'$ are of the same length. Thus $s$ itself must be a geodesic. The claim is proven.



Trivially, any subpath $s$ of $t$ of length at most $12\delta$ is a $(2,12\delta)$-quasigeodesic, thus $t$ is a $12\delta$-local geodesic. This implies that $t$ is a $(2,2\delta)$-quasigeodesic (\cite{Bridson1999}, CH. III. H 1.13). Therefore, $t$ must itself be a geodesic by the above claim. Thus
\[
\abs{gpg^{-1}}_{X\cup\mathcal{P}}=2\abs{g}_{X\cup\mathcal{P}}+1=2d_{X\cup\mathcal{P}}(1,gP_i)+1.
\]

Now suppose $\abs{p}_X<\varepsilon$. Since $X$ is finite, there are only finitely many such $p$. The statement now follows analogously to the proof of Lemma \ref{lem:axis} by considering a quadrilateral with vertices $1,g,gp,gpg^{-1}$ in $\Cay(G,X\cup\mathcal{P})$ and noting that $d_{X\cup\mathcal{P}}(gP_i,[1,gpg^{-1}])$ is bounded from above by some constant only depending on $G$, $\PP$ and $X$.
\end{proof}
\end{lemma}

The following lemma shows, that the length of the product of an element which is essentially orthogonal to a relatively quasiconvex subgroup with an element of this subgroup is the sum of the lengths of the two elements, up to an additive constant.

\begin{lemma}\label{lem:orthogonal}
Let $(G,\PP)$ be a relatively hyperbolic group, $X$ a finite generating set of $G$. Let $\delta\geq0$ be such that $\Cay(G,X\cup\mathcal{P})$ is $\delta$-hyperbolic and $H$ a $\nu$-relatively quasiconvex subgroup of $G$. Let $g\in G$ be such that $g$ has minimal $X\cup\mathcal{P}$-length in $gH$ and let $h\in H$. Then:
\[
\abs{gh}_{X\cup\mathcal{P}}\geq\abs{g}_{X\cup\mathcal{P}}+\abs{h}_{X\cup\mathcal{P}}-8\delta-2\nu
\]
\end{lemma}

The non-relative version of Lemma~\ref{lem:orthogonal} was proven by Arzhantseva (\cite{Arzhantseva2001},~Lemma~9). The proof uses only the hyperbolicity of the Cayley graph and the fact that a quasiconvex subgroup of a hyperbolic group is a quasiconvex subset of the Cayley graph. Hence, it immediately generalizes to the relatively hyperbolic case. Note that in \cite{Arzhantseva2001} a slightly different but equivalent definition of hyperbolicity is used, which is why the constants also differ slightly.

Combining all the previous lemmas of this section yields the following.
\begin{lemma}\label{lem:conjugation_length}
Let $(G,\PP)$ be a relatively hyperbolic group, $X$ a finite generating set of $G$ and $H$ a $\nu$-relatively quasiconvex subgroup of $G$. Then there exists a $D=D(G,\PP,X,\nu)$ such that the following holds:

Let $h\in H$ be of infinite order and $g\in G$ of minimal $X\cup\mathcal{P}$-length in $gH$. Then
\[
\abs{ghg^{-1}}_{X\cup\mathcal{P}}\geq\abs{h}_{X\cup\mathcal{P}}+2\abs{g}_{X\cup\mathcal{P}}-D
\]
\begin{proof}
Let $D_1=D_1(G,\PP,X)$ and $D_2=D_2(G,\PP,X)$ be the constants from Lemma~\ref{lem:axis} and Lemma \ref{lem:parabolic}. Let $\delta\geq0$ such that $\Cay(G,X\cup\mathcal{P})$ is $\delta$-hyperbolic. Let $h\in H$ be of infinite order and $g$ of minimal $X\cup\mathcal{P}$-length in $gH$.

Suppose $h$ is hyperbolic, let $Y$ be the axis of $h$ in $\Cay(G,X\cup\mathcal{P})$ and $x\in VY$ such that $\abs{gx}_{X\cup\mathcal{P}}=d_{X\cup\mathcal{P}}(1,gx)=d_{X\cup\mathcal{P}}(1,gY)$. As quadrilaterals in $\Cay(G,X\cup\mathcal{P})$ are $2\delta$-thin it follows that $x$ lies in the $2\delta$-neighborhood of a geodesic connecting two elements of $\langle h\rangle\subset H$. As $H$ is $\nu$-relatively quasiconvex, this implies that there is some $k\in H$ with $d_{X\cup\mathcal{P}}(k,x)\leq2\delta+\nu$.

Lemma~\ref{lem:orthogonal} implies:
\begin{align*}
d_{X\cup\mathcal{P}}(1,gY) &= \abs{gx}_{X\cup\mathcal{P}}\\
&\geq \abs{gk}_{X\cup\mathcal{P}}-2\delta-\nu\\
&\geq \abs{g}_{X\cup\mathcal{P}}+\abs{k}_{X\cup\mathcal{P}}-10\delta-3\nu\\
&\geq \abs{g}_{X\cup\mathcal{P}}+\abs{x}_{X\cup\mathcal{P}}-12\delta-4\nu\\
&\geq \abs{g}_{X\cup\mathcal{P}}+d_{X\cup\mathcal{P}}(1,Y)-12\delta-4\nu
\end{align*}
Since $gY$ is the axis of $ghg^{-1}$, it follows by Lemma~\ref{lem:axis}:
\begin{align*}
\abs{ghg^{-1}}_{X\cup\mathcal{P}} &\geq \tau(ghg^{-1})+2d_{X\cup\mathcal{P}}(1,gY)-D_1\\
&\geq \tau(h)+2\abs{g}_{X\cup\mathcal{P}}+2d_{X\cup\mathcal{P}}(1,Y)-24\delta-8\nu-D_1\\
&\geq \abs{h}_{X\cup\mathcal{P}}+2\abs{g}_{X\cup\mathcal{P}}-24\delta-8\nu-2D_1
\end{align*}

Now suppose that $h$ is parabolic, let $h=xpx^{-1}$ where $p\in P_i$ and $\abs{gx}_{X\cup\mathcal{P}}=d_{X\cup\mathcal{P}}(1,gxP_i)$.
Let $w\in X^\ast$ be a word representing $x$ and for $n\in\N$ let $s_n$ be a path in $\Cay(G,X\cup\mathcal{P})$ starting in $1$ and labeled by $wp^nw^{-1}$. Since $h\in H\cap xP_ix^{-1}$ is of infinite order, it follows from Lemma~\ref{lem:BCP} that the $P_i$-component of $s_n$ labeled by $p^n$ is connected to a $P_i$-component of any geodesic from $1$ to $h^n$ for large $n\in\N$. Hence, there must be some $k\in H$ with $d_{X\cup\mathcal{P}}(k,x)\leq\nu+1$.

Lemma~\ref{lem:orthogonal} implies:
\begin{align*}
d_{X\cup\mathcal{P}}(1,gxP_i) &= \abs{gx}_{X\cup\mathcal{P}}\\
&\geq \abs{gk}_{X\cup\mathcal{P}}-\nu-1\\
&\geq \abs{g}_{X\cup\mathcal{P}}+\abs{k}_{X\cup\mathcal{P}}-8\delta-3\nu-1\\
&\geq \abs{g}_{X\cup\mathcal{P}}+\abs{x}_{X\cup\mathcal{P}}-8\delta-4\nu-2\\
&\geq \abs{g}_{X\cup\mathcal{P}}+d_{X\cup\mathcal{P}}(1,xP_i)-8\delta-4\nu-2
\end{align*}
It follows by Lemma~\ref{lem:parabolic}:
\begin{align*}
\abs{ghg^{-1}}_{X\cup\mathcal{P}} &= \abs{(gx)p(gx)^{-1}}_{X\cup\mathcal{P}}\\
&\geq 2d_{X\cup\mathcal{P}}(1,gxP_i)-D_2\\
&\geq 2\abs{g}_{X\cup\mathcal{P}}+2d_{X\cup\mathcal{P}}(1,xP_i)-16\delta-8\nu-4-D_2\\
&\geq \abs{h}_{X\cup\mathcal{P}}+2\abs{g}_{X\cup\mathcal{P}}-16\delta-8\nu-5-D_2
\end{align*}

Hence $D:=\max\{24\delta+8\nu+2D_1,16\delta+8\nu+5+D_2\}$ is as required.
\end{proof}
\end{lemma}

\subsection{Paths with Long Peripheral Edges}\label{sec:Paths with Long Parabolic Edges}
The technical lemma proven in this section is one of the main tools for the proof of the first combination theorem  (Theorem~\ref{thm:main}). A similar result has been given by Mart\'inez-Pedrosa (\cite{MartinezPedroza2009}, Prop.~3.1). It states, that a path $s=s_0p_1s_1\ldots s_{k-1}p_ks_k$ in $\Cay(G,X\cup\mathcal{P})$, where the $s_j$ are geodesics and where the $p_j$ are long peripheral components such that $p_{j-1}$ and $p_j$ are not connected for $2\le j\le k$, is a quasigeodesic.

The main differences between their result and the following lemma is that the $s_j$ are only required to be quasigeodesics. These weaker assumptions necessitate stronger restrictions on the backtracking that is allowed to occur.

\begin{lemma}\label{lem:long parabolic edges}
Let $(G,\PP)$ be a relatively hyperbolic group with symmetric finite generating set $X$. For any $C,K\geq1$, $\varepsilon_1\geq0$ let $\varepsilon:=\varepsilon(G,\PP,X,1,K)$ be the constant from Lemma~\ref{lem:BCP} and let $\varepsilon_2:=2^K(K\varepsilon+\varepsilon_1)$. Then the following holds:

Let $$s=s_0p_1s_1\ldots s_{k-1}p_ks_k$$ be a path of length at most $K$ in $\Cay(G,X\cup\mathcal{P})$, such that the following conditions are satisfied:

\begin{enumerate}
\item The $p_j$ are $\PP$-edges of $s$ with $\len_X(p_j)>\varepsilon_2$,

\item every $s_j$ is a $(C,C)$-quasigeodesic, and

\item if $ps'p'$ is a subpath of $s$ with
\begin{enumerate}
\item $p$ and $p'$ are $P_i$-paths,
\item $\len_X(p),\len_X(p')\geq\varepsilon_1$, and
\item $\len_X(s')\leq\varepsilon_2$
\end{enumerate}
then $\lab(s')\notin P_i$.
\end{enumerate}

Then $s$ is a $(C^2+3C,C^2+3C)$-quasigeodesic. Moreover $p_i$ and $p_j$ are not connected for $i\neq j\in\{1,\ldots, k\}$.

\begin{proof}
Let $s=s_0p_1s_1\ldots s_{k-1}p_ks_k$ be a path of length at most $K$ in $\Cay(G,X\cup\mathcal{P})$ fulfilling 1.-3. Since 1.-3. are inherited by subpaths, it suffices to show that
\[
\abs{\lab(s)}_{X\cup\mathcal{P}}\geq\frac{\len(s)}{C^2+3C}-(C^2+3C).
\]

It can further be assumed w.l.o.g.\ that the $p_j$ are precisely the $\PP$-edges of $s$ with $\len_X(p_j)>\varepsilon_2$.

 The following  lemma is the main step in the proof of Lemma~\ref{lem:long parabolic edges}. It implies in particular that for any $i\in\{1,\ldots ,k\}$ any geodesic $r$ from $s_-$ to $s_+$ contains a peripheral edge connected to $p_i$.

\begin{lemma}\label{lem:isolating_components}
There is some $N\in\N_0$ with $N\leq\len(s)-1\leq K-1$ and paths $$s^i=s_0^ip_1^is_1^i\ldots s_{k-1}^ip_k^is_k^i$$ from $s_-$ to $s_+$ for $i\in\{0,\ldots,N\}$ such that $\len(s^i)\leq\len(s)-i$ for $i\in\{0,\ldots,N\}$, $s^N$ is without backtracking, and for each $s^i$ the following hold:
\begin{enumerate}
\item Every $p_j^i$ is a $\PP$-edge connected to $p_j$ with
\[
\len_X(p_j^i)>(2^K-2^i+1)(K\varepsilon+\varepsilon_1),
\]
\item every $s_j^i$ is a $(C,C)$-quasigeodesic, and
\item if $ps'p'$ is a subpath of $s^i$ with
\begin{enumerate}
\item $p$ and $p'$ are $P_l$-paths,
\item $\len_X(p),\len_X(p')\geq2^i(K\varepsilon+\varepsilon_1)-K\varepsilon$, and
\item $\len_X(s')\leq(2^K-2^{i+1}+2)(K\varepsilon+\varepsilon_1)$
\end{enumerate}
then $\lab(s')\notin P_l$.
\end{enumerate}

\begin{proof}
First let $s^0:=s$. It is clear, that $s^0$ fulfills 1.-3. Now suppose that $s^0,\ldots ,s^i$ are given and satisfy 1.-3. If $s^i$ is without backtracking then the claim follows with $N=i$. 

Thus, assume that $s^i$ has backtracking. Show that there is a path $s^{i+1}$ satisfying 1.-3.\ with $\len(s^{i+1})\leq\len(s)-(i+1)$.

Let $us'u'$ be a subpath of $s^i$, such that $u$ and $u'$ are connected $P_l$-edges and $s'$ is without backtracking and has no $P_l$-edges connected to $u$ (and therefore also none connected to $u'$). 

Let $s^{i+1}$ be the path obtained from $s^i$ by replacing $us'u'$ with a single $P_l$-edge $u_i$ connecting its endpoints. Clearly,
\[
\len(s^{i+1})\leq\len(s^i)-1\leq\len(s)-(i+1).
\]

As $s'$ has length less than $K$ it is a $(1,K)$-quasigeodesic. As $s'$ is moreover without backtracking and none of its $P_l$-edges are connected to $u$ or $u'$, it follows from Lemma~\ref{lem:BCP} that every $\PP$-edge of $s'$ has $X$-length at most $\varepsilon$. This implies that $s'$ contains no $p_j^i$ and that
\[
\len_X(s')\leq\len(s')\varepsilon\leq K\varepsilon\leq(2^K-2^{i+1}+2)(K\varepsilon+\varepsilon_1).
\]

Since $s^i$ satisfies 3., this implies that
\[
\len_X(u)\leq2^i(K\varepsilon+\varepsilon_1)-K\varepsilon
\]
or
\[
\len_X(u')\leq2^i(K\varepsilon+\varepsilon_1)-K\varepsilon.
\]
Hence, there is at most one $j\in\{1,\ldots,k\}$, such that $u=p_j^i$ or $u'=p_j^i$. If neither is the case, then $s^{i+1}$ clearly satisfies 1.\ and 2.\ with $p_j^{i+1}:=p_j^i$ for all $j$.

Hence, assume w.l.o.g\ $u=p_j^i$ and therefore
\[
\len_X(u')\leq2^i(K\varepsilon+\varepsilon_1)-K\varepsilon.
\]
Thus, $p_j^i$ is connected to $u_i$ and
\begin{align*}
\len_X(u_i) &\geq \len_X(u)-\len_X(s')-\len_X(u')=\len_X(p_j^i)-\len_X(s')-\len_X(u')\\
&> (2^K-2^i+1)(K\varepsilon+\varepsilon_1)-K\varepsilon-(2^i(K\varepsilon+\varepsilon_1)-K\varepsilon)\\
&= (2^K-2^{i+1}+1)(K\varepsilon+\varepsilon_1).
\end{align*}
It follows that $s^{i+1}$ satisfies 1.\ and 2. with $q_j^{i+1}:=u_i$ and $q_{j'}^{i+1}:=q_{j'}^i$ for every $j'\neq j$.

Now let $ps''p'$ be a subpath of $s^{i+1}$ fulfilling 3.(a)-(c). If $u_i$ is not an edge of $ps''p'$, it follows from $s^i$ satisfying 3.\ that $\lab(s'')\notin P_l$.

Suppose $u_i$ is an edge of $s''$, i.e.\ $s''=v_1u_iv_2$. Recall that
\[
\len_X(u')\leq2^i(K\varepsilon+\varepsilon_1)-K\varepsilon.
\]
Thus:
\begin{align*}
\len_X(v_1us'u'v_2) &= \len_X(v_1)+\len_X(u)+\len_X(s')+\len_X(u')+\len_X(v_2)\\
&\leq \len_X(v_1)+\len_X(u_i)+2\len_X(s')+2\len_X(u')+\len_X(v_2)\\
&= \len_X(s'')+2(\len_X(s')+\len_X(u'))\\
&\leq (2^K-2^{i+2}+2)(K\varepsilon+\varepsilon_1)+2(K\varepsilon+2^i(K\varepsilon+\varepsilon_1)-K\varepsilon)\\
&= (2^K-2^{i+2}+2)(K\varepsilon+\varepsilon_1)+2^{i+1}(K\varepsilon+\varepsilon_1)\\
&= (2^K-2^{i+1}+2)(K\varepsilon+\varepsilon_1)
\end{align*}
Since $s^i$ satisfies 3.\ it follows that $\lab(s'')=\lab(v_1us'u'v_2)\notin P_l$.

Finally suppose that $p=u_i$ or $p'=u_i$. W.l.o.g.\ assume $p=u_i$ and therefore $\len_X(u_i)\geq2^{i+1}(K\varepsilon+\varepsilon_1)-K\varepsilon$. It follows that
\begin{align*}
\len_X(u)+\len_X(u') &\geq \len_X(u_i)-\len_X(s')=\len_X(p)-\len_X(s')\\
&\geq 2^{i+1}(K\varepsilon+\varepsilon_1)-K\varepsilon-K\varepsilon\\
&= 2(2^i(K\varepsilon+\varepsilon_1)-K\varepsilon).
\end{align*}

Hence, $\len_X(u)\geq2^i(K\varepsilon+\varepsilon_1)-K\varepsilon$ or $\len_X(u')\geq2^i(K\varepsilon+\varepsilon_1)-K\varepsilon$. If $\len_X(u')\geq2^i(K\varepsilon+\varepsilon_1)-K\varepsilon$, let $s''':=s''$, otherwise let $s''':=s'u's''$. In either case, $s'''$ is of length at most $(2^K-2^{i+1}+2)(K\varepsilon+\varepsilon_1)$. Since $s^i$ satisfies 3., $\lab(s''')\notin P_l$, i.e.\ $p=u_i$ and $p'$ are not connected. Hence, $\lab(s'')\notin P_l$.
\end{proof}
\end{lemma}

We continue with the proof of Lemma~\ref{lem:long parabolic edges} with $s^N$ as in the conclusion of Lemma~\ref{lem:isolating_components}. Let $r$ be a geodesic from $s_-$ to $s_+$. Since $s^N$ is a $(1,K)$-quasi\-geo\-desic without backtracking and
\[
\len_X(p_j^N)>(2^K-2^N+1)(K\varepsilon+\varepsilon_1)\geq\varepsilon,
\]
for every $j$ it follows from Lemma~\ref{lem:BCP}, that every $p_j^N$ and therefore every $p_j$ must be connected to a $P_{i_j}$-edge $q_j$ of $r$. Since every $p_j^N$ is an isolated $\PP$-edge of $s^N$, the $q_j$ must be pairwise distinct. It is an easy consequence of Lemma~\ref{lem:BCP} that $q_1,\ldots,q_k$ appear in order in $r$, i.e.\ that $r$ is of the form $r=r_0q_1r_1\ldots q_kr_k$. This implies that for each $j\in\{0,\ldots,k\}$
\[
d_{X\cup\mathcal{P}}(\alpha(r_j),\alpha(s_j)),d_{X\cup\mathcal{P}}(\omega(r_j),\omega(s_j))\leq1.
\]

and therefore $$\len(r_j)\ge \len(s_j)-1\ge \abs{\lab(s_j)}_{X\cup\mathcal{P}}-1$$ for $j\in\{0,k\}$ and $$\len(r_j)\ge \len(s_j)-2\ge \abs{\lab(s_j)}_{X\cup\mathcal{P}}-2$$ for if $1\le j\le k-1$.

If $k\geq\frac{\len(s)}{C^2+3C}-(C^2+3C)$, then
\[
\abs{\lab(s)}_{X\cup\mathcal{P}}=\len(r)\geq k\geq\frac{\len(s)}{C^2+3C}-(C^2+3C).
\]

If on the other hand $k\leq\frac{\len(s)}{C^2+3C}-(C^2+3C)$, it follows that

\begin{align*}
\abs{\lab(s)}_{X\cup\mathcal{P}} &= \len(r)=\sum_{j=0}^k\len(r_j)+k\\
&\geq \sum_{j=0}^k\abs{\lab(s_j)}_{X\cup\mathcal{P}}-k\\
&\geq \sum_{j=0}^k\left(\frac{1}{C}\len(s_j)-C\right)-k\\
&= \frac{1}{C}\left(k+\sum_{j=0}^k\len(s_j)\right)-\frac{k}{C}-(k+1)C-k\\
&= \frac{1}{C}\len(s)-\frac{k}{C}-(k+1)C-k\\
&= \frac{1}{C}\len(s)-\left(C+k\left(C+1+\frac{1}{C}\right)\right)\\
&\geq \frac{1}{C}\len(s)-\left(C+\left(\frac{\len(s)}{C^2+3C}-(C^2+3C)\right)(C+2)\right)\\
&= \left(\frac{1}{C}-\frac{1}{C^2+3C}(C+2)\right)\len(s)-\left(C-(C^2+3C)(C+2)\right)\\
&\geq \frac{\len(s)}{C^2+3C}-(C^2+3C),
\end{align*}
which proves that $s$ is a $(C^2+3C,C^2+3C)$-quasigeodesic. That the $p_k$ are not connected  is an immediate consequence of the fact that the $s^N$ is without backtracking and that $p_k$ is connected to $p_k^N$ for all $k$.
\end{proof}
\end{lemma}

\section{Carrier Graphs of groups} 


It is often useful to represent subgroups of a given group $G$ by graphs labeled by generators of $G$, in the case of subgroups of free groups these are just Stallings graphs \cite{Stallings1983}. Representing subgroups by labeled graphs has further proven fruitful in studying free subgroups of small cancellation groups and more generally hyperbolic groups, see e.g. \cite{Arzhantseva1996,Arzhantseva2006}. 

For subgroups of a fundamental group $\pi_1(M)$ of a closed hyperbolic manifold $M$ a related notion has been introduced by White \cite{White2002}. The so called carrier graphs of $M$ are graphs $X$ together with a $\pi_1$-surjective map $X\to M$. White used them to establish a relationship between the injectivity radius of $M$ and the rank of $\pi_1(M)$. Carrier graphs have been employed successfully by  Biringer and Souto  \cite{Biringer2019} and \cite{Souto2008} to study the ranks of the fundamental groups of mapping tori of closed orientable surfaces, see also \cite{Biringer2024} and \cite{Biringer2017} for results related to the results of this paper.

The $(G,\PP)$-carrier graph of groups introduced in this chapter are geared towards studying subgroups of relatively hyperbolic groups.

\subsection{Carrier Graphs of groups}\label{sec:Carrier Graphs}

For the basic definitions regarding graphs of groups we refer to \cite{Serre1980} or \cite{Bass1993}, we follow the notation of \cite{Kapovich2005}. Thus, the graph underlying a graph of groups $\mathbb A$ is $A$ and the boundary monomorphisms from an edge group $A_e$ into the vertex groups $A_{\alpha(e)}$ and $A_{\omega(e)}$ of the intial vertex $\alpha(e)$ and terminal vertex $\omega(e)$ of $e$ are denoted by $\alpha_e$ and $\omega_e$. An element of $\pi_1(\mathbb A,v_0)$ is represented by equivalence classes of closed $\mathbb A$-paths based at $v_0$, i.e.\ equivalence classes of tuples of type $(a_0,e_1,a_1,\ldots ,e_k,a_k)$ with $(e_1,\ldots ,e_k)$ a closed path in $A$ based at $v_0$ and $a_0,a_k\in A_{v_0}$ and $a_i\in A_\omega (e_i)$ for $1\le i\le k-1$. A subgraph of groups of a graph of groups $\mathbb A$ is a graph of groups $\mathbb C$ such that $C$ is a subgraph of $A$, that $C_x$ is a subgroup of $A_x$ for all $x\in EC\cup VC$ and that the boundary monomorphism of $\mathbb C$ are the restrictions of the boundary monomorphisms of $\mathbb A$. $\mathbb C$ is called full if $C_x=A_x$ for all $x\in EC\cup VC$.

The following definition of a carrier graph of groups is an algebraic version of the concept of a carrier graph. 

\begin{definition}\label{def:carrier_graph}
Let $G$ be a group. A \emph{$G$-carrier graph of groups} is a triple
\[
\mathcal{A}=(\A,(g_e)_{e\in EA},E_0,v_0)
\]
such that the following hold:
\begin{enumerate}
\item $\A$ is a graph of groups such that $A_x\leq G$ for all $x\in EA\cup VA$.
\item $g_e\in G$ for any $e\in EA$ such that $g_{e^{-1}}=g_e^{-1}$ for all $e\in EA$.
\item $E_0$ is an orientation of $A$, i.e. $EA=E_0\sqcup E_0^{-1}$.
\item For any $e\in E_0$ the following hold:
\begin{enumerate}
\item $A_e$ is a subgroup of $A_{\alpha(e)}$ and  $\alpha_e\colon A_e\to A_{\alpha(e)}$ is the inclusion map.
\item The map $\omega_e\colon A_e\to A_{\omega(e)}$ is given by $\omega_e(a)=g_e^{-1}ag_e$ for all $a\in A_e$.
\end{enumerate}
\item $v_0\in VA$ (the base vertex).
\end{enumerate}

If $\C$ is a subgraph of groups of $\A$ then the \emph{$G$-subcarrier graph of groups} $\mathcal C$ of $\mathcal{A}$ is defined as the restriction of $\mathcal A$ to $\mathbb C$. $\mathcal C$ is called a \emph{full subcarrier graph of groups} if $\mathbb C$ is a full subgraph of groups.
\end{definition}

The additional information encoded by the $g_e$ yields a canonical homomorphism of the fundamental group of $\A$ to $G$

\begin{lemma+definition}\label{lem+def:nu_A}
Let $G$ be a group and $\mathcal A=(\A,(g_e)_{e\in EA},E_0,v_0)$ be a $G$-carrier graph of groups. For any $\A$-path $t=(a_0,e_1,a_1,\ldots,e_k,a_k)$ let
\[
\nu_\mathcal{A}(t):=a_0g_{e_1}a_1\cdot\ldots\cdot g_{e_k}a_k.
\]
This induces a well-defined group homomorphism
\[
\nu_{\mathcal A}\colon\pi_1(\mathbb A,v_0)\to G,\,[t]\mapsto\nu_\mathcal{A}([t]):=\nu_\mathcal{A}(t).
\]
$\mathcal A$ is then called a $G$-carrier graph of groups of $H:=\Ima\nu_{\mathcal A}$.

If $G$ is clear from the context then $\mathcal A$ will simply be called a carrier graph of groups or even just a carrier graph of $H$.
\end{lemma+definition}

Note that different choices of $v_0$ yield conjugate subgroups of $G$. When working with a carrier graph $\mathcal{A}$ it is sometimes helpful to consider a subdivision of $\mathcal{A}$. 

\begin{definition}
Let $G$ be a group and $\mathcal{A}=(\A,(g_e)_{e\in EA},E_0,v_0)$ a $G$-carrier graph of groups.

Let $e\in E_0$. A \emph{subdivision of $\mathcal{A}$ along $e$} is a $G$-carrier graph of groups $\mathcal{A}'=(\A',(g_e)_{e\in EA'},E'_0,v_0')$ with the following properties:
\begin{enumerate}
\item $VA'=VA\sqcup\{v\}$,
\item $EA'=EA\sqcup\{e'^{\pm1},e''^{\pm1}\}\setminus\{e^{\pm1}\}$, $E'_0=E_0\sqcup\{e',e''\}\setminus\{e\}$,
\item $\alpha(e')=\alpha(e)$, $\omega(e')=\alpha(e'')=v$, $\omega(e'')=\omega(e)$,

\item $g_{e'},g_{e''}\in G$ with $g_e=g_{e'}g_{e''}$ and $\abs{g_e}_{X\cup\mathcal{P}}=\abs{g_{e'}}_{X\cup\mathcal{P}}+\abs{g_{e''}}_{X\cup\mathcal{P}}$,

\item $A'_{e'}=A_e$, $A_v=A_{e''}=g_{e'}^{-1}A_eg_{e'}$, and

\item $\alpha_{e'}(a')=\alpha_e(a')$, $\omega_{e'}(a')=g_{e'}^{-1}a'g_{e'}$, $\alpha_{e''}(a'')=a''$, $\omega_{e''}(a'')=g_{e''}^{-1}a''g_{e''}$ for all $a'\in A'_{e'}, a''\in A'_{e''}$.
\item $v_0'=v_0$.
\end{enumerate}
\end{definition}

It is obvious that such a subdivision exists for any decomposition of $g_e$ as described in 4. Moreover there exists a canonical isomorphism $\phi\colon\pi_1(\A,v_0)\to\pi_1(\A',v_0')$ such that $\nu_\mathcal{A}=\nu_\mathcal{A'}\circ\phi$.

%
%

\begin{definition} Let $\A$ be a finite graph of groups. The free factors $\A_1,\ldots,\A_k$ of $\A$ are the components of the (not necessarily connected) graph of groups obtained from $\A$ by deleting all edges with trivial edge groups and vertices with trivial vertex groups. 

The \emph{relative rank of $\A$} is the Betti number of the graph obtained from $A$ by collapsing all $A_i$ to single vertices. 

The core of a non-elliptic graph of groups $\A$ is the smallest subgraph of groups of $\A$ for which the inclusion is $\pi_1$-bijective.
\end{definition}

\subsection{$(G,\mathbb P)$-carrier graphs of groups}\label{Sec:GP}
A tree $Y$ of diameter at most $2$ is called a {\em star}. A vertex $v\in VY$ of a star $Y$ is called central if any vertex of $Y$ has a distance of at most one from $v$. Note that the central vertex of a star is unique, unless $Y$ consists of a single edge.


%
\begin{definition} \label{Def:GPcarrier}
Let $(G,\PP)$ be a torsion-free relatively hyperbolic group with $\PP=\{P_1,\ldots,P_n\}$ and let $\mathcal{A}=(\A,(g_e)_{e\in EA},E_0)$ be a $G$-carrier graph of groups. Let $C\subset A$ be a star with central vertex $c$. Let $\mathbb C$ be a full subgraph of groups (of $\mathbb A$)  with underlying graph $C$ and $\mathcal C$ the corresponding $G$-subcarrier graph of groups.

$\mathcal C$ is called a $P_i$-carrier star of groups of $\mathcal A$ with central vertex $c$ if the following hold:
\begin{enumerate}
\item $C_x\le P_i$ for all $x\in EC\cup VC$. 
\item  $g_e\in P_i$ for all $e\in EC$.
\item $\alpha(e)=c$ for all $e\in E_0\cap EC$.
\item $g_eg_{e'}^{-1}\notin A_c$ for all $e\neq e'\in E_0\cap EC$.
\item All boundary monomorphisms of $\mathbb C$ are isomorphisms.
\end{enumerate}
$\mathcal C$ is called non-trivial if $\pi_1(\mathbb C)\neq 1$ and trivial otherwise.
\end{definition}

In the case of subgroups of torsion-free relatively hyperbolic groups we are mostly interested in the following restricted class of carrier graphs of groups.

\begin{definition}\label{def_GPcarriergraph}Let $G$ be a torsion-free group that is hyperbolic relative to a peripheral structure $\PP=\{P_1,\ldots,P_n\}$. A \emph{$(G,\PP)$-carrier graph of groups} is a pair
\[
(\mathcal{A},((\mathcal{C}_i,c_i))_{1\leq i\leq k})
\]
such that the following is satisfied:
\begin{enumerate}
\item  $\mathcal{A}=(\A,(g_e)_{e\in EA},E_0,v_0)$ is a finite $G$-carrier graph of groups.

\item  For each $i\in\{1,\ldots ,k\}$ there exists $m_i\in\{1,\ldots ,n\}$ such that $\mathcal{C}_i$ is a $P_{m_i}$-carrier star of groups with central vertex $c_i$. We call the vertices and edges of the $\mathcal{C}_i$ \emph{peripheral}. 


\item $C_i\cap C_j=\emptyset $ for $i\neq j$.

\item Any edge $e\in EA$  incident to $c_i$ lies in $C_i$.
\item If $A_e\neq1$ then $e$ is peripheral or $e$ is incident to a unique peripheral vertex.

\item $\alpha(e)$ is peripheral for any $e\in E_0$ with $A_e\neq1$. 
\item 
$\mathbb A$ is minimal relative to base point and relative to the non-trivial central peripheral vertices, i.e. for any valence $1$ vertex $v\in VA$ and $e\in EA$ with $\alpha(e)=v$ the boundary morphism $\alpha_e$ is not surjective unless $v$ is a non-trivial  central peripheral vertex or the base vertex.
\end{enumerate}
Non-peripheral edges of $\A$ with non-trivial edge group, and non-peripheral vertices of $\A$ with non-trivial vertex group are called \emph{essential}. 
Vertices  and edges that are neither peripheral nor essential are called \emph{free}.
\end{definition}

$(G,\PP)$-carrier graphs  will be illustrated in the following way:
\begin{enumerate}
\item Solid disks stand for peripheral vertices, and solid line segments for peripheral edges. Both are fat if their edge/vertex groups are non-trivial.
\item Red circles stand for essential vertices and red fat dashed line segments for essential edges.
\item Dotted lines stand for free edges.
\end{enumerate}

The orientation of the peripheral and essential edges is by the definition of $(G,\PP)$-carrier graphs as illustrated in Figure~\ref{fig:GP1}.

\begin{figure}[h!] 
\begin{center}
\includegraphics[scale=1]{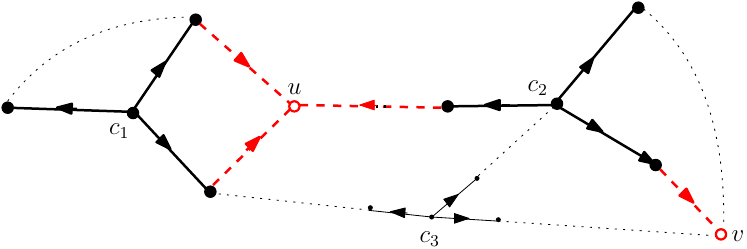}
\caption{A $(G,\PP)$-carrier graph with non-trivial stars $C_1$ and $C_2$, trivial star $C_3$, and essential vertices $u$ and $v$.}{\label{fig:GP1}}
\end{center}
\end{figure}

The following two definitions give some more useful terminology in the context of $(G,\PP)$-carrier graphs of groups.

\begin{definition}
Let $\A$ be a graph of groups and let
\[
t=(a_0,e_1,a_1,\ldots,a_{k-1},e_k,a_k)
\]
be an $\A$-path. For any $i\leq j\in\{0,\ldots,k\}$ and $\bar{a}_l\in\{1,a_l\}$ for $l\in\{i,j\}$, the $\A$-path $s=(\bar{a}_i,e_{i+1},a_{i+1},\ldots,a_{j-1},e_j,\bar{a}_j)$ is called an \emph{$\A$-subpath} of $t$. $s$ is called a full $\mathbb A$-subpath if $\bar a_l=a_l$ for $l\in\{i,f\}$.
\end{definition}

\begin{definition}
Let $(\mathcal{A},((\mathcal{C}_i,c_i))_{1\leq i\leq k})$ be a $(G,\PP)$-carrier graph of groups. Let $p=(e_1,\ldots,e_k)$ be a path in $A$ and let $v_0:=\alpha(e_1)$, $v_m:=\omega(e_m)$ for all $m\in\{1,\ldots,k\}$.


\begin{enumerate}
\item $p$ is called \emph{almost-simple}, if the following hold:
\begin{enumerate}
\item $e_i\neq e_j^{\pm1}$ for all $i\neq j\in\{1,\ldots,k\}$.
\item $v_i=v_j$ for $i\neq j\in\{0,\ldots,k\}$ implies that $v_i$ is a central peripheral vertex.
\end{enumerate}

\item $p$ is called an \emph{almost-circuit}, if $p$ is closed and the following hold:
\begin{enumerate} \item $e_i\neq e_j^{\pm1}$ for all $i\neq j\in\{1,\ldots,k\}$.
\item $v_i=v_j$ for $i\neq j\in\{0,\ldots,k\}$ implies that $\{i,j\}=\{0,m\}$ or that $v_i$ is a central peripheral vertex.
\end{enumerate}
\end{enumerate}

Let $t$ be an $\A$-path. $t$ is called \emph{almost-simple} or an \emph{$\A$-almost-circuit}, if its underlying path is almost-simple or an almost-circuit, respectively.

For any $\A$-path $t$, an $\A$-subpath $s$ of $t$ is called a \emph{full $\A$-subpath} of $t$, if any maximal $\C_i$-subpath  of $t$  is either contained in $s$ or has trivial overlap with $s$.  
\end{definition}
 
\begin{definition}\label{def:nu-normal}
Let $(G,\PP)$ be a torsion-free relatively hyperbolic group with finite symmetric generating set $X$ and $(\mathcal{A},((\mathcal{C}_i,c_i))_{1\leq i\leq k})$ a $(G,\PP)$-carrier graph of groups. 

Let $v\in VA$ be an essential vertex and let
\[
E_v:=\{e\in EA\mid\omega(e)=v,A_e\neq1\}
\]
and $\OO_v:=(\omega_e(A_e))_{e\in E_v}$. Let $M\geq0$ and $\nu:=\nu(G,\PP,X,M)$.

The vertex $v$ is called $M$-tame if the following hold:
\begin{enumerate}
\item $A_v$ is generated by elements of $X$-length at most $M$.
\item $A_v$ is relatively quasiconvex in $(G,\PP)$ (and therefore $\nu$-re\-la\-tive\-ly quasiconvex by Lemma~\&~Definition~\ref{lem:bounded-generation}).
\item The elements $g_e$ for essential edges $e\in EA$ with $\omega(e)=v$ constitute a tuple as in Lemma~\ref{lem:bounded-generation}~2.\ and $\OO_v$ is the corresponding induced structure on $A_v$.
\item $\abs{pg_ea}_X\ge \abs{g_e}_X$ for any essential edge $e\in EA$ with $\alpha(e)\in VC_i$, $p\in P_{m_i}$ and $a\in A_v$.
\end{enumerate}

$(\mathcal{A},((\mathcal{C}_i,c_i))_{1\leq i\leq k})$ is called \emph{$M$-prenormal} if any essential vertex $v\in VA$ is $M$-tame.


\end{definition}


 Let now $\mathcal A=(\mathcal{A},((\mathcal{C}_i,c_i))_{1\leq i\leq k})$ be a $(G,\PP)$-carrier graph of groups (we will continue to use this shorthand in the context of $(G,\PP)$-carrier graphs of groups)  such that all groups of essential vertices are quasiconvex and generated by elements of $X$-length at most $M$. Then $\mathcal A$ can be transformed into an $M$-prenormal carrier graph of groups by modifying the carrier graph  in the way discussed below. In eacht step of this construction the following is assumed:
 \begin{enumerate} 
 \item If a vertex group $A_v$ of some non-trivial peripheral star $\mathcal C$ is replaced by a strictly larger group $A'_v$ then all other vertex and edge groups of that peripheral star are also replaced by the appropriate conjugate of $A'_v$ such that again all boundary monomorphisms are isomorphisms. 
 \item The orientation $E_0$ induces a natural orientation of the new $\mathbb A$-graph with possibly emerging edges being oriented towards essential vertices.
 \item The image of the base vertex is the base vertex of the new $\mathbb A$-graph.
 \end{enumerate}

\smallskip 
The (pre)normalization process is applied to each essential vertex. It follows from results of the second author \cite{Carstensen2022}, that this process can be carried out effectively in many relevant situations. As the vertex group of the essential vertex is unchanged in the process the assumption on $\mathcal A$ guarantees that 1. and 2. of Definition~\ref{def:nu-normal} will always be satisfied. 

The main steps will ensure that condition 3. of Definition~\ref{def:nu-normal} is satisfied. Subsequently some minor postprocessing will be necessary. If 3.  of Definition~\ref{def:nu-normal} is not satisfied then at some essential vertex $v$ one of the following holds:
\begin{enumerate}
\item $|g_e|_X>\nu/2$ for some essential edge $e$ incident to $v$.
\item For some essential edge $e$ with $\omega(e)=v$ the subgroup $\omega_e(A_e)$ of $A_v$ is not part of some induced structure on $A_v$ and therefore a proper subgroup of some parabolic subgroup of $A_v$.
\item For essential edges $e_1\neq e_2$ with $\omega(e_1)=\omega(e_2)=v$ the subgroups $\omega_{e_1}(A_{e_1})$ and $\omega_{e_2}(A_{e_2})$ are conjugate in $A_v$.
\item For some maximal parabolic subgroup $U$ of $A_v$ there exists no essential edge $e$ such that $\omega_e(A_e)$ is in $A_v$ conjugate to a subgroup of $U$.
\end{enumerate}

We now explain how to deal with each of these issues, we deal with the first one last as this last step will also guarantee that 4. of  Definition~\ref{def:nu-normal} holds.

\medskip
\noindent (2) In this case $w:=\alpha(e)\in VC_i$ for some non-trivial star $C_i$, $A_e=\alpha_e(A_e)\in P_{m_i}$ and $\omega_e(A_e)=g_e^{-1}\alpha_e(A_e)g_e\le g_e^{-1}P_{m_i}g_e$. By assumption $g_e^{-1}\alpha_e(A_e)g_e$ is a proper subgroup of some parabolic subgroup of $A_v$, i.e. of some subgroup $A_v\cap g_e^{-1}P_{m_i}g_e$. Replace $A_e$ by $A_e':=g_e(A_v\cap g_e^{-1}P_{m_i}g_e)g_e^{-1}=g_eA_vg_e^{-1}\cap P_{m_i}$, $A_{w}$ by $\langle A_{w},A_e'\rangle$ (and adjust the peripheral star $C_i$ as discussed above). Note that on the graph of group level this corresponds to a fold of type IIA in the sense of \cite{Bestvina1991}.

\medskip
\noindent (3) In this case $w_1:=\alpha(e_1)\in VC_i$, $w_2:=\alpha(e_2)\in VC_j$ for some non-trivial stars $C_i$ and $C_j$ with $m_i=m_j$. By assumption $g_{e_1}^{-1}A_{e_1}g_{e_1}=\omega_{e_1}(A_{e_1})=g\omega_{e_2}(A_{e_2})g^{-1}=gg_{e_2}^{-1}A_{e_2}g_{e_2}g^{-1}$ for some $g\in A_v$. Thus $A_{e_1}=g_{e_1}gg_{e_2}^{-1}A_{e_2}g_{e_2}g^{-1}g_{e_1}^{-1}$ and therefore $g_{e_1}gg_{e_2}^{-1}\in P_{m_i}$ by the malnormality of $P_{m_i}$. We distinguish the cases that $C_i=C_j$ and that $C_i\neq C_j$.

\smallskip
Suppose that $C_i=C_j$. If $w_1\neq w_2$ then there exists a path $(f_1,f_2)$ in $C_i$ from $w_2$ to $w_1$ and we set $g'=g_{f_1}\cdot g_{f_k}\in P_{m_i}$. If $w_1=w_2$ we set $g'=1$. In both cases $\bar g:= g_{e_1}gg_{e_2}^{-1}g'\in P_{m_i}$ is an element represented by a closed $\mathbb A$-path at $w_1$. We remove the edge $e_2$ and replace $A_{w_1}$ by $\langle A_{w_1},\bar g\rangle$ (and adjust the peripheral star $C_i$ as discussed above), see Figure~\ref{fig:GPfold1}. Subsequently we remove $w_2$ and its only adjacent edge if $w_2$ is of valence $1$.

\begin{figure}[h!] 
\begin{center}
\includegraphics[scale=1]{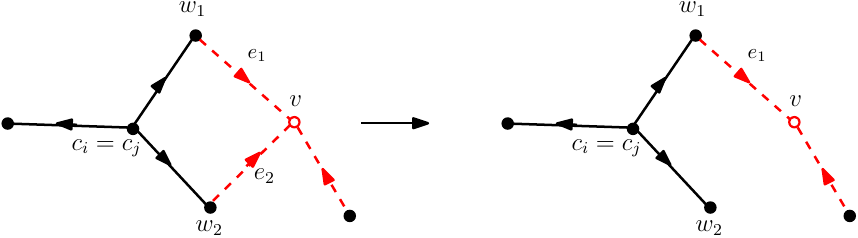}
\caption{A move that removes a redundant essential edge}{\label{fig:GPfold1}}
\end{center}
\end{figure}

\smallskip
Suppose that $C_i\neq C_j$ and therefore also $w_1\neq w_2$. Let $g_1\in P_{m_i}$ be the edge element of the unique edge from $c_i$ to $w_1$ and $g_2$ the edge element of the unique edge from $c_j$ to $w_2$. Then $\bar g:=g_1g_{e_1}gg_{e_2} ^{-1}g_2^{-1}\in P_{m_i}$. 

Combine $C_i$ and $C_j$ as follows:   Delete $e_2$ and $c_j$. Replace $A_{c_i}$ by $\langle A_{c_i},\bar gA_{c_j}\bar g^{-1}\rangle$, and every edge $e\in EC_j$ with $\alpha(e)=c_j$ by an edge $e'$ with $\alpha(e')=c_i$, $\omega(e')=\omega(e)$, $A_{e'}:=\bar gA_e\bar g^{-1}$ and $g_{e'}:=\bar gg_e$, see Figure~\ref{fig:GPfold2}.  Adjust the peripheral star $C_i$ as discussed above and remove $w_2$ and its only adjacent edge if $w_2$ is of valence $1$.
\begin{figure}[h!] 
\begin{center}
\includegraphics[scale=1]{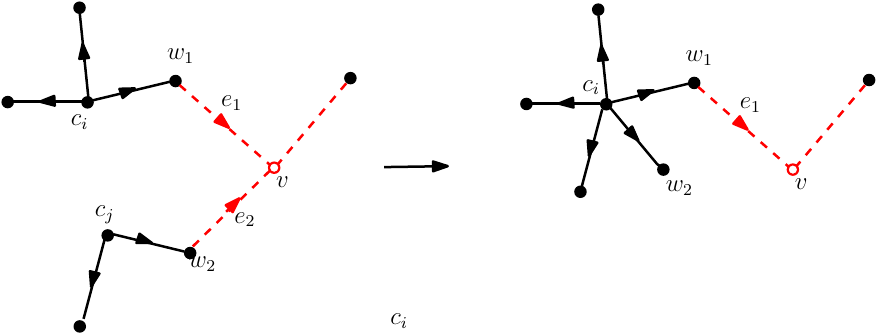}
\caption{Another move that removes a redundant essential edge}{\label{fig:GPfold2}}
\end{center}
\end{figure}

\medskip
\noindent (4) Write $U$ in the form $gQg^{-1}$ with $g\in G$ und $Q\le P_i$ for a suitable $i$. Create a new peripheral star $C$ consisting of a single edge $e'$ with $g_{e'}=1$, a central vertex $c$ and a vertex $w$ such that $A_w=A_c=Q$. Connect $w$ to $v$ by an edge $e$ with edge group $A_e=Q$ and $g_e=g$.

\begin{figure}[h!] 
\begin{center}
\includegraphics[scale=1]{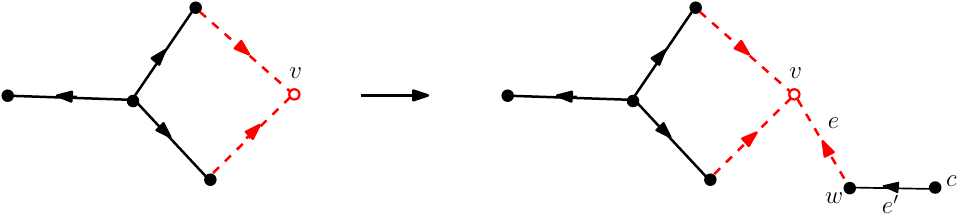}
\caption{A move that introduces an new peripheral star}{\label{fig:GPfold3}}
\end{center}
\end{figure}

\medskip
\noindent (1) In this step the elements of essential edges are modified to guarantee they are of minimal possible length, i.e. that 4. of Definition~\ref{def:nu-normal} holds. By Lemma~\ref{lem:bounded-generation} this also implies that (1) holds.

Suppose that $\abs{pg_ea}_X< \abs{g_e}_X$ where $e\in EA$ is some essential edge with $w=\alpha(e)\in VC_i$, $p\in P_{m_i}$ and $a\in A_v$.  Let $g$ be the element represented by the unique edge in $C_i$ from $c_i$ to $w$. Remove the edge $e$ and introduce a new vertex $w'$ and new edges $f'$ and $e'$ with $\alpha(f')=c_i$, $\omega(f')=\alpha(e')$, $\omega (e')=v$, $g_{f'}=gp^{-1}$ and $g_{e'}=pg_ea$, see Figure~\ref{fig:GPfold4}.

Subsequently we remove $w$ and its only adjacent edge if $w$ is of  valence $1$.

This process yields a $(G,\PP)$ carrier graph of groups that is $M$-prenormal except that possibly 4. of Definition~\ref{Def:GPcarrier} is no longer satisfied. This however is easily rectified by folding two edges of a peripheral star if they violate condition 4. 
We call this latest step the cleaning up of peripheral stars.

\begin{figure}[h!] 
\begin{center}
\includegraphics[scale=1]{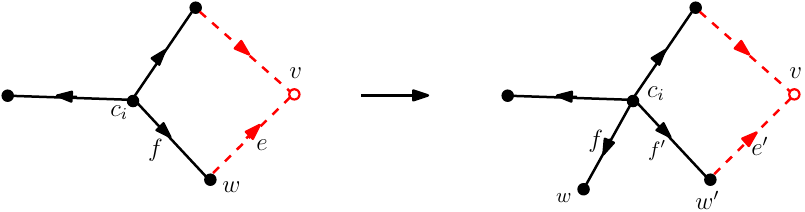}
\caption{A move that shortens the element of an essential edge}{\label{fig:GPfold4}}
\end{center}
\end{figure}

In a final step one guarantees that 7. of Definition~\ref{def_GPcarriergraph} is guaranteed by passing to the appropriate sub-$\mathbb A$-graph. There are some choices in this process and an inspection of all cases shows that the process can performed such that the vertex group of the base vertex is not conjugated. In this case, by choosing the base vertex to be image of the original base vertex under the above operations, the resulting $(G,\PP)$-carrier graph represents the same subgroup.

\medskip
The following lemma follows from the above discussion of the prenormalization process.

\begin{lemma}\label{resultprenormalization} Let $\mathcal A$ be a $(G,\PP)$-carrier graph of groups such that all vertex groups of essential vertices are generated by elements of length at most $M$ and quasiconvex.

The prenormalization process transforms  $\mathcal A$ into an $M$-prenormal $(G,\PP)$-carrier graph of groups $\mathcal A'$ such that following hold:
\begin{enumerate}
\item $\mathcal A$ and $\mathcal A'$ are carrier graphs of the same subgroup of $G$.
\item The number of free factors and the relative rank do not increase.
\item Let $\mathcal E$ and $\mathcal E'$ be the essential edges of $\mathcal A$ and $\mathcal A'$ that lie in the core of $\mathbb A$ and $\mathbb A'$ respectively. Then there exists an injective map $i\colon\mathcal E'\to\mathcal E$ such that for any $e\in \mathcal E'$ the following hold:
\begin{enumerate}
\item $A'_{\omega(e)}=A_{\omega(i(e))}$.
\item $\omega_e(A'_{e})$ is an overgroup of $\omega_{i(e)}(A_{i(e)})$.
\end{enumerate}
\end{enumerate} 
\end{lemma}

\begin{proof}1. and 2. are immediate consequences of the prenormalization process. 3. follows from the fact that any existing essential edge group is either preserved of replaced by an overgroup, the later is happening when dealing with situation 2. or 3. in the above process. Whenever a new essential edge is introduced, then it does not lie in the core.
\end{proof}

\subsection{Metrics on fundamental groups of $(G,\PP)$-carrier graphs}

The canonical homomorphism $\nu_\mathcal{A}$ from Lemma~\&~Definition~\ref{lem+def:nu_A} maps every $\A$-path $t$ in a $(G,\PP)$-carrier graph of groups $\mathcal{A}$ to some element $\nu_\mathcal{A}(t)\in G$. It is however possible to retain more information on $t$ by considering associated paths in $\Cay(G,X\cup\mathcal{P})$. 

\begin{definition}\label{def:realization}
Let $(G,\PP)$ be a torsion-free relatively hyperbolic group with symmetric finite generating set $X$ and $(\mathcal{A},((\mathcal{C}_i,c_i))_{1\leq i\leq k})$ a $(G,\PP)$-carrier graph of groups.

Let $t=(a_0,e_1,a_1,\ldots,e_k,a_k)$ be an $\A$-path, $v_0:=\alpha(e_1)$ and $v_m:=\omega(e_m)$ for $m\in\{1,\ldots,k\}$. Let $s=s_0\ldots s_{2k}$ be a path in $\Cay(G,X\cup\mathcal{P})$ such that the following hold:
\begin{enumerate}
\item $s_{2j}$ is a geodesic whose label represents $a_j$ for $0\leq j\leq k$.
\item $s_{2j-1}$ is a geodesic whose label represents $g_{e_j}$ for $1\leq j\leq k$.
\end{enumerate}

Obtain a new path $\bar{s}$ from $s$ as follows: For every maximal $\C_i$-subpath of $t$ replace the corresponding subpath $s'$ of $s$ by the edge connecting its endpoints and labeled by the element of $P_{m_i}$ that is represented by the label of $s'$.


$\bar{s}$ is called a \emph{realization of $t$ in $\Cay(G,X\cup\mathcal{P})$} or just a \emph{realization of $t$}.
\end{definition}

\smallskip
The following lemma shows how a product decomposition of a realization of some $\A$-path $t$ yields a corresponding subdivision of $t$ in some subdivision of $\mathcal{A}$.

\begin{lemma}\label{lem:subdivision}
Let $(G,\PP)$ be a torsion-free relatively hyperbolic group with symmetric finite generating set $X$ and $(\mathcal{A},((\mathcal{C}_i,c_i))_{1\leq i\leq k})$ an $M$-prenormal $(G,\PP)$-carrier graph of groups. Let $\nu:=\nu(G,\PP,X,M)$.

Let $t$ be an $\A$-path and let $s=s_1s_2$ be a realization of $t$.

Then there is a subdivision $\mathcal{A'}$ of $\mathcal{A}$ along at most one free edge and an $\A'$-path $t'$ starting in $t_-$, which has a realization $s'$ with $s'_-=s_-$ and $d_X(s'_+,(s_1)_+)\leq\nu$.
\begin{proof}
If $s_1$ is itself the realization of some $\A$-subpath $t'$ of $t$, there is nothing to show. Hence, one of the following two cases occurs:
\begin{enumerate}
\item There is some subpath $r_1r_2r_3$ of $s$ such that $s_1=r_1r_2$ and $r_2r_3$ is a geodesic representing some element $a\in A_v$ for an essential vertex $v$ of $t$. Then there must be some $g\in G$ with $d_X(g,(s_1)_+)\leq\nu$ and $g^{-1}(r_2)_-\in A_v$, since $A_v$ is $\nu$-relatively quasiconvex. Let $r_2'$ be a geodesic from $(r_2)_-$ to $g$ and $s':=r_1r_2'$. Since $r_1$ is a realization of some $\A$-subpath of $t$ starting in $t_-$ and ending in $v$, and $r_2'$ is a geodesic representing an element of $A_v$, it is clear, that $s'$ is a realization of some $\A$-path $t'$ starting in $t_-$.

\item There is some subpath $r_1r_2r_3$ of $s$ such that $s_1=r_1r_2$ and $r_2r_3$ is a geodesic representing some element $g_e$ for a non-peripheral edge $e$ of $t$.

If $e$ is free, let $\mathcal{A'}$ be the subdivision of $\mathcal{A}$ along $e$, such that for the added edge $e'$ with $\alpha(e')=\alpha(e)$, $g_{e'}=\lab(r_2)$. Then $s':=s_1$ is a realization of some $\A'$-path $t'$ starting in $t_-$ and ending in $\omega(e')$.

If on the other hand $e$ is essential, let $s':=r_1$. $s'$ is a realization of some $\A$-subpath of $t$ and
\[
d_X(s'_+,(s_1)_+)=d_X((r_2)_-,(r_2)_+)\leq\abs{g_e}_X\leq\nu.
\]
\end{enumerate}
\end{proof} 
\end{lemma}

As a consequence of the construction of the realization, any subpath $s$ of a realization of some $\A$-path has a canonical decomposition into geodesics, the piecewise $X$-length $\len_X^p(s)$ will always be taken with regard to this canonical decomposition.

It is immediate that the lengths $\len(s)$ and $\len_X^p(s)$ of a realization $s$ of some $\A$-path $t$ only depend on $t$. This observation gives rise to two metrics on $\pi_1(\A,v_0)$:

\begin{lemma+definition}\label{lem+def:induced_metrics}
Let $(G,\PP)$ be a torsion-free relatively hyperbolic group with symmetric finite generating set $X$ and $(\mathcal{A},((\mathcal{C}_i,c_i))_{1\leq i\leq k})$ a $(G,\PP)$-carrier graph of groups.

\begin{enumerate}
\item For each $\A$-path $t$ and some (any) realization $s$ of $t$ define the \emph{$X\cup\mathcal{P}$-length of $t$} as $\abs{t}_{X\cup\mathcal{P}}^\mathcal{A}:=\len(s)$. For each equivalence class $g$ of $\A$-paths define
\[
\abs{g}_{X\cup\mathcal{P}}^\mathcal{A}:=\min\left\{\abs{t}_{X\cup\mathcal{P}}^\mathcal{A}\mid[t]=g\right\}.
\]

Then 
\[
d_{X\cup\mathcal{P}}^\mathcal{A}\colon\pi_1(\A,v_0)^2\to\N,\,(g,h)\mapsto\abs{g^{-1}h}_{X\cup\mathcal{P}}^\mathcal{A}
\]
defines a metric and is called the metric on $\pi_1(\A,v_0)$\emph{ induced by $X\cup\mathcal{P}$}.

\item For each $\A$-path $t$ and some (any) realization $s$ of $t$ define the \emph{$X$-length of $t$} as $\abs{t}_X^\mathcal{A}:=\len_X^p(s)$. For each equivalency class $g$ of $\A$-paths define
\[
\abs{g}_X^\mathcal{A}:=\min\left\{\abs{t}_X^\mathcal{A}\mid [t]=g\right\}.
\]

Then 
\[
d_X^\mathcal{A}\colon\pi_1(\A,v_0)^2\to\N,\,(g,h)\mapsto\abs{g^{-1}h}_X^\mathcal{A}
\]
defines a metric and is the metric on $\pi_1(\A,v_0)$\emph{ induced by $X$}.
\end{enumerate}

An $\A$-path $t$ is called \emph{of minimal $X\cup\mathcal{P}$-length}, if $\abs{t}_{X\cup\mathcal{P}}^\mathcal{A}=\abs{[t]}_{X\cup\mathcal{P}}^\mathcal{A}$. Any element $g\in\pi_1(\A,v_0)$ contains a reduced $\A$-path of minimal $X\cup\mathcal{P}$-length.

If $t$ is an $\A$-path and $s$ a full $\A$-subpath of $t$, then $\abs{s}_{X\cup\mathcal{P}}^\mathcal{A}\leq\abs{t}_{X\cup\mathcal{P}}^\mathcal{A}$ and $\abs{s}_X^\mathcal{A}\leq\abs{t}_X^\mathcal{A}$.
\end{lemma+definition}

\begin{remark}
For the last claim of Lemma~\&~Definition~\ref{lem+def:induced_metrics} it is essential, that $s$ is a full $\A$-subpath of $t$. 

Indeed, for an $\mathbb A$-path $t$ consisting of a single $\mathbb C_i$-path $t=(1,e_1,1,e_2,1)$ with $\mathbb C_i$-subpath $s=(1,e_1,1)$  it is in general not possible to give any upper bound on $\abs{s}_X^\mathcal{A}=|g_{e_1}|_X$ in terms of $\abs{t}_X^\mathcal{A}=|g_{e_1}g_{e_2}|_X$.
\end{remark}

In the remainder of this section it is shown that any word metric on $\pi_1(\mathbb A)$ of an $M$-prenormal $(G,\PP)$-carrier graph $\mathcal{A}$, which arises from a finite generating set and the set of peripheral vertex groups, is quasiisometric to the induced metric  $d_{X\cup\mathcal P}^{\mathcal A}$.

\begin{definition}
Let $(G,\PP)$ be a torsion-free relatively hyperbolic group and $(\mathcal{A},((\mathcal{C}_i,c_i))_{1\leq i\leq k})$ a $(G,\PP)$-carrier graph of groups.

For each $i\in\{1,\ldots,k\}$ with $A_{c_i}\neq\{1\}$ let $O_i=[\gamma_i A_{c_i}\gamma_i^{-1}]\subset \pi_1(\A,v_0)$ for some $\mathbb A$-path $\gamma_i$ from $v_0$ to $c_i$.

The set $\OO_\mathcal{A}$ of these subgroups of $\pi_1(\A,v_0)$ is called a \emph{peripheral structure of $\mathcal{A}$}.
\end{definition}

\begin{remark}
Different choices for $\gamma_i$ yield conjugate $O_i$ reflecting the fact that vertex groups of a graph of groups determine a conjugacy class of subgroups. Thus the elements of $\OO_\mathcal{A}$ are uniquely determined up to conjugacy in $\pi_1(\A,v_0)$.
\end{remark}

\begin{lemma}\label{lem:equivalent metrics}
Let $(G,\PP)$ be a torsion-free relatively hyperbolic group with symmetric finite generating set $X$. Let $(\mathcal A,((\mathcal C_i,c_i))_{1\le i\le k})$ be an $M$-prenormal $(G,\PP)$-carrier graph such that $\abs{t}_{X\cup\mathcal{P}}^\mathcal{A}>0$ for any non-degenerate $\A$-almost-circuit $t$.

Let $\OO_\mathcal{A}$ be a peripheral structure of $\mathcal{A}$. Let $Y$ be a finite generating set of $\pi_1(\A,v_0)$ relative $\OO_\mathcal{A}$.

Then $\id\colon(\pi_1(\A,v_0),d_{Y\cup\mathcal{O}_\mathcal{A}})\to(\pi_1(\A,v_0),d_{X\cup\mathcal{P}}^\mathcal{A})$ is a quasiisometry.
\begin{proof}
It follows from standard arguments that the validity of the claim does not depend on the particular choice of the finite generating set $Y$ and the peripheral structure  $\OO_\mathcal{A}$. 
Hence, one may choose $\OO_\mathcal{A}$ and $Y$ as follows:

 Let $T$ be a spanning tree of $A$  that contains the stars underlying all non-trivial $\mathcal C_i$. For  each $v\in VA$ let $\gamma_v$ the unique reduced $\A$-path with trivial vertex elements from $v_0$ to $v$ in $T$. Let $\OO_\mathcal{A}$ be the set of non-trivial subgroups $[\gamma_{c_i}A_{c_i}\gamma_{c_i}^{-1}]$ of $\pi_1(\A,v_0)$ with $i\in\{1,\ldots,k\}$.

 Let  $V_{ess}$  be the set of essential vertices of $\mathcal{A}$.  As $\mathcal{A}$ is $M$-prenormal, $A_v$ is a relatively quasiconvex subgroup of $(G,\PP)$ with induced structure $\OO_v$ for every $v\in V_{ess}$. Moreover, there is a finite generating set $X_v$ of $A_v$ for any $v\in V_{ess}$. 

Now let
\[
\iota_v\colon A_v\to\pi_1(\A,v_0),\,a\mapsto[\gamma_va\gamma_v^{-1}]
\]
for any  $v\in V_{ess}$, let $y_e:=[\gamma_{\alpha(e)}(1,e,1)\gamma_{\omega(e)}^{-1}]$ for any $e\in EA\setminus ET$, and let
\[
Y:=\{y_e\mid e\in EA\setminus ET\}\cup\bigcup_{v\in E}\iota_v(X_v).
\]

Since $Y$ is finite and $\abs{o}_{X\cup\mathcal{P}}^\mathcal{A}\leq2\abs{\gamma_{c_i}}_{X\cup\mathcal{P}}^\mathcal{A}+1$ for any $o\in\OO_i$, there is a $C_1\geq1$ such that $\abs{z}_{X\cup\mathcal{P}}^\mathcal{A}\leq C_1$ for all $z\in Y\cup\mathcal{O}_\mathcal{A}$.

 Let $v\in V_{ess}$. It is easy to see that $\iota_v\colon(A_v,d_{X_v\cup \mathcal O_v})\to (\pi_1(\mathbb A,v_0),d_{Y\cup \mathcal O_{\mathcal A}})$ is a quasiisometric embedding. Moreover, $\id\colon(A_v,d_{X_v\cup\mathcal{O}_v})\to(A_v,d_{X\cup\mathcal{P}})$ is a quasiisometry  by Theorem~\ref{lem:BCP}. Thus there exists  $C_2\geq1$ such that $$\iota_v\colon(A_v,d_{Y\cup \mathcal O_{\mathcal A}})\to (\pi_1(\mathbb A,v_0),d^{\mathcal A}_{X\cup \mathcal P})$$ is a $(C_2,C_2)$-quasiisometric embedding  for any $v\in V_{ess}$.

Also, since $A$ is a finite graph and $\abs{t}_{X\cup\mathcal{P}}^\mathcal{A}>0$ for all non-degenerate $\A$-almost-circuits $t$, there is some $C_3\geq1$ such that the number of non-peripheral edges of any reduced $\A$-path $t$ is bounded by $C_3\abs{t}_{X\cup\mathcal{P}}^\mathcal{A}$.

Now let $g\in\pi_1(\A,v_0)$. It follows from the choice of $C_1$, that $\abs{g}_{X\cup\mathcal{P}}^\mathcal{A}\leq C_1\abs{g}_{Y\cup\mathcal{O}_\mathcal{A}}$.

Let $$t:=t_0s_1t_1\ldots t_{l-1}s_lt_l$$ be a reduced $\A$-path of minimal $X\cup\mathcal{P}$-length representing $g$, where the $s_j$ are the maximal peripheral $\A$-subpaths of $t$ and  $$t_j:=(a_0^j,e_1^j,a_1^j,\ldots,a_{l_j-1}^j,e_{l_j}^j,a_{l_j}^j)$$ for all $j\in\{0,\ldots,l\}$.

Define the following: 
\begin{enumerate}
\item For any $j\in\{0,\ldots,l\}$ and $m\in\{0,\ldots,l_j\}$ let $v\in VA$ such that $a_m^j\in A_v$. If $a_m^j=1$, let $w_{a_m^j}$ be the empty word. Otherwise $v$ is an essential vertex and there is a word over $Y\cup\mathcal{O}_\mathcal{A}$ representing $\iota_v(a_m^j)$ of length at most $C_2\abs{a_m^j}_{X\cup\mathcal{P}}+C_2$. Let $w_{a_m^j}$ be such a word.
\item Let $j\in\{0,\ldots,l\}$ and $m\in\{1,\ldots,l_j\}$. If $e_m^j\in ET$ let $w_{e_m^j}$ be the empty word, otherwise let $w_{e_m^j}:=y_{e_m^j}$.
\item For any $j\in\{0,\ldots,l\}$ let
\[
w_{t_j}:=w_{a_0^j}w_{e_1^j}w_{a_1^j}\ldots w_{a_{l_j-1}^j}w_{e_{l_j}^j}w_{a_{l_j}^j}.
\]
\item Let $j\in\{1,\ldots,l\}$. Since $T$ contains all peripheral stars, $\gamma_{\alpha(s_j)}s_j\gamma_{\omega(s_j)}^{-1}$ represents some element $o_j\in\mathcal{O}_\mathcal{A}$. Let $w_{s_j}:=o_j$.
\item Let $w_t:=w_{t_0}w_{s_1}w_{t_1}\ldots w_{t_{l-1}}w_{s_l}w_{t_l}$.
\end{enumerate}
By definition, $w_t\in(Y\cup\mathcal{O}_\mathcal{A})^\ast$ represents $g$. Moreover,
\[
\sum_{j=0}^l\sum_{m=0}^{l_j}\abs{w_{a_m^j}}+\sum_{j=1}^l\abs{w_{s_j}}\leq C_2\abs{t}_{X\cup\mathcal{P}}^\mathcal{A}+C_2.
\]
Since the number of non-peripheral edges is bounded by $C_3\abs{t}_{X\cup\mathcal{P}}^\mathcal{A}$, it follows that
\[
\abs{g}_{Y\cup\mathcal{O}_\mathcal{A}}\leq\abs{w_t}\leq(C_2+C_3)\abs{t}_{X\cup\mathcal{P}}^\mathcal{A}+C_2=(C_2+C_3)\abs{g}_{X\cup\mathcal{P}}^\mathcal{A}+C_2.
\]

Hence, $\id\colon(\pi_1(\A,v_0),d_{Y\cup\mathcal{O}_\mathcal{A}})\to(\pi_1(\A,v_0),d_{X\cup\mathcal{P}}^\mathcal{A})$ is a quasiisometry.
\end{proof}
\end{lemma}

\begin{remark}\label{rem:quasiisometry_constants}
It is clear, that there is no upper bound on the quasiisometry constants in Lemma~\ref{lem:equivalent metrics}, which is independent of the choice of $\OO_\mathcal{A}$ and $Y$. But even for the $\OO_\mathcal{A}$ and $Y$ constructed in the proof it is impossible to give such an upper bound without additional information on $\mathcal{A}$.
\end{remark}

\subsection{$M$-normal $(G,\PP)$-carrier graphs of groups}

Throughout this section $(G,\PP)$ is a torsion-free relatively hyperbolic group with symmetric finite generating set $X$, $(\mathcal{A},((\mathcal{C}_i,c_i))_{1\leq i\leq k})$ is an $M$-prenormal $(G,\PP)$-carrier graph of groups and $\nu:=\nu(G,\PP,X,M)$ is the constant from Lemma~\&~Definition~\ref{lem:bounded-generation}.

\begin{definition}\label{def:distinguished_edge}
Let $t=(a_0,e_1,a_1,\ldots,e_k,a_k)$ be an $\A$-path and $\bar s$ a realization of $t$. 

An edge $e$ of $\bar{s}$ which is labeled by some element of $\mathcal P$ is called a \emph{\ppp edge of $\bar{s}$}, if one of the following holds:
\begin{enumerate}
\item $e$ was added in the construction of the realization by replacing a subpath of $s$ corresponding to a maximal $\C_i$-subpath of $t$.
\item $e$ is an edge of some $s_{2i}$ (coming from an element of an essential vertex group) with $\len_X(e)>\frac{3}{2}\nu$.
\item  $e$ is the single edge of some $s_{2i}$  (coming from a non-trivial element of an essential vertex group, possibly with $\len_X(e)\le \frac{3}{2}\nu$).
\end{enumerate}
We refer to these as distinguished edges of the first, second and third type.
\end{definition}

\begin{definition}\label{def:tame}Let $t=(a_0,e_1,a_1,\ldots,e_k,a_k)$ be a reduced $\A$-path. $t$ is called tame if the following hold:
\begin{enumerate}
\item Suppose that $e_i$ and $\omega(e_i)$ are essential, that $\alpha(e_i)\in VC_j$ where $\mathcal C_j$ is a $P_{m_j}$-carrier star  and $s=s_0s_1s_2$ is a realization of the $\mathbb A$-path $(p,e_i,a_i)$ with $p\in P_{m_i}\neq 1$. Then no two distinguised edges of $s$ are connected.
\item Suppose that $e_i$ and $\alpha(e_i)$ are essential, that $\omega(e_i)\in VC_j$ where $\mathcal C_j$ is a $P_{m_j}$-carrier star  and $s=s_0s_1s_2$ is a realization of the $\mathbb A$-path $(a_{i-1},e_i,p)$ with $p\in P_{m_i}\neq 1$. Then no two distinguised edges of $s$ are connected.
\end{enumerate}
\end{definition}

Note that  condition 1. of Definition~\ref{def:tame} is satisfied for $t$ iff condition 2. is statisfied for $t^{-1}$, in particular $t$ is tame iff $t^{-1}$ is tame. The following lemma implies that we can always assume that elements are represented by tame $\mathbb A$-paths.

\begin{lemma} Any reduced $\mathbb A$-path is equivalent to a tame reduced $\mathbb A$-path.
\end{lemma}

\begin{proof} Suppose that a reduced $\mathbb A$-path  $t=(a_0,e_1,a_1,\ldots,e_k,a_k)$
is not tame. Possibly after replacing $t$ by $t^{-1}$ we may assume that there exists $i$ such that $e_i$ and $\omega(e_i)$ are essential, that $\alpha(e_i)\in VC_j$ where $\mathcal C_j$ is a $P_{m_j}$-carrier star  and that some realization $s=s_0s_1s_2$ of some $\mathbb A$-path $(p,e_i,a_i)$ with $p\in P_{m_i}\neq 1$ has connected distinguished components.

If follows from Corollary~\ref{cor_tame} that there exists some $$a\in A_{\omega(e_i)}\cap P_{m_j}^{g_j}=\hbox{Im}(\omega_{e_i})=g_j^{-1}A_{e_i}g_j$$ such that $|aa_i|_X<|a_i|_X$ and $|aa_i|_{X\cup \mathcal P}\le |a_i|_{X\cup\mathcal P}$. Write $a=g_j^{-1}a'g_j=\omega_{e_i}(a')$ with $a'\in A_{e_i}\subset P_{m_j}$ and in $t$ replace $a_i$ by $aa_i=\omega_{e_i}(a')a_i$ and $a_{i-1}$ by $a_{i-1}\alpha_{e_i}(a')=a_{i-1}a'$. This clearly yields an equivalent $\mathbb A$-path. Perform this operation as often as possible, the result is a tame $\mathbb A$-path equivalent to the original $\mathbb A$-path $t$.
\end{proof}


\begin{remark} It is in general not possible to find a tame $\mathbb A$-path that is of minimal $(X\cup\mathcal P)$-length in its equivalence class. Indeed, in the construction of the tame $\mathbb A$-graph a trivial element picked up in a peripheral star might be replaced by a non-trivial element.
\end{remark}


\begin{definition}\label{def:peripheral P-components}

Let $t$ be an $\A$-path. $t$ is said to have \emph{\ppp $\mathcal P$-components}, if the following is true for any realization $s$ of $t$:

Let $p$ be a $\mathcal P$-component of $s$ with $\len_X(p)\geq 10\nu$. Then $p=p_1ep_2$ where $e$ is a \ppp edge of $s$ and 
\[
\len_X(p_1),\len_X(p_2)\leq 5\nu.
\]
\smallskip
 Let $\mathcal A$ be an $M$-prenormal $(G,\PP)$ carrier graph of groups. $\mathcal A$ is called $M$-normal if every tame $\A$-path  has \ppp $\mathcal P$-components.
\end{definition}

\begin{remark}\label{rem_normal} In Definition~\ref{def:peripheral P-components} one could choose $p$ to be an arbitrary $P_i$-path  of $s$ as any such path is a $\mathcal P$-component of the realization of some tame $\mathbb A$-path
\end{remark}

\begin{lemma}\label{lem:peripheral_substitution}
Let $v\in VA$ be essential, $a\in A_v$, and $s$ a realization of the (degenerate) $\A$-path $(a)$, i.e.\ a geodesic whose label represents~$a$.

Suppose $s=s_1ps_2$ where $p$ is a \ppp edge of $s$ of the second type, i.e\ a $P_m$-edge with $\len(p)_X\geq \frac{3}{2}\nu$.

Then there are an essential edge $e\in EA$ with $\alpha(e)=v$ and $A_e\subseteq P_m$, an $\A$-path $t'=(a',e,1)$, and a realization $s'$ of $t'$ starting in $s_-$, such that
\begin{enumerate}
\item $\lab(s')\in\lab(s_1)P_m$, and
\item $\abs{\len_X^p(s')-\len_X^p(s_1)}\leq2\nu$.
\end{enumerate}

\end{lemma}
\begin{proof}
Since $\A$ is $M$-normal, it follows from Lemma~\ref{lem:bounded-generation} that there is some essential edge $e\in EA$ with $\alpha(e)=v$ and a path $s''=s_1'r_1p'r_2s_2'$ from $s_-$ to $s_+$ with
\begin{enumerate}
\item $s_1'$ and $s_2'$ are geodesics with $\lab(s_1'),\lab(s_2')\in A_v$,
\item $r_1$ and $r_2$ are geodesics with $\lab(r_1)=\lab(r_2)^{-1}=g_e$, and
\item $p'$ is a $P_m$-edge connected to $p$ with 
\[
d_X(\alpha(p),\alpha(p')),d_X(\omega(p),\omega(p'))\leq\nu.
\]
\end{enumerate}

Let $t':=(\lab(s_1'),e,1)$. Then it is clear, that $s':=s_1'r_1$ is a realization of $t'$ starting in $s_-$. Since $p$ and $p'$ are connected, it follows that
\[
\lab(s')=\lab(s_1'r_1)\in\lab(s_1)P_m,
\]
and further:
\begin{align*}
\abs{\len_X^p(s')-\len_X^p(s_1)} &= \abs{d_X((s_1')_-,(s_1')_+)+d_X((r_1)_-,(r_1)_+)-d_X((s_1)_-,(s_1)_+)}\\
&\leq \abs{d_X((s_1')_-,(s_1')_+)-d_X((s_1)_-,(s_1)_+)}+\abs{d_X((r_1)_-,(r_1)_+)}\\
&\leq d_X((s_1')_+,(s_1)_+)+d_X((r_1)_-,(r_1)_+)\\
&\leq d_X((r_1)_+,(s_1)_+)+2d_X((r_1)_-,(r_1)_+)\\
&= d_X(\alpha(p'),\alpha(p))+2\abs{g_e}_X\\
&\leq \nu+2\cdot\frac{\nu}{2}=2\nu.
\end{align*}\hspace*{0pt}
\end{proof}

\begin{remark}\label{remark_trivial_edge_elements} Note that condition 4 of Definition~\ref{def:nu-normal} implies that if $A_v\cap P_i\neq 1$ for some $i$ then there exists $e\in E_v$ with $\omega_e(A_e)=A_v\cap P_i$ and $g_e=1$. It is moreover a consequence of Lemma~\ref{lem:peripheral_substitution} 
 that this also holds if there exists $a\in A_v$ such that a geodesic word representing $a$ starts or ends with  element of $P_i$ of $X$-length at least $\nu/2$.
\end{remark}

Lemma~\ref{lem:peripheral_substitution} gives some justification for the definition of a \ppp edge in Definition~\ref{def:distinguished_edge}. Indeed, let $t=(\ldots ,a,\ldots)$ be an $\A$-path where $a$ is an element of some essential vertex group, $s$ is a realization of $t$ and $p$ is a \ppp edge in the geodesic subpath of $s$ associated with $a$. Then it follows from Lemma~\ref{lem:peripheral_substitution}, that $t$ can be modified into an $\A$-path $t'$ with $\nu_\mathcal{A}(t')=\nu_\mathcal{A}(t)$ by replacing the $\A$-subpath $(a)$ of $t$ with some $\A$-path $(a_1,e,a',e^{-1},a_2)$ for an essential edge $e$, such that for some realization $s'$ of $t'$ with the same endpoints as $s$, $p$ is connected to the \ppp edge $p'$ of $s'$ corresponding to $a'$ and the $X$-distance of the endpoints of $p$ and $p'$ is at most $2\nu$.

\begin{lemma}\label{lem:normalization}
	Let $(G,\PP)$ be a torsion-free relatively hyperbolic group with symmetric finite generating set $X$, $(\mathcal{A},((\mathcal{C}_i,c_i))_{1\leq i\leq k})$ an $M$-prenormal $(G,\PP)$-carrier graph and $\nu=\nu(G,\PP,X,M)$ be the constant from Lemma~\&~Definition~\ref{lem:bounded-generation}.
	
	Then either $\mathcal{A}$ is $M$-normal, or one of the following holds:
	\begin{enumerate}
		
		
		\item\label{item3} There is some non-degenerate, almost-simple $\A$-path $t$ such that $t_-$ and $t_+$ are essential and $\abs{t}_X^\mathcal{A}\leq 3\nu$.
		
		\item\label{item4} There is a circuit $(e_1,\ldots,e_l)$ in $A$ and some $m\in\{1,\ldots,n\}$, such that
		\begin{enumerate}
			\item $g_{e_j}\in P_m$ for all $j\in\{1,\ldots,l\}$, and
			
			\item $A_{e_i}=1$ for some $i\in\{1,\ldots,l\}$.
		\end{enumerate}
		
		\item\label{item5} 
		There exists some free  edge $e\in EA$ such that some geodesic word representing $g_e$ contains a letter from $\mathcal P$.\color{black}

		\item\label{item6} There is some simple path $(e_1,\ldots,e_l)$ in $A$ with 
		\begin{enumerate}
			\item $\alpha(e_1)\in VC_i$, $\omega(e_l)\in VC_j$, $m_i=m_j$,
			\item $g_{e_j}\in P_{m_i}$ for all $j\in\{1,\ldots,l\}$,
			\item $e_j$ is  free \color{black} for some $j\in\{1,\ldots,l\}$, and
			\item \[
			\sum_{j=1}^l\abs{g_{e_j}}_X\leq9\nu.
			\]
		\end{enumerate}
	\end{enumerate}
\end{lemma}

	\begin{proof}
		Assume $\mathcal{A}$ is not $M$-normal and show that this implies that one of \ref{item3}.-\ref{item6}. holds.
		
		First note that a circuit of non-peripheral edges in $EA$ which are all labeled by elements of some $P_m$ must either contain a  free  edge and satisfy \ref{item4}., or it has a subpath $(e,e')$ consisting of two essential edges such that $\alpha(e)$ and $\omega(e')$ are essential. In the second case, \ref{item3}.\ is fulfilled by this subpath.
		
		We can therefore assume that all reduced paths of non-peripheral edges labeled by elements of some $P_m$ are simple (and not closed). 
		
		Since $\mathcal{A}$ is not $M$-normal, there is some tame $\A$-path $t$ which does not have \ppp $\mathcal P$-components. Let $s$ be a realization of $t$. By assumption there is a $P_m$-component $p$ of $s$ for some $m$ with $\len_X(p)\geq 10\nu$ which is not of the form $p_1ep_2$ for some \ppp edge $e$ of $s$ with $\len_X(p_1),\len_X(p_2)\leq 5\nu$.
		
Clearly one of the following cases occurs:
		
		\begin{enumerate}
			\item $p$ has no \ppp edges,
			\item $p$ has at least two \ppp edges, or
			\item $p$ is of the form $p_1ep_2$, where $e$ is a \ppp edge of $s$, $p_1$ and $p_2$ have no \ppp edges, and $\len_X(p_1)>5\nu$ or $\len_X(p_2)>5\nu$.
		\end{enumerate}

We may assume that \ref{item5}.\ of Lemma~\ref{lem:normalization} does not occur, in particular no free edge is labeled by a non-trivial peripheral element.

\smallskip
(1)	Suppose that $p$ has no \ppp edges. It follows that  $p=p_1\cdot\ldots\cdot p_k$ such that the following hold:

\begin{itemize} 
\item $p_1$ ($p_k$) is a suffix (prefix) of length at most $\frac{3}{2}\nu$ of some geodesic representing an element of an essential vertex or of the element of some essential edge.
\item  For all $i\in\{2,\ldots,k-1\}$ the path $p_i$ is a geodesic representing the element of some essential edge (and is therefore of length at most $\nu/2$).
\end{itemize}

As $\len_X(p)\geq 10\nu\ge 5\nu$ it follows that $k\ge 6$. Thus, for some $i$ the subpath $p_ip_{i+1}$ is a subpath of $X$-length at most $\nu$ that corresponds to a sub-$\A$-path of $t$ connecting two essential vertices which is simple by the  discussion at the beginning of the proof. Hence, \ref{item3}.\ occurs.

		\smallskip
	(2)	Suppose now that $p$ has two \ppp edges. Choose \ppp edges $e_1$ and $e_2$ such that\ $p$ has a subpath of the form $e_1qe_2$ where $q$ does not contain a \ppp edge. If $\len_X(q)\geq3\nu$, one can argue as in (1), hence we may assume that $\len_X(q)\leq3\nu$.
		

If $e_1$ and $e_2$ are \ppp edges of the first type it follows that $q=q_1\cdot\ldots\cdot q_k$, where each $q_i$ is a geodesic representing the element of some essential edge. Indeed a subpath of $q$ corresponding to an element of a vertex group would define a \ppp edge (of the second or third type) contradicting our assumption. The above remark implies that $q$ is the realization  of some non-degenerate, simple $\A$-subpath $t'$ of $t$ without peripheral edges, such that $t'_-\in VC_i$ and $t'_+\in VC_j$ with $m_i=m_j=m$.

		 As $\mathcal{A}$ is $M$-prenormal and hence any vertex group $A_v$ of some essential vertex $v$ has induced structure $\OO_v$, it follows  (see Remark~\ref{remark_g_i}) that if $t'$ contains only essential edges, it must pass through at least two essential vertices. Thus, a subpath of the path underlying $t'$ fulfills \ref{item3}. Otherwise, $t'$ passes through at least one  free   edge (with necessarily trivial edge label) and the underlying path fulfills~\ref{item6}.
		
		If both $e_1$ and $e_2$  are \ppp  edges of the second or third type,  then 
		by Remark~\ref{remark_trivial_edge_elements}, 
		$q$ is the realization of an $\A$-subpath of $t$ which fulfills \ref{item3}.
		
		Now let exactly one of $e_1$ and $e_2$ be a \ppp edge of the first type. W.l.o.g. we may assume that $e_1$ is of the first type and $e_2$ of the second or third type coming from an element of the vertex group of some essential vertex $v$. It follows from  Lemma~\ref{lem:peripheral_substitution} and Remark~\ref{remark_trivial_edge_elements} that there exists a unique essential edge $e$ with $\alpha(e)=v$ and $g_e=1$ such that $\omega(e)$ lies in some $P_m$-star $C_i$.
		
		If $t'$ ist the $\mathbb A$-subpath of $t$ representing $q$ then $\hat t=t'\cdot (1,e,1)$ is an $\mathbb A$ path connecting the $P_m$ star corresponding to $e_1$ to the $P_m$-star $C_i$. Note that $q$ is a realization of $t'$ as $g_e=1$. Note further that $t'$ and therefore also $\hat t$ does not contain a peripheral edge.
		
Note that the tameness of $t$ imply that $\hat t$ is reduced. Thus  $\hat t$ is reduced and non-degenerate and the same argument as in the first case shows that $\mathcal{A}$ fulfills \ref{item3}.\ or~\ref{item6}.

\smallskip
(3) 	Lastly, suppose $p$ is of the form $p_1ep_2$, where $e$ is a \ppp edge of $s$, $p_1$ and $p_2$ have no \ppp edges, and w.l.o.g.\ $\len_X(p_2)>5\nu$. We can now argue as in (1) with $p_2$ playing the role of $p$.
\end{proof}

\subsection{Trivial Segments of Graphs of Groups}\label{sec:Trivial Segments of Graphs of Groups}

In the statement of the combination theorem (Theorem~\ref{thm:main}) the notion of a trivial segment of an $\A$-path plays a crucial role.

\begin{definition}
Let $\A$ be a finite graph of groups and $v_0\in VA$. A non-degenerate, simple or simple closed path $s=(e_1,\ldots,e_k)$ in $A$ is called a \emph{trivial segment} of $\A$ relative $v_0$, if
\begin{enumerate}
\item $\val(\omega(e_l))=2$ and $\omega(e_l)\neq v_0$ for all $l\in\{1,\ldots,k-1\}$, and

\item $A_{\omega(e_l)}=1$ for all $l\in\{1,\ldots,k-1\}$.
\end{enumerate}

Note that $s$ is a (maximal) trivial segment, if and only if $s^{-1}$ is a (maximal) trivial segment. Note moreover that any edge $e\in EA$ with trivial edge group is contained in a unique maximal trivial segment. We denote by $N(\mathbb A,v_0)$ the number of (unordered) pairs $\{s,s^{-1}\}$ of maximal trivial segments of $\mathbb A$ relative $v_0$. For a $(G,\PP)$-carrier graph $\mathcal A$ (with base vertex $v_0$) we put $N(\mathcal A):=N(\mathbb A,v_0)$. 


Let $t=(a_0,e_1, \ldots ,e_k,a_k)$ be an $\A$-path and let
\[
s=(1,e_i,a_i\ldots,a_{j-1},e_j,1)
\]
be an $\A$-subpath of $t$. Then $s$ is called a \emph{trivial segment} of $t$, if $(e_i,\ldots,e_j)$ is a trivial segment of $\A$. Note that necessarily $a_l=1$ for $l\in\{i,\ldots,j-1\}$.
\end{definition}

The following lemma provides an upper bound on $N(\mathcal A)$ for  a $(G,\PP)$-carrier graph $\mathcal A$.

\begin{lemma}\label{lem:N(A)} Let $\A$ be a finite graph of groups and $v_0\in VA$ such that any valence $1$ vertex distinct from $v_0$ has a non-trivial vertex group. Let $n$ be the number of free factors of $\A$. 
Then
\[
N(\A)\leq3b_1(A)+2(n+1),
\]
where $b_1(A)$ denotes the first Betti number of $A$.

\begin{proof}
Let $\bar{A}$ be the graph obtained from $A$ by replacing every maximal trivial segment of $\A$ relative $v_0$ with a single edge and every $A_i$ with a single vertex.

A simple induction shows, that for any graph $\Gamma$ the number of edges (counting both orientations) is bounded from above by $6b_1(\Gamma)+4k$, where $k$ is the number of vertices of valence at most $2$.

Every vertex of $\bar{A}$ of valence at most 2 corresponds to some $\A_i$ or to $v_0$. Since $N(\A)$ is at most half the number of edges of $\bar{A}$, it follows that
\[
N(\A)\leq\frac{1}{2}\abs{E\bar{A}}\leq3b_1(\bar{A})+2\abs{\{v\in V\bar{A}\mid\val(v)\leq2\}}\leq3b_1(A)+2(n+1).
\]
\end{proof}
\end{lemma}

\section{Combination Theorem}\label{sec:Combination Theorem}
The goal of this section is to prove a combination theorem for relatively quasiconvex subgroups.

\subsection{$\pi_1$-injective $(G,\PP)$-carrier graphs of groups}

The proof of the combination theorem in Section~\ref{section:combination}
is based on Proposition~\ref{prop:1}. It gives sufficient conditions for the map $\nu_{\mathcal A}$ associated to a $(G,\PP)$-carrier graph of groups to be a quasi-isometric embedding and therefore injective. It is an application of Lemma~\ref{lem:long parabolic edges} in the context of $(G,\PP)$-carrier graphs of groups.

\begin{proposition}\label{prop:1}
Let $G$ be a group, which is torsion-free, hyperbolic relative to $\PP=\{P_1,\ldots,P_n\}$ and has a symmetric finite generating set $X$. Let $(\mathcal{A},((\mathcal{C}_i,c_i))_{1\leq i\leq k})$ be an $M$-normal $(G,\PP)$-carrier graph of groups. Let $C\geq1$.

There are constants $C'=C'(G,\PP,X,C)$ and $L'=L'(G,\PP,X,M,C)$, such that one of the following holds:

\begin{enumerate}

\item There is a reduced $\A$-path $t$ with $\abs{t}_X^\mathcal{A}\leq L'$, $t_-\in VC_i$, $t_+\in VC_j$ and $m_i=m_j$, which is not entirely contained in $C_i$ and whose label $\nu_\mathcal{A}(t)$ represents an element of $P_{m_i}$.

\item There is a subdivision $\mathcal{A}'$ of $\mathcal{A}$ along at most two free edges and a reduced $\A'$-path $t$ with $\abs{t}_X^\mathcal{A'}\leq L'$, such that
\[
\abs{\nu_\mathcal{A'}(t)}_{X\cup\mathcal{P}}<\frac{\abs{t}_{X\cup\mathcal{P}}^\mathcal{A'}}{C}-C.
\]

\item  $\nu_{\mathcal{A}}\colon(\pi_1(\A,v_0),d_{X\cup\mathcal{P}}^\mathcal{A})\to(G,d_{X\cup\mathcal{P}})$ is a $(C',C')$-quasiisometric embedding. Moreover, for any peripheral structure $\OO_\mathcal{A}$ of $\mathcal{A}$, $\nu_{\mathcal{A}}(\OO_\mathcal{A})$ is an induced structure of $(G,\PP)$ on $\Ima\nu_\mathcal{A}$.
\end{enumerate}

Moreover for any tame $\mathbb A$-path with realization $s$ the following hold:
\begin{enumerate}
\item[(i)] $s$ is a $(C',C')$-quasigeodesic.
\item[(ii)] If $p$ and $p'$ are $\mathcal P$-components in $s$ such that $\len_X(p),\len_X(p')\ge L'$ then $p$ and $p'$ are not connected.
\end{enumerate}
\end{proposition}

\begin{proof}
Let $\nu:=\nu(G,\PP,X,M)$ be the constant from Lemma~\&~Definition~\ref{lem:bounded-generation}. Let $C_1:=C+4\nu$, $C_2:=C_1^2+3C_1$ and $C',K\geq1$ such that the following hold:
\begin{enumerate}
\item Any $(K,C_2,C_2)$-local-quasigeodesic in $\Cay(G,X\cup\mathcal{P})$ is a $(C',C')$-quasigeodesic.
\item $d_{X\cup\mathcal P}(\gamma_-,\gamma_+)>1$ for any $(C',C')$-quasigeodesic $\gamma$ of length at leat $K-2$.
\end{enumerate} 

Such $C'$ and $K$ exist  by Théorème 1.4 of Chapitre 3 of \cite{Coornaert1990}.  Let $\varepsilon_1:=10\nu$, $\varepsilon:=\varepsilon(G,\PP,X,1,K)$ as in Lemma~\ref{lem:BCP}, and let $\varepsilon_2:=2^K(K\varepsilon+\varepsilon_1)$. Let $L':=\varepsilon_2K+14\nu$.

Suppose that 1. and 2. are false and show that this implies 3.

Let $g\in\pi_1(\A,v_0)$, $t$ some tame $\A$-path representing $g$ and $s$ a realization of $t$. The $M$-normality guarantees that $t$ has  \ppp $\mathcal P$-components. As \[
\abs{g}_{X\cup\mathcal{P}}^\mathcal{A}\le \abs{t}_{X\cup\mathcal{P}}^\mathcal{A}=\len(s)
\]
and
\[
\abs{\nu_\mathcal{A}(g)}_{X\cup\mathcal{P}}\le d_{X\cup\mathcal{P}}(s_-,s_+).
\] it suffices to show  that $s$ is a $(K,C_2,C_2)$-local-quasigeodesic and therefore  a $(C',C')$-quasigeodesic. Thus one needs to show that any subpath of $s$ of length at most $K$ satifies the conditions spelt out in Lemma~\ref{lem:long parabolic edges} where the constant $C$ in the statement of Lemma~\ref{lem:long parabolic edges} is $C_1$.

Write $s=s_0p_1s_1\ldots s_{l-1}p_ls_l$, where the $p_m$ are the \ppp edges of $s$ with $\len_X(p_m)>\varepsilon_2$.  Let $i_m$ be such that $p_m$ is a $P_{i_m}$-edge. As $t$ has \ppp $\mathcal P$-components this implies that $\PP$-edges of any $s_j$  have $X$-length at most $\varepsilon_2\geq 10\nu$. It follows that
\[
\len_X^p(\bar{s})\leq\len_X(\bar{s})\leq\varepsilon_2\len(\bar{s})\leq\varepsilon_2K
\]
 for any subpath $\bar s$ of some $s_j$ of length at most $K$.

Any subpath of $s$ of length at most $K$ fulfills Lemma~\ref{lem:long parabolic edges}~1.\ by definition and it remains to show that it also fulfills Lemma~\ref{lem:long parabolic edges}~2.\ and 3..

\smallskip

It suffices to show that every $s_j$ is a $(K,C_1,C_1)$-local-quasigeodesic to show that every subpath of $s$ of length at most $K$ fulfills Lemma~\ref{lem:long parabolic edges}~2.

Let $\bar{s}$ be a subpath of $s_j$ with $\len(\bar{s})\leq K$ . 
If $\bar{s}$ is a subpath of some geodesic subpath of $s$ corresponding to some essential vertex element of $t$ or to some $g_e$ for an edge $e$ of $t$, $\bar{s}$ is itself a geodesic and there is nothing to show. Assuming that this is not the case, Lemma~\ref{lem:subdivision} can be applied twice to obtain a subdivision $\mathcal{A'}$ of $\mathcal{A}$ along at most two free edges and a path $s'$ with $d_X(s'_-,\bar{s}_-),d_X(s'_+,\bar{s}_+)\leq\nu$, which is a realization of some $\A'$-path $t'$.

It follows that
\[
\len(\bar{s})\leq\len(s')+2\nu,
\]
\[
d_{X\cup\mathcal{P}}(s'_-,s'_+)\leq d_{X\cup\mathcal{P}}(\bar{s}_-,\bar{s}_+)+2\nu,
\]
and
\[
\abs{t'}_X^\mathcal{A'}=\len_X^p(s')\leq\len_X^p(\bar{s})+2\nu\leq\varepsilon_2K+2\nu\leq L'.
\]

%
%
%

Note that
\[
\abs{\nu_\mathcal{A'}(t')}_{X\cup\mathcal{P}}\geq\frac{\abs{t'}_{X\cup\mathcal{P}}^\mathcal{A'}}{C}-C.
\]
as 2.\ was assumed to be false. It follows that
\begin{align*}
\frac{\len(\bar{s})}{C_1}-C_1 &\leq \frac{\len(s')+2\nu}{C_1}-C_1\\
&\leq \frac{\abs{t'}_{X\cup\mathcal{P}}^\mathcal{A'}}{C}-C-2\nu\\
&\leq \abs{\nu_\mathcal{A'}(t')}_{X\cup\mathcal{P}}-2\nu\\
&= d_{X\cup\mathcal{P}}(s'_-,s'_+)-2\nu\\
&\leq d_{X\cup\mathcal{P}}(\bar{s}_-,\bar{s}_+).
\end{align*}

Therefore, $s_j$ is a $(K,C_1,C_1)$-local-quasigeodesic and any subpath of $s$ of length at most $K$ fulfills Lemma~\ref{lem:long parabolic edges}~2.

\medskip
Suppose $s$ does not fulfill Lemma~\ref{lem:long parabolic edges}~3. Thus, there exists a subpath $ps'p'$ of $s$ with
\begin{enumerate}
\item $p$ and $p'$ are $P_i$-paths of $s$,
\item $\len_X(p),\len_X(p')\geq\varepsilon_1=10\nu$, and
\item $\len_X(s')\leq\varepsilon_2$
\end{enumerate}
where $s'$ is labeled by an element of $P_i$. In particular,
\[
\len_X^p(s')\leq\len_X(s')\leq\varepsilon_2.
\]

Recall that  $t$  has \ppp $\mathcal P$-components. Since $\len_X(p),\len_X(p')\geq10\nu$, $p$ and $p'$ must be of the form $q_1qq_2$ and $q_1'q'q_2'$, where $q_1$, $q_2$, $q_1'$ and $q_2'$ are $P_i$-paths of $X$-length at most $5\nu$ and $q$ and $q'$ are \ppp $P_i$-edges, see also Remark~\ref{rem_normal}. 

By applying Lemma~\ref{lem:peripheral_substitution} to $q_2s'q_1'$ at most twice, obtain a path $s''$, which is a realization of some (not necessarily reduced) $\A$-path $t'$ with $t'_-$ and $t'_+$ peripheral, which is labeled by an element of $P_i$, and which fulfills
\[
\len_X^p(s'')\leq\len_X^p(q_2s'q_1')+4\nu.
\]
$t'$ is not entirely contained in some peripheral star, since $s'$ was not a $P_i$-path. Moreover
\[
\abs{t'}_X=\len_X^p(s'')\leq\len_X^p(q_2s'q_1')+4\nu\leq 5\nu+\varepsilon+5\nu+4\nu=\varepsilon_2+14\nu\leq L'.
\]

It $t'$ is not reduced one obtains a contradiction to the tameness of $t$. Thus $t'$ is reduced. Hence, $t'$ fulfills 1., a contradiction.

Thus $\nu_{\mathcal{A}}\colon(\pi_1(\A,v_0),d_{X\cup\mathcal{P}}^\mathcal{A})\to(G,d_{X\cup\mathcal{P}})$ is a $(C',C')$-quasi\-iso\-metric embedding.

\smallskip

Let $\OO_\mathcal{A}$ be a peripheral structure of $\mathcal{A}$. $\abs{g^n}_{X\cup\mathcal{P}}^\mathcal{A}$ is unbounded for $g\in\pi_1(\A,v_0)$ that   is not conjugate into a vertex group. Since $\nu_\mathcal{A}$ is a quasiisometric embedding, this also implies that $\abs{\nu_\mathcal{A}(g)^n}_{X\cup\mathcal{P}}$ is unbounded. Thus, $\nu_\mathcal{A}(g)$ is not parabolic. It is now easy to see, that every element of $\pi_1(\A,v_0)$, which is mapped to a parabolic element under $\nu_\mathcal{A}$ must be a parabolic element of some vertex group of $\A$. Since $\mathcal{A}$ is $M$-normal, it must then be conjugate into some element of $\OO_\mathcal{A}$.

Now let $O,O'\in\OO_\mathcal{A}$, $o\in O$, $o'\in O'$ and $g\in\pi_1(\A,v_0)$ such that $\nu_\mathcal{A}(gog^{-1}o')$ is parabolic. By the above, $gog^{-1}o'$ must be a parabolic element of some vertex group of $\A$. Since $\mathcal{A}$ is $M$-normal, it follows that $O=O'$. Thus $\nu_\mathcal{A}(\OO_\mathcal{A})$ is an induced structure of $(G,\PP)$ on $\Ima\nu_\mathcal{A}$.

This proves that for the choosen constants one of 1.-3. holds. (i) is implicit in the proof and (ii) follows from item 2. in the choice of the constants $C'$ and $K$ and the last claim of Lemma~\ref{lem:long parabolic edges}\vspace{-\belowdisplayskip}.\[\,\]

\end{proof}

\subsection{Combination Theorem}\label{section:combination}

It is now possible to prove the following combination theorem. It states provides concrete (and often checkable) conditions that guarantee that Proposition~\ref{prop:1} can be applied to show that the map $\nu_{\mathcal A}$ is a quasiisometric embedding.

\begin{theorem}\label{thm:main}
Let $G$ be a group, which is torsion-free, hyperbolic relative to $\PP=\{P_1,\ldots,P_n\}$ and has finite symmetric generating set $X$. Let $(\mathcal A,((\mathcal C_i,c_i))_{1\le i\le k})$ be an $M$-normal $(G,\PP)$-carrier graph.

There are $C=C(G,\PP,X,M,N(\mathcal A))$ and $L=L(G,\PP,X,M,N(\mathcal A))$ such that the following holds:

$\nu_{\mathcal{A}}\colon(\pi_1(\A,v_0),d_{X\cup\mathcal{P}}^\mathcal{A})\to(G,d_{X\cup\mathcal{P}})$ is a $(C,C)$-quasiisometric embedding and $\nu_{\mathcal{A}}(\OO_\mathcal{A})$ is an induced structure of $(G,\PP)$ on $\Ima\nu_\mathcal{A}$ for any peripheral structure $\OO_\mathcal{A}$ of $\mathcal{A}$, unless one of the following holds:

\begin{enumerate}

\item[(A1)]\label{item:CT1} \begin{enumerate}
\item There is some $\A$-path $t$ with $\abs{t}_X^\mathcal{A}\leq L$, $t_-\in VC_i$, $t_+\in VC_j$, $m_i=m_j$ and $\nu_\mathcal{A}(t)\in P_{m_i}$, which passes through some trivial segment of $\A$ exactly once.

\item There is some $\A$-path $t$ with $\abs{t}_X^\mathcal{A}\leq L$, $t_-\in VC_i$, $t_+$ essential, and $A_{t_+}\cap\nu_\mathcal{A}(t)^{-1}P_{m_i}\nu_\mathcal{A}(t)\neq1$, which passes through some trivial segment of $\A$ exactly once.
\end{enumerate}

\item[(A2)]\label{item:CT2} There are essential vertices $v,v'\in VA$ and some non-degenerate, almost-simple $\A$-path $t$ starting in $v$ and ending in $v'$ with $\abs{t}_X^\mathcal{A}\leq L$.

\item[(A3)]\label{item:CT3} There is an $\A$-almost-circuit $t$ with $\abs{t}_X^\mathcal{A}\leq L$ which has a trivial segment.

\item[(A4)]\label{item:CT4}
\begin{enumerate}
\item There is some essential edge $e$ with $\alpha(e)\in VC_i$ and some $p\in A_{\alpha(e)}\setminus\alpha_e(A_e)$ with
\[
\abs{p}_X\leq L.
\]

\item There is some almost-simple $\A$-path $t=(a_0,e_1,a_1,\ldots,e_k,a_k)$ with $\abs{t}_X^\mathcal{A}\leq L$, $t_-\in VC_i$, $A_{e_1}=1$ and $t_+$ essential, and some $p\in A_{t_-}\setminus\{1\}$ such that
\[
\abs{p}_X\leq L.
\]

\end{enumerate}

\item[(A5)]\label{item:CT5} There is some almost-simple $\A$-path $t$ with $\abs{t}_X^\mathcal{A}\leq L$, $t_-\in VC_i$, $t_+\in VC_j$ and passing through some trivial segment of $\A$, and elements $p\in A_{t_-}\setminus\{1\},p'\in A_{t_+}\setminus\{1\}$, such that
\[
\abs{p}_X,\abs{p'}_X\leq L.
\]


\item[(A6)]\label{item:CT6} There is some subdivision $\mathcal{A}'$ of $\mathcal{A}$ along at most two trivial edges, an $\A'$-path $t$ with $\abs{t}_X^\mathcal{A'}\leq L$, some trivial segment $t'$ of $t$ and some $a\in A_{t_+}$ with
\begin{enumerate}
\item $t$ passes exactly once through $t'$ or $t'^{-1}$,
\item $\abs{\nu_\mathcal{A'}(t)a}_{X\cup\mathcal{P}}+5<\abs{t'}_{X\cup\mathcal{P}}^\mathcal{A'}$.
\end{enumerate}
\end{enumerate}
\end{theorem}

It follows from Theorem~\ref{thm:main} and Lemma~\ref{lem:equivalent metrics} that for any $M$-normal $(G,\PP)$-carrier graph $\mathcal{A}$ which does not satisfy any of the conditions (A1)-(A6) the map $$\nu_{\mathcal{A}}\colon(\pi_1(\A,v_0),d_{Y\cup\mathcal{O}_\mathcal{A}})\to(G,d_{X\cup\mathcal{P}})$$ is a quasiisometric embedding for any peripheral structure $\OO_\mathcal{A}$ of $\mathcal{A}$ and finite generating set $Y$ of $\pi_1(\A,v_0)$ relative $\OO_\mathcal{A}$.

Although Theorem~\ref{thm:main} yields that $\nu_{\mathcal{A}}\colon(\pi_1(\A,v_0),d_{X\cup\mathcal{P}}^\mathcal{A})\to(G,d_{X\cup\mathcal{P}})$ is a $(C,C)$-quasiisometric embedding, it follows from Remark~\ref{rem:quasiisometry_constants} that without further assumptions on $\mathcal{A}$, it is generally not possible to give a $\bar{C}$ such that $\nu_{\mathcal{A}}\colon(\pi_1(\A,v_0),d_{Y\cup\mathcal{O}_\mathcal{A}})\to(G,d_{X\cup\mathcal{P}})$ is a $(\bar{C},\bar{C})$-quasiisometric embedding.

For the remainder of this section, assume that $\mathcal{A}$ is as in Theorem~\ref{thm:main}, i.e.\ an $M$-normal $(G,\PP)$-carrier graph. Let $\nu=\nu(G,\PP,X,M)$. We establish some lemmas before we proceed with the proof of Theorem~\ref{thm:main}.

The first lemma gives a relationship between Proposition~\ref{prop:1}~2.\ and (A6).

\begin{lemma}\label{lem:sublemma3}
Let $\delta\in(0,1]$,
\[
C\geq\max\left\{10,\sqrt{\frac{4}{\delta}},\frac{N(\mathcal A)+1}{\delta}\right\},
\]
$\mathcal{A'}$ some subdivision of $\mathcal{A}$ along at most two trivial edges, and $t=s_1t's_2$ a reduced $\A'$-path such that
\begin{enumerate}
\item\label{lem:sublemma3_1} $t'$ is almost-simple and there is no edge pair which is crossed by $t'$ and by $s_1$ or by $t'$ and $s_2$,

\item\label{lem:sublemma3_2} any vertex element of $t'$ that comes from an essential vertex is trivial,
\item\label{lem:sublemma3_3}  no edge of $t'$ is essential,


\item $\abs{t'}_{X\cup\mathcal{P}}^\mathcal{A'}\geq\delta\abs{t}_{X\cup\mathcal{P}}^\mathcal{A'}$, and

\item $\abs{\nu_\mathcal{A'}(t)}_{X\cup\mathcal{P}}<\frac{\abs{t}_{X\cup\mathcal{P}}^\mathcal{A'}}{C}-C$.
\end{enumerate}

Then $t'$ contains a trivial segment $s$ such that  $|s|^{\mathcal A'}_{X\cup\PP}>\abs{\nu_\mathcal{A'}(t)}_{X\cup\mathcal{P}}+6$. In particular $t$ travels $s$ once and $s^{-1}$ not at all.
\end{lemma}
\begin{proof}
By assumptions \ref{lem:sublemma3_2}. and \ref{lem:sublemma3_3}., $t'$ is of the form
\[
t'=p_0t_1p_1\ldots t_Np_N,
\]
where
\begin{enumerate}
\item every $t_l$ is a maximal trivial segment of $t'$,

\item every $p_l$ is either trivial or a peripheral $\A'$-subpath of $t$ that travels in some non-trivial peripheral star.
\end{enumerate}

Since $t'$ is almost-simple and therefore reduced, $t_l$ is a maximal trivial segment of $\A$ for $l\in\{2,\ldots,N-1\}$. It follows, that $\A'$ has at least $N-1$ pairwise distinct maximal trivial segments. Since $N(\mathcal A')=N(\mathcal A)$, this implies $N\leq N(\mathcal A)+1\leq\delta C$.

Let
\[
L_t:=\sum_{l=1}^N\abs{t_l}_{X\cup\mathcal{P}}^\mathcal{A'}
\]
and let $L_p$ be the number of the non-trivial $p_l$. Thus, $L_p\leq N+1$.

It follows that
\begin{align*}
\abs{t}_{X\cup\mathcal{P}}^\mathcal{A'} &\leq \abs{s_1}_{X\cup\mathcal{P}}^\mathcal{A'}+\abs{t'}_{X\cup\mathcal{P}}^\mathcal{A'}+\abs{s_2}_{X\cup\mathcal{P}}^\mathcal{A'}\\
&= L_t+L_p+\abs{s_1}_{X\cup\mathcal{P}}^\mathcal{A'}+\abs{s_2}_{X\cup\mathcal{P}}^\mathcal{A'}\\
&\leq L_t+L_p+\abs{t}_{X\cup\mathcal{P}}^\mathcal{A'}-\abs{t'}_{X\cup\mathcal{P}}^\mathcal{A'}+2\\
&\leq L_t+N+(1-\delta)\abs{t}_{X\cup\mathcal{P}}^\mathcal{A'}+3.
\end{align*}

By assumption 5., $\abs{t}_{X\cup\mathcal{P}}^\mathcal{A'}>C\abs{\nu_\mathcal{A'}(t)}_{X\cup\mathcal{P}}^\mathcal{A'}+C^2\ge C^2$ and hence
\[
L_t+N\geq\delta\abs{t}_{X\cup\mathcal{P}}^\mathcal{A'}-3\geq\delta C^2-3\geq\delta\cdot \frac{4}{\delta}-3=1
\]
Since $N=0$ implies $L_t=0$ it follows that $N\ge 1$. Now, as $\delta C\geq N(\mathcal A)+1\ge N$ and $C\geq10$, it follows that
\[
\frac{L_t}{N}\geq\frac{\delta\abs{t}_{X\cup\mathcal{P}}^\mathcal{A'}}{N}-\frac{3}{N}-1\geq\frac{\delta\abs{t}_{X\cup\mathcal{P}}^\mathcal{A'}}{\delta C}-4\geq\frac{\abs{t}_{X\cup\mathcal{P}}^\mathcal{A'}}{C}-C+6>\abs{\nu_\mathcal{A'}(t)}_{X\cup\mathcal{P}}+6.
\]

Since $L_t/N$ is the average $\abs{\cdot}_{X\cup\mathcal{P}}^\mathcal{A'}$-length of the $t_l$, there has to be some $l\in\{1,\ldots,N\}$ such that $\abs{t_l}_{X\cup\mathcal{P}}^\mathcal{A'}>\abs{\nu_\mathcal{A'}(t)}_{X\cup\mathcal{P}}+6$. Thus the conclusion holds for $s=t_l$.
\end{proof}

\begin{lemma}\label{lem:sublemma1}
Let $L\geq0$ and suppose that none of (A2)-(A5) hold for this $L$. Let $\mathcal{A'}$ be some subdivision of $\mathcal{A}$ along at most two trivial edges. Let $t$ be a reduced $\A'$-path with $\abs{t}_X^\mathcal{A'}\leq L$. Then one of the following holds:
\begin{enumerate}
\item $t$ is almost-simple and has at most one essential vertex, or

\item the path $r$ underlying $t$ is of the form $r_1r_2(r_2)^{-1}r_3$ with $r_2$ non-degenerate, such that the following hold:
\begin{enumerate}
\item $r_1r_2$, $r_1r_3$ and $(r_2)^{-1}r_3$ are almost-simple, and

\item all vertices of $r$ distinct from $(r_2)_+$ are non-essential.

\end{enumerate}
\end{enumerate}

\begin{proof}
Let $t=(a_0,e_1,\ldots,e_k,a_k)$ be a reduced $\A'$-path with $\abs{t}_X^\mathcal{A'}\leq L$, $v_0:=\alpha(e_1)$, $v_i:=\omega(e_i)$ for $i\in\{1,\ldots,k\}$. Suppose that $t$ is not not as in 1.\ or 2. Show that this implies that one of (A2)-(A5) must hold.

It is easy to see by considering a shortest $\A$-subpath of $t$ that is also not not as in 1.\ or 2., that one of the following must be true:
\begin{enumerate}
\item There are $l_1<l_2\in\{0,\ldots,k\}$ such that
\[
t':=(1,e_{l_1+1},a_{l_1+1},\ldots,a_{l_2-1},e_{l_2},1)
\]
is almost-simple, $v_l$ is non-essential for $l\in\{l_1+1,\ldots,l_2-1\}$, and for each $i\in\{1,2\}$ one of the following holds:
\begin{enumerate}
\item $v_{l_i}$ is essential, or

\item  $v_{l_i}$ is non-essential and $l_i\notin\{0,k\}$ and $e_{l_i}=e_{l_i+1}^{-1}$.
\end{enumerate}

\item There are $l_1<l_2\in\{0,\ldots,k\}$ with
\begin{enumerate}
\item $v_{l_1}=v_{l_2}$,
\item $e_{l_1}=e_{l_2+1}^{-1}$ if $v_{l_1}=v_{l_2}$ is a central peripheral vertex, and
\item $(e_{l_1+1},\ldots,e_{l_2})$ is either an almost-circuit with at most one essential vertex or as described in 2.\ above with $r_1$ and $r_3$ non-degenerate.
\end{enumerate}
\end{enumerate}

Show that in both cases one of (A2)-(A5) must hold.
\begin{enumerate}
\item Note first that in case (b), i.e.  if $e_{l_i}=e_{l_i+1}^{-1}$ for $i\in\{1,2\}$, it follows that $a_{l_i}\notin\omega_{e_{l_i}}(A_{e_{l_i}})$ as $t$ is reduced. This implies that $e_{l_i}$ and $e_{l_i+1}$ are not peripheral and that $\omega(e_{l_i})$  does no come from a subdivision. Thus in both cases $v_{l_1},v_{l_2}\in VA\subset VA'$, and $t'$ can be assumed to be an $\A$-path rather than an $\A'$-path. It follows, that $e_{l_1+1}$ and $e_{l_2}$ (and in case (b) $e_{l_1}$ and $e_{l_2+1}$) are non-peripheral. Hence, $t'$ and $(a_{l_i})$ are full $\A'$-subpaths of $t$ and therefore $\abs{t'}_X^\mathcal{A}\leq\abs{t}_X^\mathcal{A'}\leq L$ and $\abs{a_{l_i}}_X\leq\abs{t}_X^\mathcal{A'}\leq L$ for $i\in\{1,2\}$.

If $v_{l_1}$ and $v_{l_2}$ are both essential, $t'$ fulfills (A2).

If neither $v_{l_1}$ nor $v_{l_2}$ is essential, then both must be peripheral since $a_{l_1},a_{l_2}\neq1$. As $e_{l_1+1}$ is non-peripheral and since $v_{l_1+1}$ is non-essential, this implies that $A_{e_{l_1+1}}=1$. Hence $a_{l_1}$, $a_{l_2}$ and $t'$ fulfill (A5) with $p=a_{l_1}$ and $p'=a_{l_2}$.

In the remaining case that exactly one of $v_{l_1}$ and $v_{l_2}$ is essential it follows analogously that $a_{l_1}$ and $t'$ or $a_{l_2}$ and $t'^{-1}$ fulfill (A4(a)), if $l_2=l_1+1$ and $A_{e_{l_2}}\neq1$, and (A4(b)), otherwise.

\item If $v_{l_1}=v_{l_2}$ is not a central peripheral vertex, there exists some full $\A'$-subpath $t'$ of $t$ with underlying path $(e_{l_1+1},\ldots,e_{l_2})$.

If $v_{l_1}=v_{l_2}$ is a central peripheral vertex there is no full $\A'$-subpath of $t$ with underlying path $(e_{l_1+1},\ldots,e_{l_2})$. Instead, consider the following $\A'$-path
\[
t':=(1,e_{l_1+2},a_{l_1+2},\ldots,e_{l_2},a_{l_2}g_{e_{l_1}}^{-1}a_{l_2+1}a_{l_1-1}g_{e_{l_1}}a_{l_1},e_{l_1+1},a_{l_1+1}),
\]
which is equivalent to a cyclic permutation of the full $\A'$-subpath of $t$ with underlying path $(e_{l_1},\ldots,e_{l_2+1})$.

In both cases, $t'$ is an $\A'$-almost-circuit with at most one essential vertex or as described in 2.\ above. Moreover, it follows in both cases that $\abs{t'}_{X\cup\mathcal{P}}^\mathcal{A'}\leq\abs{t}_{X\cup\mathcal{P}}^\mathcal{A'}$.

Shorten $t'$ by the following procedure:

Let $s=(b_1,f_1,\ldots,f_{k'},b_{k'})$ an $\mathcal{A'}$-path and $l\in\{1,\ldots,k'-1\}$ with $f_l=f_{l+1}^{-1}$. Let
\[
s':=(b_1,f_1',\ldots,f_{l-1},b_{l-1}b_{l+1},f_{l+2},\ldots,f_{k'},b_{k'}),
\]
if $f_l$ is non-peripheral, and
\[
s':=(b_1,f_1,\ldots,f_{l-1},b_{l-1}g_{f_l}b_lg_{f_l}^{-1}b_{l+1},f_{l+2},\ldots,f_{k'},b_{k'}),
\]
if $f_l$ is peripheral. In both cases it follows that $\abs{s'}_X^\mathcal{A'}\leq\abs{s}_X^\mathcal{A'}$.

Repeating this process eventually yields an $\A'$-almost-circuit $t''$ with $\abs{t''}_X^\mathcal{A'}\leq\abs{t}_X^\mathcal{A'}\leq L$ and at most one essential vertex. Since $t''$ is an $\A'$-almost-circuit, $t''$ can be assumed as an $\A$-path rather than an $\A'$-path after a cyclic permutation.

If $t''$ has an essential vertex, then some cyclic permutation of $t''$ fulfills (A2). Otherwise, $t''$ must contain an edge with trivial edge group and fulfill (A3).
\end{enumerate}
\end{proof}
\end{lemma}

\begin{lemma}\label{lem:sublemma2}
Let $L\geq0$ and $\mathcal{A'}$ be some subdivision of $\mathcal{A}$ along at most two free edges. Let $t_i=(a_0^i,e_1,a_1^i,\ldots,a_{l-1}^i,e_l,1)$ be almost-simple $\A'$-paths for $i\in\{1,2\}$ (with the same underlying path $(e_1,\ldots, e_l)$), $v_0:=\alpha(e_1)$, $v_m:=\omega(e_m)$ for $m\in\{1,\ldots,l\}$ and $a\in A_{v_l}\setminus\omega_{e_l}(A_{e_l})$. If
\begin{enumerate}
\item $\abs{t_1(a)t_2^{-1}}_X^\mathcal{A'}\leq L$, and

\item $[t_2^{-1}t_1]\notin [\omega_{e_l}(A_{e_l})]:=\{[(a)]\mid a\in\omega_{e_l}(A_{e_l})\}$,
\end{enumerate}
then one of (A2), (A4) or (A5) must be true for the given $L$.

\begin{proof}
Suppose first  that  $v_m$ is essential for some $m\in\{0,\ldots,l-1\}$. 

Let $t:=(1,e_{m+1},a_{m+1}^1,\ldots,a_{l-1}^1,e_l,1)$. If $v_l$ is essential then $t$ satisfies (A2). Otherwise the same argument as in case 1.\ in the proof of Lemma~\ref{lem:sublemma1} yields that $t^{-1}$ and  $a_m$  fulfill  (A4). Therefore assume that $v_m$ is non-essential for all $m\in\{0,\ldots,l-1\}$.

Replace $t_i$ by an equivalent $\A'$-path such that $a_{m-1}^i=1$ for every $m\in\{1,\ldots,l\}$ with $e_m$ peripheral. This is possible as all boundary homomorphisms in peripheral stars are bijective. It is immediate, that these modifications preserve the assumptions on $t_1$  and $t_2$.


As $[t_2^{-1}t_1]\notin[\omega_{e_l}(A_{e_l})]$, there is some $m\in\{0,\ldots,l-1\}$ such that
\[
p:=(a_m^2)^{-1}a_m^1\notin\alpha_{e_{m+1}}(A_{e_{m+1}}).
\]
Choose $m$ minimal with this property. By the previous replacement of the $t_i$, $e_{m+1}$ is non-peripheral. 
Hence, $s_1:=(a_0^1,e_1,\ldots,e_m,a_m^1)$ and $s_2:=(a_0^2,e_1,\ldots,e_m,a_m^2)$ are full $\A'$-subpaths of $t_1$ and $t_2$, respectively.


Thus, it follows that
\begin{align*}
\abs{p}_X &= \abs{(a_m^2)^{-1}a_m^1}_X\\
&= \abs{\nu_\mathcal{A'}(s_2)^{-1}\nu_\mathcal{A'}(s_1)}_X\\
&\leq \abs{s_1}_X^\mathcal{A'}+\abs{s_2}_X^\mathcal{A'}\\
&\leq \abs{t_1(a)t_2^{-1}}_X^\mathcal{A'}\\
&\leq L.
\end{align*}

Note that $e_l$ is non-peripheral as $a\in A_{v_l}\setminus\omega_{e_l}(A_{e_l})$. Thus $s:=(1,e_{m+1},a_{m+1}^1,\ldots,a_{l-1}^1,e_l,1)$ is a full $\A'$-subpath of $t_1$. Since $v_m$ and $v_l$ have non-trivial vertex groups, $s$ can be assumed to be an $\A$-path rather than an $\A'$-path. Furthermore, $(a)$ is a full $\A'$-subpath of $t_1(a)t_2^{-1}$ and it follows that
\[
\abs{a}_X,\abs{s}_X^\mathcal{A}\leq\abs{t_1(a)t_2^{-1}}_X^\mathcal{A'}\leq L.
\]

If $v_l$ is essential and $l=m+1$, then $s$ and $p$ fulfill (A4(a)). If $v_l$ is essential and $l>m+1$ then $s$ must have an edge with trivial edge group and $s$ and $p$ fulfill (A4(b)). Otherwise, $v_l$ is peripheral and $A_{e_l}=1$, since $a\in A_{v_l}\setminus\omega_{e_l}(A_{e_l})$. Hence $s$, $p$ and $p':=a$ fulfill (A5).
\end{proof}
\end{lemma}

\begin{proof}[of Theorem~\ref{thm:main}]
Choose $C_0$ as in Lemma~\ref{lem:sublemma3} for $\delta=1$. Let $D=D(G,\PP,X,\nu)$ and $D_2:=D_2(G,\PP,X)$ be the constants from Lemma~\ref{lem:conjugation_length} and Lemma~\ref{lem:parabolic}, respectively, and
\[
C_1:=2C_0+\max\{D,D_2+1\}.
\]
Choose $C_2$ according to Lemma~\ref{lem:sublemma3} with $\delta=(3C_1+2)^{-1}$ and with 
\[
C_2\geq\max(4\nu,3C_1+2,C_1+6).
\]
Let $C:=C'(G,\PP,X,C_2)$ and $L:=L'(G,\PP,X,M,C_2)$ be the constants from Proposition~\ref{prop:1}.

Suppose now that assumption 1.\ of Proposition \ref{prop:1} is true, i.e.\ there is some reduced $\A$-path $t=(a_0,e_1,\ldots,e_k,a_k)$ with $\abs{t}_X^\mathcal{A}\leq L$, $t_-\in VC_i$, $t_+\in VC_j$ and $m_i=m_j$, which is not entirely contained in $C_i$ and whose label $\nu_\mathcal{A}(t)$ represents an element of $P_{m_i}$. Assume $t$ to be minimal in the sense that no proper $\A$-subpath of $t$ has the same properties. In particular $e_1\notin EC_i$, $e_k\notin EC_j$ and $a_0=a_k=1$.

Let $v_0:=\alpha(e_1)$ and $v_m:=\omega(e_m)$ for $m\in\{1,\ldots,k\}$. By Lemma~\ref{lem:sublemma1}, it can be assumed that one of the following cases occurs:

\textbf{Case 1:} $t$ is almost-simple and has at most one essential vertex.

Show that $t$ has an edge with trivial edge group and hence fulfills (A1(a)).

By the structure of $(G,\PP)$-carrier graphs and since at most one vertex of $t$ is essential, it is easy to see that $t$ must contain an edge with trivial edge group in which case we are done or else be of length $k=2$ with $v_1$, $e_1$ and $e_2$ essential. If $t$ is of this form, it follows that
\[
\omega_{e_1}(A_{e_1})=g_{e_1}^{-1}A_{e_1}g_{e_1}\subseteq g_{e_1}^{-1}P_{m_i}g_{e_1}
\]
and, since $g_{e_1}a_1g_{e_2}=\nu_\mathcal{A}(t)\in P_{m_1}$,
\begin{align*}
a_1\alpha_{e_2}(A_{e_2})a_1^{-1} &= a_1g_{e_2}A_{e_2}g_{e_2}^{-1}a_1^{-1}\\
&\subseteq a_1g_{e_2}P_{m_i}g_{e_2}^{-1}a_1^{-1}\\
&= a_1g_{e_2}t^{-1}P_{m_i}tg_{e_2}^{-1}a_1^{-1}\\
&= g_{e_1}^{-1}P_{m_i}g_{e_1}
\end{align*}
This is a contradiction to $\omega_{e_1}(A_{e_1})$ and $\alpha_{e_2}(A_{e_2})$ being part of the induced structure $\OO_{v_1}$ of $G$ on $A_{v_1}$, which is guaranteed by $\mathcal{A}$ being $M$-normal.

\textbf{Case 2a:} The underlying path of $t$ is of the form $rr^{-1}$ for an almost-simple path $r$ where all vertices of $rr^{-1}$ distinct from $r_+$ are non-essential.

This implies $k=2l$ for some $l$ and $e_m=e_{k-m+1}^{-1}$ for $m\in\{1,\ldots,l\}$. Let $t_1:=(a_0,e_1,\ldots,a_{l-1},e_l,1)$ and $t_2:=(a_k^{-1},e_1,\ldots,a_{l+1}^{-1},e_l,1)$. Then $t=t_1(a_l)t_2^{-1}$ and $t_1$ and $t_2$ are almost-simple.

By Lemma~\ref{lem:sublemma2}, it can be assumed that $a':=\nu_\mathcal{A}(t_2^{-1}t_1)\in\omega_{e_l}(A_{e_l})$. Now  $a_la'\notin\omega_{e_l}(A_{e_l})$ as $a_l\notin\omega_{e_l}(A_{e_l})$ and moreover:
\[
a_la'=a_l\nu_\mathcal{A}(t_2^{-1}t_1)=\nu_\mathcal{A}(t_1^{-1}t_1(a_l)t_2^{-1}t_1)=\nu_\mathcal{A}(t_1^{-1}tt_1)\in\nu_\mathcal{A}(t_1)^{-1}P_{m_i}\nu_\mathcal{A}(t_1)
\]

If $v_l$ is peripheral therefore  $a_la'\in A_{v_l}\subset P_j$ for some $j$, it follows from the malnormality of the family $\PP$, that $a_la',\nu_\mathcal{A}(t_1)\in P_{m_i}$. As in the first case, this implies that $t_1$ fulfills (A1(a)). Hence, assume that $v_l$ is essential.

If $t_1$ does not have an edge with trivial edge group, it is of length $1$ as no vertex of $t$ besides $v_l$ is essential. Hence
\[
a_la'\notin\omega_{e_l}(A_{e_l})=g_{e_l}^{-1}A_{e_l}g_{e_l}=\nu_\mathcal{A}(t_1)^{-1}P_{m_i}\nu_\mathcal{A}(t_1)\cap A_{v_l},
\]
a contradiction. Thus $t_1$ does have an edge with trivial edge group and therefore fulfills (A1(b)).

\textbf{Case 2b:} The underlying path of $t$ is of the form $r_1r_2(r_2)^{-1}r_3$ as described in Lemma~\ref{lem:sublemma1} 2.\ with $r_1r_3$ non-degenerate.

This implies that $r_1$ travels $e_1$ or that $r_3$ travels $e_k$. W.l.o.g.\ assume that $e_1$ is part of $r_1$. In particular, since $r_2$ is non-degenerate, $v_1$ cannot be essential. Since $e_1$ was assumed to be non-peripheral, this implies $A_{e_1}=1$ and thus $t$ fulfills (A1)a.

\bigskip

Now assume Proposition~\ref{prop:1}~2. is true. Let $\mathcal{A'}$ be a subdivision of $\mathcal{A}$ along at most two free edges and $t=(a_0,e_1,\ldots,e_k,a_k)$ a reduced $\A'$-path with $\abs{t}_X^\mathcal{A'}\leq L$, such that
\[
\abs{\nu_\mathcal{A'}(t)}_{X\cup\mathcal{P}}<\frac{\abs{t}_{X\cup\mathcal{P}}^\mathcal{A'}}{C_2}-C_2.
\]

Let $v_0:=\alpha(e_1)$ and $v_l:=\omega(e_l)$ for all $l\in\{1,\ldots,k\}$. By Lemma~\ref{lem:sublemma1}, it can be assumed that one of the following cases occurs:

\textbf{Case 1:} $t$ is almost-simple and has at most one essential vertex.

If $t$ has no essential vertex, it follows from $C_2\geq C_0$ and Lemma~\ref{lem:sublemma3} with $t':=t$ that $t$ fulfills (A6).

Hence assume that there is some $l\in\{0,\ldots,k\}$ such that $v_l$ is essential. Let $s=s_l(a_l)s_{l+1}$ where $s_i=(1,e_i,1)$ if $A_{e_i}\neq1$ (i.e. if $e_i$ is essential), and else $s_i=(1)$. Note that
\[
\abs{s}_{X\cup\mathcal{P}}^\mathcal{A'}\leq\abs{a_l}_{X\cup\mathcal{P}}+2\nu\leq\abs{\nu_\mathcal{A'}(s)}_{X\cup\mathcal{P}}+4\nu.
\]
Let $t_1$ and $t_2$ be such that $t=t_1st_2$. It follows that
\begin{align*}
\frac{\abs{t}_{X\cup\mathcal{P}}^\mathcal{A'}}{C_2}-C_2 &> \abs{\nu_\mathcal{A'}(t)}_{X\cup\mathcal{P}}\\
&\geq \abs{\nu_\mathcal{A'}(s)}_{X\cup\mathcal{P}}-(\abs{\nu_\mathcal{A'}(t_1)}_{X\cup\mathcal{P}}+\abs{\nu_\mathcal{A'}(t_2)}_{X\cup\mathcal{P}})\\
&\geq \abs{s}_{X\cup\mathcal{P}}^\mathcal{A'}-4\nu-(\abs{t_1}_{X\cup\mathcal{P}}^\mathcal{A'}+\abs{t_2}_{X\cup\mathcal{P}}^\mathcal{A'})\\
&= \abs{s}_{X\cup\mathcal{P}}^\mathcal{A'}-4\nu-(\abs{t}_{X\cup\mathcal{P}}^\mathcal{A'}-\abs{s}_{X\cup\mathcal{P}}^\mathcal{A'})\\
&= 2\abs{s}_{X\cup\mathcal{P}}^\mathcal{A'}-4\nu-\abs{t}_{X\cup\mathcal{P}}^\mathcal{A'}.
\end{align*}
As $4\nu\leq C_2$ and $3\leq C_2$, this implies
\[
\abs{s}_{X\cup\mathcal{P}}^\mathcal{A'}\leq\frac{1}{2}\left(1+\frac{1}{C_2}\right)\abs{t}_{X\cup\mathcal{P}}^\mathcal{A'}+\frac{1}{2}(4\nu-C_2)\leq\frac{2}{3}\abs{t}_{X\cup\mathcal{P}}^\mathcal{A'}.
\]

Thus, $\abs{t_i}_{X\cup\mathcal{P}}^\mathcal{A'}\geq\frac{1}{6}\abs{t}_{X\cup\mathcal{P}}^\mathcal{A'}$ for some $i\in\{1,2\}$. With $\frac{1}{6}\geq\frac{1}{3C_1+2}$ and $t':=t_i$, it follows from Lemma \ref{lem:sublemma3} that $t$ fulfills (A6).

\textbf{Case 2:} The underlying path of $t$ is of the form $r_1r_2(r_2)^{-1}r_3$ as described in Lemma~\ref{lem:sublemma1}~2.

Let $s$ be the maximal full $\A'$-subpath of $t$ such that the underlying path of $s$ is a subpath of $r_2(r_2)^{-1}$. Let $s'$ be the minimal full $\A'$-subpath of $t$ such that $r_2(r_2)^{-1}$ is a subpath of the underlying path of $s'$. Since $(r_2)_-$ is non-essential, this implies that
\[
\abs{s}_{X\cup\mathcal{P}}^\mathcal{A'}\leq\abs{s'}_{X\cup\mathcal{P}}^\mathcal{A'}\leq\abs{s}_{X\cup\mathcal{P}}^\mathcal{A'}+2.
\]
Let $s_1$ and $s_2$ be the full $\A'$-subpaths of $t$ such that $t=s_1s's_2$.

It follows from Lemma~\ref{lem:sublemma3}, that $t$ fulfills (A6), if
\[
\abs{s_1}_{X\cup\mathcal{P}}^\mathcal{A'}+\abs{s_2}_{X\cup\mathcal{P}}^\mathcal{A'}\geq\frac{2}{3C_1+2}\abs{t}_{X\cup\mathcal{P}}^\mathcal{A'},
\]
as $\delta=(3C_1+2)^{-1}$.
If on the other hand
\[
\abs{s_1}_{X\cup\mathcal{P}}^\mathcal{A'}+\abs{s_2}_{X\cup\mathcal{P}}^\mathcal{A'}\leq\frac{2}{3C_1+2}\abs{t}_{X\cup\mathcal{P}}^\mathcal{A'},
\]
then
\[
\abs{t}_{X\cup\mathcal{P}}^\mathcal{A'}\le\abs{s'}_{X\cup\mathcal{P}}^\mathcal{A'}+\abs{s_1}_{X\cup\mathcal{P}}^\mathcal{A'}+\abs{s_2}_{X\cup\mathcal{P}}^\mathcal{A'}\leq\abs{s'}_{X\cup\mathcal{P}}^\mathcal{A'}+\frac{2}{3C_1+2}\abs{t}_{X\cup\mathcal{P}}^\mathcal{A'}
\]
and therefore
\[
\abs{t}_{X\cup\mathcal{P}}^\mathcal{A'}\leq\left(1-\frac{2}{3C_1+2}\right)^{-1}\abs{s'}_{X\cup\mathcal{P}}^\mathcal{A'}\leq\left(1+\frac{2}{3C_1}\right)(\abs{s}_{X\cup\mathcal{P}}^\mathcal{A'}+2).
\]
It follows that
\begin{align*}
\abs{\nu_\mathcal{A'}(s)}_{X\cup\mathcal{P}} &\leq \abs{\nu_\mathcal{A'}(s')}_{X\cup\mathcal{P}}+2\\
&\leq \abs{\nu_\mathcal{A'}(t)}_{X\cup\mathcal{P}}+\abs{s_1}_{X\cup\mathcal{P}}^\mathcal{A'}+\abs{s_2}_{X\cup\mathcal{P}}^\mathcal{A'}+2\\
&= \abs{\nu_\mathcal{A'}(t)}_{X\cup\mathcal{P}}+\abs{t}_{X\cup\mathcal{P}}^\mathcal{A'}-\abs{s'}_{X\cup\mathcal{P}}^\mathcal{A'}+2\\
&< \left(1+\frac{1}{C_2}\right)\abs{t}_{X\cup\mathcal{P}}^\mathcal{A'}-\abs{s}_{X\cup\mathcal{P}}^\mathcal{A'}-C_2+2\\
&\leq \left(1+\frac{1}{C_2}\right)\left(1+\frac{2}{3C_1}\right)(\abs{s}_{X\cup\mathcal{P}}^\mathcal{A'}+2)-\abs{s}_{X\cup\mathcal{P}}^\mathcal{A'}-(C_2-2)\\
&\leq \left(1+\frac{1}{3C_1+2}\right)\left(1+\frac{2}{3C_1}\right)(\abs{s}_{X\cup\mathcal{P}}^\mathcal{A'}+2)-\abs{s}_{X\cup\mathcal{P}}^\mathcal{A'}-(C_1+4)\\
&= \left(1+\frac{1}{C_1}\right)(\abs{s}_{X\cup\mathcal{P}}^\mathcal{A'}+2)-\abs{s}_{X\cup\mathcal{P}}^\mathcal{A'}-(C_1+4)\\
&= \frac{\abs{s}_{X\cup\mathcal{P}}^\mathcal{A'}}{C_1}-\left(C_1+4-2\left(1+\frac{1}{C_1}\right)\right)\\
&\leq \frac{\abs{s}_{X\cup\mathcal{P}}^\mathcal{A'}}{C_1}-C_1.
\end{align*}

Let $l\in\{0,\ldots,k\}$ with $v_l=(r_2)_+$. Let $t_1$ and $t_2$ be $\A'$-paths whose underlying path is a subpath of $r_2$ such that $s=t_1(a_l)t_2^{-1}$. Because of  Lemma \ref{lem:sublemma2} we may assume that $\nu_\mathcal{A'}(t_2^{-1}t_1)\in A_{v_l}$.

\textbf{Case 2a:} $v_l$ is essential.

Let $g\in\nu_\mathcal{A'}(t_1)A_{v_l}=\nu_\mathcal{A'}(t_2)A_{v_l}$ of minimal $X\cup\mathcal{P}$-length in $gA_{v_l}$ and $h_i\in A_{v_l}$ such that $g=\nu_\mathcal{A'}(t_i)h_i$ for $i\in\{1,2\}$. Since $A_{v_l}$ is $\nu$-relatively quasiconvex in $(G,\PP)$, Lemma~\ref{lem:conjugation_length} implies:
\begin{align*}
\frac{\abs{s}_{X\cup\mathcal{P}}^\mathcal{A'}}{C_1}-C_1 &> \abs{\nu_\mathcal{A'}(s)}_{X\cup\mathcal{P}}\\
&= \abs{\nu_\mathcal{A'}(t_1)a_l\nu_\mathcal{A'}(t_2)^{-1}}_{X\cup\mathcal{P}}\\
&= \abs{gh_1^{-1}a_lh_2g^{-1}}_{X\cup\mathcal{P}}\\
&\geq \abs{h_1^{-1}a_lh_2}_{X\cup\mathcal{P}}+2\abs{g}_{X\cup\mathcal{P}}-D
\end{align*}
Moreover:
\begin{align*}
\abs{s}_{X\cup\mathcal{P}}^\mathcal{A'} &= \abs{t_1}_{X\cup\mathcal{P}}^\mathcal{A'}+\abs{a_l}_{X\cup\mathcal{P}}+\abs{t_2}_{X\cup\mathcal{P}}^\mathcal{A'}\\
&\leq \abs{h_1^{-1}a_lh_2}_{X\cup\mathcal{P}}+\abs{t_1}_{X\cup\mathcal{P}}^\mathcal{A'}+\abs{h_1}_{X\cup\mathcal{P}}+\abs{t_2}_{X\cup\mathcal{P}}^\mathcal{A'}+\abs{h_2}_{X\cup\mathcal{P}}\\
&= \abs{h_1^{-1}a_lh_2}_{X\cup\mathcal{P}}+\abs{t_1(h_1)}_{X\cup\mathcal{P}}^\mathcal{A'}+\abs{t_2(h_2)}_{X\cup\mathcal{P}}^\mathcal{A'}
\end{align*}
It follows:
\begin{align*}
2\abs{g}_{X\cup\mathcal{P}} &< \left(\frac{1}{C_1}-1\right)\abs{h_1^{-1}a_lh_2}_{X\cup\mathcal{P}}+\frac{\abs{t_1(h_1)}_{X\cup\mathcal{P}}^\mathcal{A'}+\abs{t_2(h_2)}_{X\cup\mathcal{P}}^\mathcal{A'}}{C_1}-C_1+D\\
&\leq \frac{\abs{t_1(h_1)}_{X\cup\mathcal{P}}^\mathcal{A'}+\abs{t_2(h_2)}_{X\cup\mathcal{P}}^\mathcal{A'}}{C_0}-2C_0
\end{align*}
Hence, for some $i\in\{1,2\}$:
\[
\abs{\nu_\mathcal{A'}(t_i(h_i))}_{X\cup\mathcal{P}}=\abs{g}_{X\cup\mathcal{P}}<\frac{\abs{t_i(h_1)}_{X\cup\mathcal{P}}^\mathcal{A'}}{C_0}-C_0
\]
Thus, by Lemma~\ref{lem:sublemma3}, $t_i$ contains some trivial segment of length greater than $\abs{\nu_\mathcal{A'}(t_i)h_i}_{X\cup\mathcal{P}}+2$, which $t_i$ only passes through once. Since $t_i$ is a full $\A'$-subpath of $t$, $\abs{t_i}_X^\mathcal{A'}\leq\abs{t}_X^\mathcal{A'}\leq L$ and $t_i$ fulfills (A6).

\textbf{Case 2b:} $v_l\in VC_j$ is peripheral.

Let $g\in\nu_\mathcal{A'}(t_1)P_{m_j}=\nu_\mathcal{A'}(t_2)P_{m_j}$ of minimal $X\cup\mathcal{P}$-length in $gP_{m_j}$ and $h_i\in P_{m_j}$, such that $g=\nu_\mathcal{A'}(t_i)h_i$ for $i\in\{1,2\}$.  Lemma~\ref{lem:parabolic} implies:

\begin{align*}
\frac{\abs{s}_{X\cup\mathcal{P}}^\mathcal{A'}}{C_1}-C_1 &> \abs{\nu_\mathcal{A'}(s)}_{X\cup\mathcal{P}}\\
&= \abs{\nu_\mathcal{A'}(t_1)a_l\nu_\mathcal{A'}(t_2)^{-1}}_{X\cup\mathcal{P}}\\
&= \abs{gh_1^{-1}a_lh_2g^{-1}}_{X\cup\mathcal{P}}\\
&\geq 2\abs{g}_{X\cup\mathcal{P}}-D_2
\end{align*}
Moreover, since $e_l$ is non-peripheral:
\begin{align*}
\abs{s}_{X\cup\mathcal{P}}^\mathcal{A'} &= \abs{t_1}_{X\cup\mathcal{P}}^\mathcal{A'}+\abs{t_2}_{X\cup\mathcal{P}}^\mathcal{A'}+1\\
&\leq \abs{t_1(h_1)}_{X\cup\mathcal{P}}^\mathcal{A'}+\abs{t_2(h_2)}_{X\cup\mathcal{P}}^\mathcal{A'}+1
\end{align*}
It follows:
\begin{align*}
2\abs{g}_{X\cup\mathcal{P}} &< \frac{\abs{t_1(h_1)}_{X\cup\mathcal{P}}^\mathcal{A'}+\abs{t_2(h_2)}_{X\cup\mathcal{P}}^\mathcal{A'}+1}{C_1}-C_1+D_2\\
&\leq \frac{\abs{t_1(h_1)}_{X\cup\mathcal{P}}^\mathcal{A'}+\abs{t_2(h_2)}_{X\cup\mathcal{P}}^\mathcal{A'}}{C_1}-C_1+D_2+1\\
&\leq \frac{\abs{t_1(h_1)}_{X\cup\mathcal{P}}^\mathcal{A'}+\abs{t_2(h_2)}_{X\cup\mathcal{P}}^\mathcal{A'}}{C_0}-2C_0
\end{align*}
Hence, for some $i\in\{1,2\}$:
\[
\abs{\nu_\mathcal{A'}(t_i(h_i))}_{X\cup\mathcal{P}}=\abs{g}_{X\cup\mathcal{P}}<\frac{\abs{t_i(h_1)}_{X\cup\mathcal{P}}^\mathcal{A'}}{C_0}-C_0
\]
Thus, by Lemma~\ref{lem:sublemma3}, $t_i$ contains some trivial segment of length greater than $$\abs{\nu_\mathcal{A'}(t_i)h_i}_{X\cup\mathcal{P}}+6\ge \abs{\nu_\mathcal{A'}(t_i)}_{X\cup\mathcal{P}}+5,$$ which $t_i$ only passes through once. Since $t_i$ is a full $\A'$-subpath of $t$, $\abs{t_i}_X^\mathcal{A'}\leq\abs{t}_X^\mathcal{A'}\leq L$ and $t_i$ fulfills (A6).
\end{proof}

Combining Lemma~\ref{lem:normalization} and Theorem~\ref{thm:main} we obtain the following:

\begin{corollary}\label{cor:main}
Let $G$ be a group, which is torsion-free, hyperbolic relative to $\PP=\{P_1,\ldots,P_n\}$ and has finite symmetric generating set $X$. Let $(\mathcal A,((\mathcal C_i,c_i))_{1\le i\le k})$ be an $M$-prenormal $(G,\PP)$-carrier graph.

There are $C=C(G,\PP,X,M,N(\mathcal A))$ and $L=L(G,\PP,X,M,N(\mathcal A))$ such that the following holds:

$\nu_{\mathcal{A}}\colon(\pi_1(\A,v_0),d_{X\cup\mathcal{P}}^\mathcal{A})\to(G,d_{X\cup\mathcal{P}})$ is a $(C,C)$-quasiisometric embedding and for any peripheral structure $\OO_\mathcal{A}$ of $\mathcal{A}$ is $\nu_{\mathcal{A}}(\OO_\mathcal{A})$ an induced structure of $(G,\PP)$ on $\Ima\nu_\mathcal{A}$, unless one of the following holds:

\begin{enumerate}

\item[(A1)]\label{item:CCT1} \begin{enumerate}
\item There is some $\A$-path $t$ with $\abs{t}_X^\mathcal{A}\leq L$, $t_-\in VC_i$, $t_+\in VC_j$, $m_i=m_j$ and $\nu_\mathcal{A}(t)\in P_{m_i}$, which passes through some trivial segment of $\A$ exactly once.

\item There is some $\A$-path $t$ with $\abs{t}_X^\mathcal{A}\leq L$, $t_-\in VC_i$, $t_+$ essential, and $A_{t_+}\cap\nu_\mathcal{A}(t)^{-1}P_{m_i}\nu_\mathcal{A}(t)\neq1$, which passes through some trivial segment of $\A$ exactly once.
\end{enumerate}

\item[(A2)]\label{item:CCT2} There are essential vertices $v,v'\in VA$ and some non-degenerate, almost-simple $\A$-path $t$ starting in $v$ and ending in $v'$ with $\abs{t}_X^\mathcal{A}\leq L$.

\item[(A3)]\label{item:CCT3} There is an $\A$-almost-circuit $t$ with $\abs{t}_X^\mathcal{A}\leq L$ which has a trivial segment.

\item[(A4)]\label{item:CCT4}
\begin{enumerate}
\item There is some essential edge $e$ with $\alpha(e)\in VC_i$ and some $p\in A_{\alpha(e)}\setminus\alpha_e(A_e)$ with
\[
\abs{p}_X\leq L.
\]

\item There is some almost-simple $\A$-path $t=(a_0,e_1,a_1,\ldots,e_k,a_k)$ with $\abs{t}_X^\mathcal{A}\leq L$, $t_-\in VC_i$, $A_{e_1}=1$ and $t_+$ essential, and some $p\in A_{t_-}\setminus \{1\}$ such that
\[
\abs{p}_X\leq L.
\]

\end{enumerate}

\item[(A5)]\label{item:CCT5} There is some almost-simple $\A$-path $t$ with $\abs{t}_X^\mathcal{A}\leq L$, $t_-\in VC_i$, $t_+\in VC_j$ and passing through some trivial segment of $\A$, and elements $p\in A_{t_-}\setminus\{1\},p'\in A_{t_+}\setminus\{1\}$, such that
\[
\abs{p}_X,\abs{p'}_X\leq L.
\]


\item[(A6)]\label{item:CCT6} There is some subdivision $\mathcal{A}'$ of $\mathcal{A}$ along at most two non-pe\-ri\-pheral, non-essential edges, an $\A'$-path $t$ with $\abs{t}_X^\mathcal{A'}\leq L$, some trivial segment $t'$ of $t$ and some $a\in A_{t_+}$ with
\begin{enumerate}
\item $t$ passes exactly once through $t'$ or $t'^{-1}$,
\item $\abs{\nu_\mathcal{A'}(t)a}_{X\cup\mathcal{P}}+5<\abs{t'}_{X\cup\mathcal{P}}^\mathcal{A'}$.
\end{enumerate}
\item[(A7)]\label{item:CCT7} There is a circuit $(e_1,\ldots,e_l)$ in $A$ and some $m\in\{1,\ldots,n\}$, such that
		\begin{enumerate}
			\item $g_{e_j}\in P_m$ for all $j\in\{1,\ldots,l\}$, and
			
			\item $A_{e_i}=1$ for some $i\in\{1,\ldots,l\}$.
		\end{enumerate}
		
\item[(A8)]\label{item:CCT8} 
		 There exists some free  edge $e\in EA$ such that some geodesic word representing $g_e$ contains a letter from $\mathcal P$.

\end{enumerate}
\end{corollary}

\begin{proof} The only point worth mentioning is that 1. of Lemma~\ref{lem:normalization} is subsumed in (A2) and that 4. of Lemma~\ref{lem:normalization} is subsumed in (A1(a)).
\end{proof}

\section{Folds and applications} \label{Chapter:Folds}

The purpose of this chapter is two-fold. Firstly a version of Stallings folds are introduced that can be applied to $(G,\PP)$-carrier graphs of groups that do not induce a quasi-isometric embedding on the group level and secondly those folds are used to prove finiteness results for locally quasiconvex subgroups of a  given torsion free relatively hyperbolic group with Noetherian parabolic subgroups. 

\subsection{Equivalence classes of carrier graphs of groups}

To formulate the finiteness results in Section~\ref{finiteness_results} we need the appropriate equivalence classes of carrier graphs of groups.

The following trivial facts are used implicitly in the subsequent discussion:

\begin{enumerate}
\item If $e$ is an essential edge with $\alpha(e)\in C_i$ then $\alpha_e(A_e)$ is a subgroup of $A_{\alpha(e)}$ and therefore represents a conjugacy class of subgroups of $A_{c_i}$.
\item If in addition all peripheral subgroups are Abelian then $\alpha_e(A_e)$ is a subgroup of $A_{\alpha(e)}=A_{c_i}$.
\end{enumerate}

\begin{definition}\label{def:equivalence}
Let $(G,\PP)$ be a torsion-free relatively hyperbolic group.
Two $(G,\PP)$-carrier graphs without trivial edge groups $(\mathcal{A},((\mathcal{C}_i,c_i))_{1\leq i\leq k})$ and $(\mathcal{B},((\mathcal{C'}_i,c'_i))_{1\leq i\leq k})$ are called \emph{equivalent}  if there is a graph isomorphism $f\colon A\to B$  such that the following hold:
\begin{enumerate}
\item $f(C_i)=C'_i$ and $f(c_i)=c_i'$ for $1\le i\le k$.
\item $B_{f(x)}=A_x$ for all non-peripheral $x\in VA\cup EA$, and

\item $g_{f(e)}=g_e$ for all non-peripheral $e\in EA$.
\end{enumerate}

If in addition for any $i\in\{1,\ldots ,k\}$ there exists an isomorphism $\eta_i:A_{c_i}\to B_{c_i'}$ such that $\eta_i(\alpha_e(A_e))$ is conjugate 
 to $\alpha_{f(e)}(B_{f(e)})$ for any essential edge $e$ with $\alpha(e)\in VC_i$ then $(\mathcal{A},((\mathcal{C}_i,c_i))_{1\leq i\leq k})$ and $(\mathcal{B},((\mathcal{C'}_i,c'_i))_{1\leq i\leq k})$ are called strongly equivalent.
\end{definition}


\begin{definition} Let $\mathcal{A},\mathcal{B}$ be $(G,\PP)$-carrier graphs and let $\mathcal{A}_1,\ldots,\mathcal{A}_k$ and $\mathcal{B}_1,\ldots,\mathcal{B}_l$ be the components of $\mathcal{A}$ and $\mathcal{B}$ respectively, which arise from deleting all edges and vertices with trivial edge or vertex group.

Then $\mathcal{A}$ and $\mathcal{B}$ are called \emph{(strongly) equivalent}, if
\begin{enumerate}
\item $k=l$,
\item there is a bijection $f\colon\{1,\ldots,k\}\to\{1,\ldots,k\}$ such that $\mathcal{A}_i$ is (strongly) equivalent to $\mathcal{B}_{f(i)}$ for all $i\in\{1,\ldots,k\}$, and
\item $b_1(A)=b_1(B)$.
\end{enumerate}
\end{definition}

\begin{remark}
Lemma~\ref{lem:N(A)} gives an upper bound for the number $N(\mathcal A)$ of maximal trivial segments of $\A$, which only depends on the equivalence class of $\mathcal{A}$.
\end{remark}

\begin{remark}\label{stronimpliesisomorphic} It is immediate that the fundamental groups of the graphs of groups underlying strongly equivalent $(G,\PP)$-carrier graphs of groups are isomorphic.
\end{remark}



We record the following consequence of the classification of finitely generated Abelian groups.

\begin{lemma+definition}\label{lemma:envelope} Let $A$ be a finitely generated free Abelian group and $H\le A$ a subgroup. Then there exists a unique (free Abelian) subgroup $K\le A$ such that the following hold:
\begin{enumerate}
\item $H\subset K$ and $|K:H|<\infty$.
\item $A=K\oplus
 N$ for some (free Abelian) subgroup $N$ of $A$.
\end{enumerate}
$K$ is closed under taking roots and we refer to $K$ as the root closure of $H$ in $A$. Any homomorphism $\varphi:K\to \mathbb Z^n$ is determined by $\varphi|_H$.
\end{lemma+definition}

\begin{proof} The group $K$ consists of all elements $a\in A$ such that $a^n\in H$ for some $n\neq 0$. As $\varphi(a)^n=\varphi(a^n)$ and as $n$-{th}  roots are unique in free Abelian groups, the claim follows.
\end{proof}

Let $(G,\PP)$ be a torsion-free toral relatively hyperbolic group. For any $(G,\PP)$-carrier graph of groups $(\mathcal{A},((\mathcal{C}_i,c_i))_{1\leq i\leq k})$ and $i\in\{1,\ldots ,k\}$ denote by $\bar A_{c_i}$ the root closure of $$\tilde A_{c_i}=\langle \{\alpha_e(A_e)\mid e\hbox{ essential with }\alpha(e)\in C_i\}\rangle$$ in $A_{c_i}$. Thus for any $i$ there exists a free Abelian group $N_i$ such that $A_{c_i}=\bar A_{c_i}\oplus N_i$. Let now $\bar{\mathbb A}$ the graph of groups obtained from $\mathbb A$ by replacing each $A_{c_i}$ by $\bar A_{c_i}$. Clearly $$\pi_1(\mathbb A)/\langle\langle \pi_1(\bar{\mathbb A})\rangle\rangle\cong\underset{1\le i\le k}{\ast}N_i,$$ in particular $\sum_{i=1}^k \hbox{rank}(N_i)$ is bounded from above by the rank of $\pi_1(\mathbb A)$ by Grushko's theorem.

\begin{lemma}\label{lemma:toral}Let $(G,\PP)$ be a finitely generated torsion-free toral relatively hyperbolic group. 

\begin{enumerate}

\item\label{lemma:toral1} Let $(\mathcal{A},((\mathcal{C}_i,c_i))_{1\leq i\leq k})$ and $(\mathcal{B},((\mathcal{C'}_i,c'_i))_{1\leq i\leq k})$ be strongly equivalent $(G,\PP)$-carrier graphs of groups. Let $f$ and $(\eta_i)_{1\le i\le k}$ be as in Definition~\ref{def:equivalence}.

 Then $\bar A_{c_i}=\bar B_{c_i'}$ and $\eta_i|_{\bar A_{c_i}}=\hbox{id}|_{\bar A_{c_i}}$.
\item\label{lemma:toral2} Let $(\mathcal{A},((\mathcal{C}_i,c_i))_{1\leq i\leq k})$ be an equivalence class and $n\in\mathbb N$. Then there are only finitely many strong equivalence classes with $n$-generated fundamental group contained in the equivalence class of  $(\mathcal{A},((\mathcal{C}_i,c_i))_{1\leq i\leq k})$.
\end{enumerate}
\end{lemma}

\begin{proof} (1) For any essential edge $e$ with $\alpha(e)\in C_i$ the definition of equivalence implies that $\eta_i(\alpha_e(A_e))\subset P_{m_i}$ is in $G$ conjugate to $\alpha_{f(e)}(B_{f(e)})\subset P_{m_i}$. As $P_{m_i}$ is Abelian and malnormal in $G$ it follows that $\eta_i(\alpha_e(A_e))=\alpha_{f(e)}(B_{f(e)})$. Thus $\tilde A_{c_i}=\tilde B_{c_i'}$ and therefore $\bar A_{c_i}=\bar B_{c_i'}$ by Lemma~\ref{lemma:envelope}.

Clearly $\eta_i|_{\tilde A_{c_i}}=\hbox{id}|_{\tilde A_{c_i}}$. The last claim now follows as $\eta_i|_{\bar A_{c_i}}$ is uniquely determined by $\eta_i|_{\tilde A_{c_i}}$ by Lemma~\ref{lemma:envelope}.

(2) For any $i$ the group $\tilde A_{c_i}$ is determined by the equivalence class and there are only finitely many finite index overgroups of $\tilde A_{c_i}$ in $P_i$ and therefore only finitely many possibilities for $\bar A_{c_i}$. Moreover there are only finitely many possibilities for the $N_i$ as the $N_i$ are free Abelian and as $\sum_{i=1}^k \hbox{rank}(N_i)\le n$.
\end{proof}

\subsection{Folds and AO-Moves}

If the conclusion of Corollary~\ref{cor:main} does not hold, i.e.\ if at least one of the cases (A1)-(A8) occurs, then the carrier graph can be altered in some controlled way as discussed below. In cases (A1)-(A5) and (A7) these modification can be thought of (up to some preprocessing) as folds of graphs of groups as introduced by Bestvina and Feighn \cite{Bestvina1991} and Dunwoody \cite{Dunwoody1998}, see also \cite{Kapovich2005}. In case (A6) a variation of a move introduced by Arzhantseva and Olshanskii in the context of graphs can be applied, see \cite{Arzhantseva1996}. In case (A8) the modification is merely cosmetic. 

These alterations are designed in such a way that they simplify the carrier graph of groups by decreasing the values of certain invariants. Among them is the following complexity associated to a $(G,\PP)$-carrier graph $(\mathcal{A},((\mathcal{C}_i,c_i))_{1\leq i\leq k})$:

$$c(\mathcal A):=\sum_{\substack{e\in E_0\\ \ e\ non-peripheral}}|g_e|_{X\cup\mathcal P}+\sum_{1\le i\le k}\hbox{val}(c_i)$$

\medskip
In the remainder of this Section folds and other moves are introduced that modify a $(G,\PP)$-carrier graph $\mathcal A$ if one of the situations (A1)-(A8) spelled out in Corollary~\ref{cor:main} occurs. These modifications preserve the subgroup of $G$ that is represented, i.e. the image of the map $\nu_{\mathcal A}:\pi_1(\mathbb A,v_0)\to G$. The following is always assumed without explicit mentioning:

\begin{enumerate}
\item Whenever peripheral stars are introduced or modified, only the vertex groups of a single vertex (usually the central vertex) and the edge elements are specified, as everything else is then implicit.
\item Whenever a trivial segement is deleted this can also affect trivial peripheral stars in the sense that some edges of the star get deleted. If all (at most two) edges of a trivial peripheral star are deleted then the whole star is deleted.
\end{enumerate}

Let $L$ be the constant from the conclusion of Corollary~\ref{cor:main}. We may assume that $L\ge \nu(G,\PP,X,M)$ and therefore $L>|g_e|_X$ for all essential edges $e\in EA$. An analysis of the proof of Corollary~\ref{cor:main} actually implies that this also follows from the choices of the various constants.

\begin{enumerate}
\item[(A1)] \begin{enumerate} 
\item This case is very similiar to step (3) in the prenormalization procedure in Section~\ref{Sec:GP}. Let $s$ be the trivial segment of $\A$ which $t$ passes through exactly once. Let $g:=g_1\nu_\mathcal{A}(t)g_2^{-1}$, where $g_1$ is the label of the reduced path in $C_i$ from $c_i$ to $t_-$ and $g_2$ is the label of the reduced path from $c_j$ to $t_+$ in $C_j$. Clearly $g\in P_{m_i}$

If $i=j$: 
 Delete $s$ and replace $A_{c_i}$ by $\langle A_{c_i},g\rangle$. 

If $i\neq j$: Combine $\mathcal C_i$ and $\mathcal C_j$ as follows: 
Delete $s$ and $c_j$. Replace $A_{c_i}$ by $\langle A_{c_i},gA_{c_j}g^{-1}\rangle$, and every edge $e\in EC_j$ with $\alpha(e)=c_j$ by an edge $e'$ with $\alpha(e')=c_i$, $\omega(e')=\omega(e)$, $A_{e'}:=gA_eg^{-1}$ and $g_{e'}:=gg_e$. 

\item Let $a\in A_{t_+}\cap\nu_\mathcal{A}(t)^{-1}P_{m_i}\nu_\mathcal{A}(t)\setminus\{1\}$. Since $\mathcal{A}$ is $M$-normal, there is some $a'\in A_{t_+}$, some essential edge $e\in EA$ with $\alpha(e)=t_+$ and $\omega(e)\in VC_j$, and $p\in A_e\setminus\{1\}$ such that $a=a'\alpha_e(p)a'^{-1}$. As $$(a'g_e)p(a'g_e)^{-1}=a'(g_epg_e^{-1})(a')^{-1}=a'\alpha_e(p)a'^{-1}=a\in\nu_\mathcal{A}(t)^{-1}P_{m_i}\nu_\mathcal{A}(t),$$ it follows that $m_i=m_j$ and $\nu_\mathcal{A}(t)a'g_e\in P_{m_i}$.  Proceed as in (a) with the $\A$-path $t(a',e,1)$.
\end{enumerate}

Conclude by performing the prenormalization process. In this case only the last step, i.e. the cleaning up of the peripheral stars is necessary.

Upshot: The new carrier graph $\mathcal A'$ is again $M$-prenormal. Moreover one of the following holds:
\begin{itemize}
\item The relative rank decreases by one. This happens if $i=j$ or if $\mathcal C_i$ and $\mathcal C_j$ are non-trivial peripheral stars in the same free factor.
\item The relative rank is preserved and the number of free factors decreases by one. This happens if $\mathcal C_i$ and $\mathcal C_j$ are non-trivial peripheral stars of different free factors.
\item $i\neq j$ and $\mathcal C_i$ or $\mathcal C_j$ is trivial. In this case the number of (trivial) peripheral stars decreases by one and $\mathcal A'$ is strongly equivalent to $\mathcal A$. Moreover $c(\mathcal A')\le c(\mathcal A)$ and $c(\mathcal A')= c(\mathcal A)$ if and only if the edges of the deleted trivial segement all had  trivial labels and if the cleaning up of peripheral stars does not fold any edges.
\end{itemize}


\item[(A2)] Suppose $t$ is such a path of minimal length. If $t$ has an edge with trivial edge group, let $e$ be such an edge. Otherwise $t$ is a path in some free factor of $\mathbb A$ whose initial and terminal edges are essential connecting $v$ and $v'$ to the same non-trivial peripheral star. Let $e$ be initial essential edge of $t$.

Suppose $v=v'$. Delete $e$ and replace $A_v$ with $\langle A_v,\nu_\mathcal{A}(t)\rangle$.

Suppose now $v\neq v'$. Delete $e$ and replace $A_v$ by $\langle A_v,\nu_\mathcal{A}(t)A_{v'}\nu_\mathcal{A}(t)^{-1}\rangle$, replace every edge $e'$ with $\alpha(e')=v'$ by an edge $e$ with $\alpha(e)=v$, $\omega(e)=\omega(e')$, $A_e=A_{e'}$ and $g_e=\nu_\mathcal{A}(t)g_e$. Delete $v'$.

Note that the vertex group of every essential vertex is generated by elements of length at most $M+2L$. Conclude by performing the prenormalization process.

 Upshot: The new carrier graph is $(M+2L)$-normal. Moreover one of the following holds. 
\begin{itemize}
\item If the edge $e$ has trivial group then either the relative rank decreases by one (if $v=v'$) or the relative rank is preserved and the number of free factors decreases by one. 
\item If $e$ is essential then the number of essential edges in the core of the only affected free factor decreases by at least one. Note that the prenormalization process might add new essential edges, however these added essential edges will not lie in the core of affected free factor.
\end{itemize}
\color{black}

\item[(A3)] W.l.o.g.\ one may assume that $A_e=1$ for the first edge $e$ of $t$  and that $\omega(e)$ is a central peripheral vertex if $e$ is peripheral\color{black}. If $e$ is non-peripheral and therefore free, add a new vertex $v$ with $A_v=\langle\nu_\mathcal{A}(t)\rangle$, replace $e$ by an edge $e'$ with $\alpha(e')=v$, $\omega(e')=\omega(e)$, $g_e'=g_e$, $A_{e'}=1$.

If $e$ is peripheral, add two new vertices $v,v'$ with $A_v=\langle\nu_\mathcal{A}(t)\rangle$ and $A_{v'}=1$, and replace $e$ by two edges $e',e''$ with $\alpha(e')=v$, $\omega(e')=\alpha(e'')=v'$, $\omega(e'')=\omega(e)$, $g_e'=1$, $g_{e''}=g_e$, $A_{e'}=A_{e''}=1$.

In both cases the vertex $v$ is essential with no adjacent essential edge. Every essential vertex group of the resulting $(G,\PP)$-carrier graph is generated by elements of length at most $\max(L,M)$. Conclude by performing the prenormalization process.

Upshot: The resulting $(G,\PP)$-carrier graph is $\max(L,M)$-prenormal. The relative rank decreases by one and the number of free factors increases by one. \color{black}

\item[(A4)]
\begin{enumerate}
\item Replace $A_e$ by $\langle A_e,p\rangle$ and $A_{\omega(e)}$ by $\langle A_{\omega(e)},g_e^{-1}pg_e\rangle$.

\item Note that $e_1\notin EC_i$ as $C_i$ is non-trivial and $A_{e_1}=1$. Replace the maximal trivial segment containing  $e_1$ by an edge $e$ with $\alpha(e)=\alpha(e_1)$, $\omega(e)=t_+$, $g_e=\nu_\mathcal{A}(t)$ and $A_e=\langle p\rangle$, and replace $A_{t_+}$ by $\langle A_{t_+},\nu_\mathcal{A}(t)^{-1}p\nu_\mathcal{A}(t)\rangle$. The new edge $e$ is essential.
\end{enumerate}

Every essential vertex group of the resulting $(G,\PP)$-carrier graph is generated by elements of length at most $\max(3L,M)$. Conclude by performing the prenormalization process.

 Upshot: The resulting $(G,\PP)$-carrier graph is $\max(3L,M)$-prenormal. Moreover one of the following holds:
\begin{itemize}
\item At least one essential edge group in the core of some free factor replaced by a proper overgroup, this happens in (a). In the prenormalization process some edges might be folded.
\item 
Either the relative rank or the number of free factors decreases by one.  This happens in (b).
\end{itemize}
\color{black}

\item[(A5)] Delete the trivial segment of $t$. Add edges $e'$ and $e''$ as well as a vertex $v$ with $\alpha(e')=t_-$, $\omega(e')=\alpha(e'')=v$, $\omega(e'')=t_+$, $A_v=\langle p,\nu_\mathcal{A}(t)p'\nu_\mathcal{A}(t)^{-1}\rangle$, $A_{e'}=\langle p\rangle$, $A_{e''}=\langle p'\rangle$, $g_{e'}=1$ and $g_{e''}=\nu_\mathcal{A}(t)$. The new vertex and edges are essential.

 If $t_-$, respectively  $t_+$, is the central vertex of $C_i$, respectively $C_j$, introduce a peripheral edge with trivial edge element between $t_-$ and $\alpha(e')$, respectively $t_+$ and $\omega(e'')$, to ensure that every edge adjacent to the central vertices are peripheral.

All essential vertex groups of the resulting $(G,\PP)$-carrier graph are generated by elements of length at most $\max(3L,M)$. Conclude by performing the prenormalization process.

Upshot:  The new $(G,\PP)$-carrier graph is $\max(3L,M)$-prenormal. The relative rank or the number of free factors decreases by one, two new essential edges and one essential vertex occur.\color{black}

\item[(A6)] 
Replace $t(a)$ by a full $\mathbb A$-path $\hat t$ by adding another edge in the beginning (or end) if $t_-$ (or $t_+$) is a central peripheral vertex. Note that $$\abs{\nu_\mathcal{A'}(t)a}_{X\cup\mathcal{P}}+2\ge \abs{\nu_\mathcal{A'}(\hat t)}_{X\cup\mathcal{P}}$$ and therefore  $$\abs{\nu_\mathcal{A'}(\hat t)}_{X\cup\mathcal{P}}+3<\abs{t'}_{X\cup\mathcal{P}}^\mathcal{A'}.$$ Replace $\mathcal{A}$ by $\mathcal{A}'$ in the following way: Delete $t'$ and add a new free edge $e$ with $\alpha(e)=t_-$, $\omega(e)=t_+$, $g_e=\nu_\mathcal{A}(\hat t)$ and $A_e=1$. This strictly decreases the complexity of $\mathcal A$.

Upshot: The relative rank and all free factors are preserved, thus the new carrier graph is again $M$-prenormal and strongly equivalent to $\mathcal A$. Moreover $c(\mathcal A')<c(\mathcal A)$.\color{black}

\item[(A7)] This case is similar to (A3). W.l.o.g.\ one may assume that $A_e=1$ for the first edge $e$ along $t$  and that $\omega(e)$ is a central peripheral vertex if $e$ is peripheral.

 If $e$ is non-peripheral (and therefore free), add a new vertex $v$ with $A_v=\langle\nu_\mathcal{A}(t)\rangle\le P_m$, replace $e$ by a free edge $e'$ with $\alpha(e')=v$, $\omega(e')=\omega(e)$, $g_e'=g_e$, $A_{e'}=1$.

If $e$ is peripheral, add two new vertices $v,v'$ with $A_v=\langle\nu_\mathcal{A}(t)\rangle\le P_m
$ and $A_{v'}=1$, and replace $e$ by two edges $e',e''$ with $\alpha(e')=v$, $\omega(e')=\alpha(e'')=v'$, $\omega(e'')=\omega(e)$, $g_e'=1$, $g_{e''}=g_e$, $A_{e'}=A_{e''}=1$.

In both cases introduce a new peripheral star consisting of the single vertex $v$. Note that all vertex groups of essential vertices are generated by elements of length at most $M$. Conclude by performing the prenormalization process.

 Upshot: The relative rank decreases by one and the number of free factors increases by one. The resulting $(G,\PP)$-carrier graph is $M$-prenormal.\color{black}

\item[(A8)] Choose a geodesic word $upv$ with $p\in P_i$ for some $i$. Subdivide $e$ twice to obtain a segment consisting of the edges $e_1$, $e_2$, $e_3$ and $e_4$ with labels $u$, $p$, $1$ and $v$ respectively. Declare the subcarrier graph consisting of $e_2$ and $e_3$ to be an additional peripheral star with central vertex $\omega(e_2)=\alpha(e_3)$.

Upshot: The resulting carrier graph of groups is $M$-prenormal and in the same strong equivalence class. The total length of the free edges has decreased and $c(\mathcal A)=c(\mathcal A')$.\color{black}
\end{enumerate}

We will refer to the modifications made in cases A(1)-A(5) and A(7) that do not preserve the strong equivalence class to folds, otherwise to a weak AO-move, this happens in case A(1) if one of the peripheral stars involved is trivial. We also call the modification made in case A(6) an AO-move. We moreover call the modification made in case A(8) the introduction of a peripheral segment.

\smallskip
The following two lemmas follow immediately from the above discussion:

\begin{lemma}\label{lem:folds} If $\mathcal A'$ is obtained from $\mathcal A$ by a fold then one of the following holds:
\begin{enumerate}
\item The relative rank decreases.
\item The relative rank is preserved and the number of free factors decreases.
\item The conclusion of Lemma~\ref{resultprenormalization} holds and $\omega_e(A'_{e})$ is a strict overgroup of $\omega_{i(e)}(A_{i(e)})$ for some $e\in \mathcal E'$.
\end{enumerate}
\end{lemma}


\begin{lemma}\label{lem:reductionmoves}
\begin{enumerate}
\item If $\mathcal A'$ is obtained from $\mathcal A$ by a weak AO-move or the introduction of peripheral segments then $c(\mathcal A')\le c(\mathcal A)$.
\item If $\mathcal A'$ is obtained from $\mathcal A$ by an AO-move  then $c(\mathcal A')< c(\mathcal A)$.
\end{enumerate}
\end{lemma}


A finitely generated subgroup $U$ of some relatively hyperbolic group $G$ is called {\em locally relatively quasiconvex} if every finitely generated subgroup of $U$ is relatively quasiconvex in $G$. Clearly every subgroup has this property if $G$ is locally relatively quasiconvex.

\begin{proposition}\label{prop:main}
Let $G$ be a group, which is torsion-free, hyperbolic relative to $\PP=\{P_1,\ldots,P_n\}$ and has a finite symmetric generating set $X$. Let $\frak C$ be an equivalence class of $M$-prenormal $(G,\PP)$-carrier graphs.

There are $C=C(G,\PP,X,M,N(\A))$ and $D=D(G,\PP,X,M,N(\A))$ and finitely many equivalence classes $\frak C_1,\ldots ,\frak C_q$ of $D$-prenormal carrier graphs such that for any carrier graph  $(\mathcal A,((\mathcal C_i,c_i))_{1\le i\le k})\in\frak C$ representing a locally relatively quasiconvex subgroup there exists $\mathcal A'$ obtained from $\mathcal A$ by (weak) AO-moves and the introduction of peripheral segments such that one of the following holds:

\begin{enumerate}
 \item $\nu_{\mathcal{A'}}\colon(\pi_1(\A',v_0),d_{X\cup\mathcal{P}}^\mathcal{A'})\to(G,d_{X\cup\mathcal{P}})$ is a $(C,C)$-quasiisometric embedding and for any peripheral structure $\OO_\mathcal{A'}$ of $\mathcal{A'}$ is $\nu_{\mathcal{A'}}(\OO_\mathcal{A})$ an induced structure of $(G,\PP)$ on $\Ima\nu_\mathcal{A'}=\Ima\nu_\mathcal{A}$
\item There exists $\mathcal A''$ obtained from $\mathcal A'$ by a fold such that $\mathcal A''\in \frak C_i$ for some~$i$.
\end{enumerate}
\end{proposition}

\begin{proof} Let $(\mathcal A,((\mathcal C_i,c_i))_{1\le i\le k})\in\frak C$ such that $\mathcal A$ represents a locally relatively quasiconvex subgroup. Perform (weak) AO-moves and the introduction of peripheral segments as long as possible. This terminates for the following reasons: There can be only finitely many AO-moves as they decrease the complexity $c$ and the other two types do not increase $c$, see Lemma~\ref{lem:reductionmoves}. 

Moreover there is no infinite sequence of weak AO-moves and introductions of peripheral segments as any introduction of a peripheral segment decreases the sum of the $X\cup\PP$-length of the free edges and each weak AO-move reduces the number of peripheral stars while not increasing the sum of the $X\cup\PP$-length of the free edges. Let $\mathcal A'$ be the resulting carrier graph of groups. 

Now either $\nu_{\mathcal{A'}}\colon(\pi_1(\A',v_0),d_{X\cup\mathcal{P}}^\mathcal{A'})\to(G,d_{X\cup\mathcal{P}})$ is a quasi-isometric embedding or a fold is applicable. Note that by Corollary~\ref{cor:main} the constant for the quasi-isometric embedding only depends on $\frak C$ and not on $\mathcal A$ itself as $N(\mathbb A)$ can be bounded in terms of $\frak C$.

Thus we may assume that a fold is applicable resulting in a carrier graph $\mathcal A''$. Now one easily verifies that $\mathcal A''$ lies in one of finitely many equivalence classes of $(3L+M)$-prenormal equivalence classes only depending of $\frak C$ as $L$ only only depends on $\frak C$.
\end{proof}




\subsection{Finiteness Theorems}\label{finiteness_results}

In this section we establish some finiteness results for subgroups of relatively hyperbolic groups. 




\begin{theorem}\label{mainfinitenesstheorem} Let $G$ be a finitely generated torsion-free group that is  hyperbolic relative to a family $\PP=\{P_1,\ldots,P_n\}$ of noetherian groups. Let $n\in\mathbb N$. Then there exist finitely many equivalence classes $\frak C_1,\ldots ,\frak C_l$ of $(G,\PP)$-carrier graphs such that for any locally relatively quasiconvex, $n$-generated subgroup $U$ of $G$ there exists some $j\in\{1,\ldots ,l\}$ and $(\mathcal A,((\mathcal C_i,c_i))_{1\le i\le k})\in\frak C_j$ such that the following hold:
\begin{enumerate}
\item  $\nu_{\mathcal{A}}\colon(\pi_1(\A,v_0),d_{X\cup\mathcal{P}}^\mathcal{A})\to(G,d_{X\cup\mathcal{P}})$ is a quasiisometric embedding with $\hbox{Im}(\nu_{\mathcal A})=U$.
\item For any peripheral structure $\OO_\mathcal{A}$ of $\mathcal{A}$ is $\nu_{\mathcal{A}}(\OO_\mathcal{A})$ an induced structure of $(G,\PP)$ on $\Ima\nu_\mathcal{A}$.
\end{enumerate}
If the $P_i$ are finitely generated free Abelian groups then the claim holds for finitely many strong equivalence classes $\frak C_1,\ldots ,\frak C_l$.
\end{theorem}

\begin{proof}Let $\frak C_0$ be the equivalence class of $(G,\PP)$-carrier graphs $\mathcal A$ such that all vertex and edge groups are trivial and that $b(A)=n$. Clearly any $n$-generated subgroup of $G$ is represented by some element $\mathcal A_0\in\frak C_0$.

By Proposition~\ref{prop:main} there exists $\mathcal A'\in\frak C_0$ obtained from $\mathcal A_0$ such that $\nu_{\mathcal A'}$ is a quasi-isometric embedding with $\Ima\nu_{\mathcal A'}=U$ or one obtains $\mathcal A_1$ from one of finitely many equivalence classes such that $\mathcal A_1$ is obtained from $\mathcal A'$ by a fold and that $\Ima(\nu_{\mathcal A_1})=U$.

One now continues this argument for each of these finitely many equivalence classes. In this process case (1) and case (2) of Lemma~\ref{lem:folds} can clearly only occur finitely many times as new free factors can only emerge if the relative rank decreases. In the remaining case the edge group of at least one essential edge of the core increases (and some possibly disappear). As the $P_i$ are assumed to be noetherian, this can also only happen finitely many times. The first claim now follows from K\"onig's Lemma.

The last claim is an immediate consequence of Lemma~\ref{lemma:toral}(\ref{lemma:toral2}). \end{proof}

\begin{corollary} Let $G$ be a group, which is torsion-free, hyperbolic relative to a family $\PP=\{P_1,\ldots,P_n\}$ of noetherian groups. Let $n\in\mathbb N$.

Then there are only finitely many conjugacy classes of non-parabolic $n$-generated locally quasiconvex subgroups that do not split over a (possibly trivial) peripheral subgroup.
\end{corollary}

\begin{proof} Let $\frak C_1,\ldots ,\frak C_k$ be as in the conclusion of Theorem~\ref{mainfinitenesstheorem} and let $H$ be a non-parabolic $n$-generated locally quasiconvex subgroups that does not split over a peripheral subgroup. There exists $i\in\{1,\ldots ,k\}$ and $\mathcal A\in \frak C_i$ such that $\nu_{\mathcal A}$ maps $\pi_1(\mathbb A)$ isomorphically onto $H$.

It follows that $\pi_1(\mathbb A)$ does not split over a parabolic subgroup as otherwise $H$ would need to do the same. It follows that $\pi_1(\mathbb A)$ is carried by some vertex group. If it is carried by a vertex of some peripheral star then it maps to some parabolic subgroup which is excluded. Otherwise it is carried by some essential vertex. As the essential vertices are determined by the equivalence class  it follows that there are only finitely many possibilities.
\end{proof}

The following corollary applies in particular to limit groups by a result of Dahmani \cite{Dahmani2003}, it is an immediate consequence of Theorem~\ref{mainfinitenesstheorem} and Remark~\ref{stronimpliesisomorphic}.

\begin{corollary}\label{cor:locqua} Let $G$ be a torsion-free toral relatively hyperbolic group and $n\in\mathbb N$. Then there are only finitely many isomorphism classes of $n$-generated locally quasiconvex subgroups.
\end{corollary}

Corollar~\ref{cor:locqua} applies in particular if all subgroups of $G$ are relatively quasiconvex:

\begin{corollary}\label{cor:locqua2}[Theorem~\ref{thm:locqua}] Let $G$ be a torsion-free locally relatively quasiconves toral relatively hyperbolic group and $n\in\mathbb N$. Then there are only finitely many isomorphism classes of $n$-generated  subgroups.
\end{corollary}

We now interpret the above in the context of Kleinian groups. In this context much stronger finiteness results were proven by Biringer and Souto under the additional assumption of a lower bound  on the injectivity radius of the corresponding hyperbolic manifold~\cite{Biringer2017}.

\begin{corollary}\label{cor:kleinian} Let $G$ be a torsion-free Kleinian group and $n\in\mathbb N$.
\begin{enumerate}
\item If $G$ is of infinite covolume then there are only finitely many isomorphism classes of $n$-generated subgroups.
\item If $G$ is of finite covolume  then there are only finitely many isomorphism classes of $n$-generated subgroups of infinite index.
\end{enumerate}
\end{corollary}

\begin{proof} It is a consequence of the tameness theorem of Agol \cite{Agol2004} and Calegari-Gabai \cite{Calegari2006} and Thurston's geometrization theorem ~\cite{Thurston1982}, \cite{Kapovich2001a} that any torsion-free finitely generated Kleinian group $G$ is isomorphic to a geometrically finite Kleinian group and therefore to a toral relatively hyperbolic. Canary's \cite{Canary1996} covering theorem and the tameness theorem moverover imply that finitely generated subgroups are either geometrically finite and therefore relatively quasiconvex or virtual fibers, the second case only occurs if $G$ is of finite covolume.

(1) is now an immediate consequence of Corollary~\ref{cor:locqua} as the above discussion implies that  Kleinian groups of infinite covolume are locally relatively quasiconvex. 

(2) The above discussion moreover implies that subgroups of infinite index are virtual fibers or locally relatively quasiconvex. Because of Corollary~\ref{cor:locqua} it suffices to show that there are only finitely many virtual fibers of rank $n$. 

This follows from the theorem of Thurston that was already discussed in the introduction. He proved that for any surface there are only finitely many conjugacy classes of properly immersed surfaces without accidental parabolics, see Corollary~8.8.6 of \cite{Thurston} or see \cite{Dahmani2006} for a generalization in the contex of relatively hyperbolic groups. As virtual fibers are of this type and as there are only finitely many surfaces of given rank there are only finitely many virtual fibers of given rank.
\end{proof}

\section{Torsion}

The theory developed in this article is, similarly to that in \cite{Kapovich2004}, mainly geared towards torsion-free groups. This route was taken as the notions are quite technical as they are and as dealing with torsion would have further complicated matters. We conclude this article by commenting on the changes that are necessary to encompass torsion.

The problem is not torsion itself but torsion elements with infinite centralizers, the absence of such torsion elements was called almost torsion-free in \cite{Kapovich2004}. If torsion elements have infinite centralizers then the conclusion of various lemmas of Section~\ref{sec:Elements of Infinite Order in Relatively Hyperbolic Groups} and in particular Lemma~\ref{lem:conjugation_length} may fail.  To deal with this situation one needs to allow finite edge groups for non-essential non-peripheral edge groups, these edge groups will have infinite centralizers. The proof of finiteness results similar to those in Chapter~\ref{Chapter:Folds} would then need a combination of the arguments of the present paper and those of the folding proofs of Linnell accessibility \cite{Linnell1983}, see \cite{Dunwoody1998} and \cite{Weidmann2012}.



\addcontentsline{toc}{section}{References}
\bibliography{References}
\bibliographystyle{abbrv}

\end{document}